\documentclass[a4paper, 11pt]{article}

\usepackage{a4wide}
\usepackage[english]{babel}
\usepackage{amsmath}
\usepackage{psfrag}
\usepackage[latin1]{inputenc}
\usepackage[T1]{fontenc}
\usepackage{amsthm}
\usepackage{amssymb,amsxtra}
\usepackage[usenames]{color}
\usepackage{stmaryrd}
\usepackage{mathrsfs}
\usepackage{upgreek}
\usepackage{comment}
\usepackage{amsbsy}

\usepackage{geometry}

\usepackage{pstricks,pst-node,pst-plot}

\DeclareMathAlphabet{\mathpzc}{OT1}{pzc}{m}{it}
\newcommand{\NN}{\mathbb{N}}
\newcommand{\Z}{\mathbb{Z}}

\DeclareMathOperator{\res}{Res}
\DeclareMathOperator{\ind}{Ind}
\DeclareMathOperator{\Soc}{Soc}
\DeclareMathOperator{\Rad}{Rad}
\DeclareMathOperator{\Hd}{Hd}

\DeclareMathOperator{\Inf}{Inf}

\renewcommand{\leq}{\leqslant}
\renewcommand{\geq}{\geqslant}
\renewcommand{\unlhd}{\trianglelefteqslant}
\renewcommand{\unrhd}{\trianglerighteqslant}

\newcommand{\sym}[1]{\mathfrak{S}_{#1}}
\renewcommand{\P}{\mathscr{P}}
\newcommand{\JM}{\mathrm{JM}}
\newcommand{\RP}{\mathscr{RP}}
\newcommand{\m}{{\pmb{\mathscr{M}}}}

\newcommand{\sgn}{\operatorname{\mathrm{sgn}}}
\newcommand{\Rho}{\pmb{\mathscr{V}}}

\begin{document}

\swapnumbers
\setcounter{MaxMatrixCols}{20}

\theoremstyle{definition}

\newtheorem{defn}{Definition}[section]
\newtheorem{rem}[defn]{Remark}
\newtheorem{ques}[defn]{Question}
\newtheorem{eg}[defn]{Example}
\newtheorem{conj}[defn]{Conjecture}
\newtheorem{nota}[defn]{Notation}
\newtheorem{noth}[defn]{}

\theoremstyle{plain}

\newtheorem{prop}[defn]{Proposition}
\newtheorem{lem}[defn]{Lemma}
\newtheorem{cor}[defn]{Corollary}
\newtheorem{thm}[defn]{Theorem}

\baselineskip=14pt

\begin{center}
{\Large\bf Signed Young Modules and Simple Specht Modules}

\medskip

Susanne Danz and Kay Jin Lim

\today
\end{center}



\begin{abstract}
\noindent
By a result of Hemmer, every simple Specht module of a finite symmetric group over a field of odd characteristic is a signed Young module. While Specht modules are parametrized by partitions, indecomposable
signed Young modules are parametrized by certain pairs of partitions. The main result of this article establishes the signed Young module
labels of  simple Specht modules. Along the way we prove a number of results concerning indecomposable signed Young modules that are of independent interest.
In particular, we determine the label of the indecomposable signed Young module obtained
by tensoring a given indecomposable signed Young module with the sign representation.
As consequences, we obtain the Green vertices, Green correspondents, cohomological varieties, and complexities of all simple Specht modules and a class of simple modules of symmetric groups, and extend the results of Gill on periodic Young modules to periodic indecomposable signed Young modules.
\end{abstract}

\section{Introduction}\label{S:intro}

Young modules and Specht modules are of central importance for the representation theory of symmetric groups, and have
been studied intensively for more than a century.
Given a field $F$ and the symmetric group $\sym{n}$ of degree $n\geq 0$, both the isomorphism classes of  Young $F\sym{n}$-modules and Specht $F\sym{n}$-modules are well known
to be labelled by the partitions of $n$. 
If $F$ has characteristic 0 then the Specht $F\sym{n}$-modules are precisely the simple $F\sym{n}$-modules,
otherwise every simple $F\sym{n}$-module occurs as a quotient of a suitable Specht module.
In 2001 Donkin introduced the notion of indecomposable {\sl signed Young modules} over fields of odd characteristic, which generalize the classical Young modules (see \cite{SD}). Donkin proved that the isomorphism classes of indecomposable signed
Young $F\sym{n}$-modules are parametrized by certain pairs of partitions.

Suppose now that $F$ has odd characteristic $p$. A tight connection between the indecomposable signed Young modules and simple Specht modules has been established by the result of Hemmer \cite{DH}.
Hemmer showed that every simple Specht $F\sym{n}$-module is isomorphic to an indecomposable signed Young module. However, to our knowledge, it has so far been an open problem
to determine the signed Young module label of a given simple Specht module (see \cite[Problem~5.2]{DH}).
A conjecture concerning this labelling was first put up by the first author in  \cite[Vermutung~5.4.2]{SDanz} and later, independently, by the second author \cite[Conjecture 8.2]{KJL2} and Orlob \cite[Vermutung A.1.10]{JO}. In an unpublished note, the second author and Orlob verified the conjecture in the case when the simple Specht modules are labelled by hook partitions.
The main result  of the present article, Theorem~\ref{T: irred specht}, confirms this conjecture. 

\medskip

Our strategy towards the proof of Theorem~\ref{T: irred specht} follows the ideas used by Fayers in \cite{MF2}
and by Hemmer in \cite{DH} as follows. Given a simple Specht $F\sym{n}$-module $S^\lambda$, there is always some $m\geq n$
and some simple Specht $F\sym{m}$-module $S^\mu$ belonging to a {\sl Rouquier block} such that $S^\mu$ is isomorphic
to a direct summand of the induced module $\ind_{\sym{n}}^{\sym{m}}(S^\lambda)$ and $S^\lambda$ is isomorphic to a direct summand of the
restricted module $\res_{\sym{n}}^{\sym{m}}(S^\mu)$. The structure of a Rouquier block is generally much better understood
than that of arbitrary blocks of symmetric groups. We prove our main result, Theorem~\ref{T: irred specht}, in two steps: first we reduce the proof to the Rouquier block case in Proposition~\ref{P: reduction}, and then verify our result for this special case.


\medskip

Altogether our proof of the main result requires three key ingredients: firstly, we need a
`twisting formula' for indecomposable signed Young modules, that is, given an indecomposable signed Young
$F\sym{n}$-module $Y$, we determine the label of the signed Young $F\sym{n}$-module obtained by tensoring $Y$ with
the sign representation. This formula is given in Theorem~\ref{T:twists indecomposable signed Young modules}, and should be of independent interest.

Secondly, we shall exploit the theory of {\sl Young vertices} and a generalized version of Green correspondence, which we shall refer to as {\sl Young--Green correspondence}. This theory has been
introduced by Grabmeier \cite{JGrab} in order to generalize Green's classical vertex theory of finite groups.
Donkin has shown in \cite{SD} that the indecomposable signed Young $F\sym{n}$-modules are precisely the
indecomposable $F\sym{n}$-modules with linear Young sources; in particular, they have trivial Green sources.
Donkin has also described explicitly the Young vertices and Young--Green correspondents of the indecomposable signed Young $F\sym{n}$-modules.

Lastly, we shall make use of a recent generalization of the well-known 
Young Rule. Let $F$ be any field. This generalization is due to the second author and Tan \cite{KJLKMT}, and it describes the multiplicity of a Specht $F\sym{n}$-module appearing as factors of a particular filtration of a signed Young permutation $F\sym{n}$-module (see Theorem~\ref{T:Specht factors of sgn permutation}).
If $F$ has characteristic 0, this determines precisely the composition factors of a signed Young permutation module, counting
multiplicities. In order to apply this result in the proof of Theorem~\ref{T: irred specht} we shall have to compare
simple Specht modules and signed Young permutation modules over fields of characteristic 0 and over fields of characteristic $p\geq 3$.

\medskip

Some immediate consequences can be drawn from the twisting and labelling formulae
in Theorem~\ref{T:twists indecomposable signed Young modules} and Theorem~\ref{T: irred specht}, respectively.
In particular, we are able to describe the Green vertices, Green correspondents, cohomological varieties, and complexities of all simple Specht modules and a class of simple modules of symmetric groups over any (algebraically closed) field of odd characteristic.
In doing so, we recover a result of Wildon \cite{MW} on simple Specht modules labelled by hook partitions, and extend the results of Gill \cite{CG} on periodic Young modules to periodic indecomposable signed Young modules.

\medskip

The present paper is organized as follows. In Section~\ref{S: Prelim}, we set up the notation needed throughout, and we give a
summary of some well-known basic properties of symmetric groups and their modules.
In Section~\ref{sec signed Young}, we first review the theory of Young vertices and the Young--Green correspondence,
and then prove the twisting formula Theorem~\ref{T:twists indecomposable signed Young modules}. Along the way, we obtain the Green correspondents, with respect to any subgroup containing the normalizers of their respective Green vertices, and the Brauer constructions of indecomposable signed Young modules. In Section~\ref{sec simple Specht}, we analyze the set of partitions that label the simple Specht modules in odd characteristic $p$. The analysis is based on Fayers's characterization of these partitions \cite[Proposition 2.1]{MF2}.
We show that the $p$-cores of such partitions can be obtained by the independent procedures of stripping off horizontal and vertical $p$-hooks only.
Moreover, we
analyze the relation between two such partitions that are {\sl adjacent} in the sense of Fayers (see \ref{noth Rouquier}).
The latter analysis is particularly important for the proof of the reduction of Theorem~\ref{T: irred specht} to the Rouquier block case in Proposition~\ref{P: reduction}.
Section~\ref{sec labels} is devoted to the proof of the labelling formula of a simple Specht module as an indecomposable signed Young module.
In Section~\ref{S: conseq}, we apply our main results to derive the above-mentioned information concerning Green vertices, Green correspondents, cohomological varieties and complexities of all simple Specht modules and a class of simple modules of the symmetric groups. 
Furthermore, using Gill's result \cite{CG}, we describe the periods and, in the case of blocks of weight 1, we also give a minimal projective resolution of non-projective periodic indecomposable signed Young modules.

Since our notation is slightly different from Donkin's and there are some subtleties regarding sign representations of various permutation groups involved, in Appendix~\ref{appendix}, we follow closely the argument given by Donkin in \cite[\S5.2]{SD} to give a proof of his characterization of the Young--Green correspondents of indecomposable signed Young modules. In Appendix~\ref{S: signed Young Rule}, to make our paper self-contained, we present a signed version of Young's Rule following \cite{KJLKMT}.


\bigskip
\noindent
{\bf Acknowledgements:}\, The first author has been supported through a
Marie Curie Career Integration Grant (PCIG10-GA-2011-303774), and the
DFG Priority Programme `Representation Theory'
(grant \# DA1115/3-1). The second author is supported by the Singapore Ministry of Education AcRF Tier 1 grant RG13/14.

Both authors should like to thank John Murray for pointing out a mistake in an earlier version of this article. We would also like to thank the referee for the valuable suggestions and comments on an earlier version.

\section{Preliminaries}\label{S: Prelim}

We begin by fixing  the notation used throughout this article. We also
collect some well-known facts about the representation theory of symmetric groups that we shall
need repeatedly. For further details, we refer the reader to \cite{GJ1,GJAK,NT}.

\begin{nota}\label{nota general}
%
%
%
%
Suppose that $G$ is a finite group and that $R\neq\{0\}$ is a commutative ring. By an {\sl $RG$-module} we always understand a finitely generated
left $RG$-module that is $R$-free. If $M_1$ and $M_2$ are $RG$-modules such that $M_1$ is isomorphic to a direct summand of $M_2$ then we write $M_1\mid M_2$.

If $H\leq G$ and if $M$ and $N$ are an $RG$-module and $RH$-module, respectively, then we denote by $\res_H^G(M)$ the restriction of $M$ to $H$ and by
$\ind_H^G(N)$ the induction of $N$ to $G$.
If $H\unlhd G$ and if $N$ is an $R[G/H]$-module then
$\Inf_{G/H}^G(N)$ denotes the inflation of $N$ to $G$.

The Jacobson radical of an $RG$-module $M$ will be denoted by $\Rad(M)$. Moreover, the head and socle of $M$ are denoted by $\Hd(M)=M/\Rad(M)$ and $\Soc(M)$, respectively.

For subgroups $H$ and $K$ of $G$, we write $H\leq_GK$ if $H$ is $G$-conjugate to a subgroup of $K$, and we write $H=_GK$ if
$H$ is $G$-conjugate to $K$. For $g\in G$, we set ${}^gH:=gHg^{-1}$.

For two finite groups $G$ and $H$, an $RG$-module $M$ and an $RH$-module $N$, we denote by $M\boxtimes N:=M\otimes_R N$ the outer
tensor product of $M$ and $N$, which is naturally an $R[G\times H]$-module.

\end{nota}

\begin{noth}\label{noth abacus}{\bf Partitions, Young diagrams and abaci.}\;
Denote by $\NN$ and $\Z^+$ the sets of non-negative and positive integers, respectively. Let $n\in\NN$, and let $p$ be  a prime.

\smallskip

(a)\, A {\sl composition} of $n$ is a sequence of non-negative integers that sum to $n$. The unique composition of $0$ is denoted by $\varnothing$. Suppose that $\alpha=(\alpha_1,\ldots,\alpha_r)$ and $\beta=(\beta_1,\ldots,\beta_s)$ are compositions with $r\leq s$ and $k$ is positive integer. We set
\begin{align*}
  \alpha+\beta&:=(\alpha_1+\beta_1,\ldots,\alpha_r+\beta_r,\beta_{r+1},\ldots,\beta_s),\\
  k\alpha&:=(k\alpha_1,\ldots,k\alpha_r).
\end{align*} Also, $\beta+\alpha=\alpha+\beta$ and $\beta-\alpha$ has the obvious meaning if it is a composition. Similarly, if every part of $\alpha$ is a multiple of $k$ then we have the composition $\frac{1}{k}\alpha$ where $(\frac{1}{k}\alpha)_i=\frac{\alpha_i}{k}$ for $i\in\{1,\ldots,r\}$. A composition is called a {\sl partition} if its parts are weakly decreasing. The set of partitions of $n$ is denoted by $\P(n)$.

Let $\lambda=(\lambda_1,\ldots,\lambda_r)\in \P(n)$, where we shall mostly assume $\lambda_r>0$. In this case, we write $|\lambda|=n$ or $\lambda\vdash n$. 
The partition $\lambda$ is identified with its Young diagram \[[\lambda]=\{(i,j)\in\NN^2\ :\ 1\leq i\leq r,\ 1\leq j\leq \lambda_i\}.\] The elements of $[\lambda]$ are called the {\sl nodes} of $[\lambda]$. If $m\in\{0,1,\ldots,n\}$ and if $\mu$ is a partition of $m$ such that $\mu_i\leq \lambda_i$, for all $i\geq 1$, then $\mu$ is
called a {\sl subpartition of $\lambda$}. One then has $[\mu]\subseteq [\lambda]$, and one calls
$[\lambda]\smallsetminus [\mu]$ a {\sl skew diagram}.

The {\sl conjugate} of a partition $\lambda$ of $n$ is the partition whose Young diagram is obtained by transposing $[\lambda]$;
one denotes this partition by $\lambda'$.

\smallskip

(b)\, The
{\sl $p$-core} $\kappa_p(\lambda)$ of $\lambda$ is the partition obtained from $\lambda$ by successively removing all rim $p$-hooks from $[\lambda]$. The $p$-core of $\lambda$ is independent of the procedure of  rim $p$-hooks removal.
The {\sl $p$-weight} $\omega_p(\lambda)$ of $\lambda$ is the total number of such rim $p$-hooks removed to reach its $p$-core.
Thus $\omega_p(\lambda)=\frac{1}{p}(|\lambda|-|\kappa_p(\lambda)|)$.

The {\sl hook length} $h_\lambda(i,j)$ of a node $(i,j)$ of a diagram $[\lambda]$ is the total number of nodes $(x,j)$ and $(i,y)$ of $[\lambda]$ satisfying $x\geq i$ and $y\geq j$, respectively.

The {\sl $p$-residue diagram} of $\lambda$ is obtained by replacing each node $(i,j)$ of the Young diagram $[\lambda]$
by the residue of $j-i$ modulo $p$.

Given $\lambda$, for $i\in\{0,\ldots,p-1\}$, let $c_i$ be the number of nodes of $[\lambda]$ of $p$-residue $i$. Then $c_p(\lambda):=(c_0,\ldots,c_{p-1})$ is called the {\sl $p$-content of $\lambda$}. Partitions $\lambda,\mu\in\P(n)$ have
the same $p$-core if and only if they have the same $p$-content.

The partition $\lambda$ is called a {\sl $p$-singular} if $$\lambda_{i+1}=\lambda_{i+2}=\cdots=\lambda_{i+p}\geq 1\,,$$
for some $i\in\NN$. A partition that is not $p$-singular is called {\sl $p$-regular}.
If $\lambda$ is $p$-regular then the conjugate partition $\lambda'$ is called {\sl $p$-restricted}. Note that the $p$-core of a partition is both $p$-regular and $p$-restricted. We denote the set of all $p$-restricted partitions of $n$ by $\RP(n)$.

\medskip

(c)\, By \cite[Lemma~7.5]{JGrab}, every partition $\lambda$ of $n$ can be written as
$\lambda=\sum_{i=0}^{r_\lambda}p^i\cdot \lambda(i)$, for some uniquely determined $p$-restricted partitions $\lambda(0),\ldots,\lambda(r_\lambda)$ where $\lambda(r_\lambda)\neq \varnothing$ if $\lambda\neq\varnothing$. One calls this the {\sl $p$-adic expansion of $\lambda$}.

More explicitly, we obtain the $p$-adic expansion in the following inductive way. Suppose that $\lambda=(\lambda_1,\ldots,\lambda_r)$ with $\lambda_r>0$. One first removes all possible horizontal rim $p$-hooks from the $r$th row of $[\lambda]$, then all horizontal rim $p$-hooks from the $(r-1)$st row of the resulting diagram, and so on. This procedure leads to the remaining Young diagram of $\lambda(0)$. Furthermore, $\lambda-\lambda(0)$ is a partition such that every part is a multiple of $p$. Repeat the above procedure with the partition $\frac{1}{p}(\lambda-\lambda(0))$ to obtain the partition $\lambda(1)$. Continuing this process gives us the $p$-adic expansion of $\lambda$. The partitions $\lambda(0),\ldots,\lambda(r_\lambda)$ are clearly $p$-restricted and uniquely determined by both $\lambda$ and $p$.

\medskip

(d)\, An {\sl abacus (display)} consists of $p$ {\sl runners}, labelled from $0$ to $p-1$ from left to right, with positions $(i-1)p+j$ belonging to the $i$th row and runner $j$ such that each position of the abacus is either vacant or occupied by  a {\sl bead}. The position $k$ is said to be {\sl higher} (respectively, {\sl lower}) than the position $\ell$ if $k<\ell$ (respectively, $k>\ell$).
Every partition $\lambda$ can be associated to an abacus in the following way. Suppose that $\lambda=(\lambda_1,\ldots,\lambda_r)$ is a partition. Let $s\geq r$. A sequence of {\sl $\beta$-numbers} is a sequence of non-negative integers $(\beta_1,\ldots,\beta_s)$ such that $\beta_i=\lambda_i-i+s$, which is also called an {\sl $s$-element $\beta$-set}. The abacus associated to these $\beta$-numbers is the abacus whose position $(i-1)p+j$ is occupied if and only if, for some $1\leq k\leq s$, one has $\beta_k=(i-1)p+j$. Clearly, the abacus display of $\lambda$ depends on the choice of $s$.

Conversely, given an abacus with $p$ runners and a finite number of beads, we can read off the partition it represents in the following way. For each bead  on the abacus, we count the number of vacant positions that are higher than the position the bead occupies on the abacus. 
Suppose that the numbers obtained in this manner are $\lambda_1\geq \lambda_2\geq \cdots\geq \lambda_r$. Then the partition associated to this abacus is $(\lambda_1,\lambda_2,\ldots,\lambda_r)$ where $\lambda_r$ may be zero.

Moreover, we obtain the $p$-core of $\lambda$ as the partition obtained from an abacus of $\lambda$ by moving all beads on the abacus as high up as possible along every runner. The $p$-weight $\omega_p(\lambda)$ of $\lambda$ is precisely the total number of such bead movements.

\smallskip

Given an abacus of a partition $\lambda$ with $p$ runners, for each $i\in\{0,\ldots,p-1\}$, let $\lambda^{(i)}_j$ be the number of unoccupied positions higher than the $j$th lowest occupied position on runner $i$. Then we obtain the partition $\lambda^{(i)}=(\lambda^{(i)}_1,\lambda^{(i)}_2,\ldots)$ with respect to the given  abacus. The {\sl $p$-quotient} of $\lambda$ with respect to the abacus is the sequence
$$(\lambda^{(0)},\lambda^{(1)},\ldots,\lambda^{(p-1)})\,.$$
The $p$-quotient of $\lambda$ is unique up to cyclic place permutations. In particular, \[\omega_p(\lambda)=|\lambda^{(0)}|+|\lambda^{(1)}|+\cdots+|\lambda^{(p-1)}|.\] We demonstrate this by an example.
\end{noth}

\begin{eg} Suppose that we are given an abacus as follows. For each bead, we record the number of vacant positions that are higher than the position the bead occupies. We obtain the following configuration.
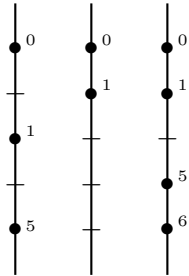
\begin{center}
  \begin{pspicture}(0,0)(2,4) \psset{yunit=.6cm}
    \rput(0,1){$\bullet$} \rput(1,1){$-$} \rput(2,1){$\bullet$} \psline(0,0)(0,6) \psline(1,0)(1,6) \psline(2,0)(2,6)
    \rput(0,2){$-$} \rput(1,2){$-$} \rput(1,3){$-$} \rput(2,3){$-$} \rput(0,4){$-$} \rput(2,2){$\bullet$} \rput(0,3){$\bullet$} \rput(1,4){$\bullet$} \rput(2,4){$\bullet$} \rput(0,5){$\bullet$} \rput(1,5){$\bullet$} \rput(2,5){$\bullet$}  \rput(2.2,2.2){{\tiny $5$}} \rput(0.2,3.2){{\tiny $1$}} \rput(1.2,4.2){{\tiny $1$}} \rput(2.2,4.2){{\tiny $1$}} \rput(0.2,5.2){{\tiny $0$}} \rput(1.2,5.2){{\tiny $0$}} \rput(2.2,5.2){{\tiny $0$}} \rput(0.2,1.1){{\tiny $5$}} \rput(2.2,1.2){{\tiny $6$}}
     \end{pspicture}
\end{center}
The partition this abacus represents is $\lambda=(6,5,5,1,1,1)$. As we move the beads on every runner as high up as possible, we obtain the following abacus.

 
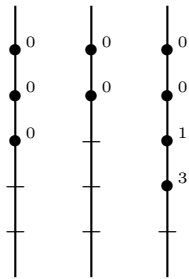
\begin{center}
  \begin{pspicture}(0,0)(2,4) \psset{yunit=.6cm}
    \rput(0,4){$\bullet$} \rput(1,1){$-$} \rput(2,2){$\bullet$} \psline(0,0)(0,6) \psline(1,0)(1,6) \psline(2,0)(2,6)
    \rput(0,2){$-$} \rput(1,2){$-$} \rput(1,3){$-$} \rput(2,1){$-$} \rput(0,1){$-$} \rput(2,3){$\bullet$} \rput(0,3){$\bullet$} \rput(1,4){$\bullet$} \rput(2,4){$\bullet$} \rput(0,5){$\bullet$} \rput(1,5){$\bullet$} \rput(2,5){$\bullet$} \rput(2.2,3.2){{\tiny $1$}} \rput(2.2,2.2){{\tiny $3$}} \rput(2.2,4.2){{\tiny $0$}} \rput(2.2,5.2){{\tiny $0$}} \rput(1.2,4.2){{\tiny $0$}} \rput(1.2,5.2){{\tiny $0$}} \rput(.2,4.2){{\tiny $0$}} \rput(.2,5.2){{\tiny $0$}} \rput(.2,3.2){{\tiny $0$}}
  \end{pspicture}
\end{center}
So $\kappa_p(\lambda)=(3,1)$. The number of bead movements carried out is $5$, and hence $\omega_3(\lambda)=5$. The $3$-quotient of $\lambda$ with respect to the abacus is \[(\lambda^{(0)},\lambda^{(1)},\lambda^{(2)})=((2,1),\varnothing,(1,1)).\]
\end{eg}

Next we recall various classical modules of the symmetric groups, and set up the notation used throughout.

\begin{noth}\label{noth perm Specht}{\bf Young permutation modules and Specht modules.}\,
Let $n\in\NN$, let $\sym{n}$ be the symmetric group of degree $n$, and let $R\neq \{0\}$ be a commutative ring.. By convention, $\sym{0}$ is the trivial group. Given a composition $\alpha=(\alpha_1,\ldots,\alpha_r)$ of $n$, we have the Young subgroup
$$\sym{\alpha}:=\sym{\alpha_1}\times\cdots\times\sym{\alpha_r}\,.$$
of $\sym{n}$.

\smallskip

(a)\,  For $\lambda\in\P(n)$, let $R$ be the trivial $R\sym{n}$-module. We have the {\sl Young permutation $R\sym{n}$-module}
$$M_R^\lambda:=\ind_{\sym{\lambda}}^{\sym{n}}(R)\,.$$
By a  classical result of James, see \cite[\S4]{GJ1}, the permutation module $M^\lambda_R$ contains the {\sl Specht $R\sym{n}$-module}
$S^\lambda_R$. In fact, the latter is defined in terms of a basis that does not depend on the ground ring $R$, and one has
$S^\lambda_R\cong R\otimes_{\mathbb{Z}}S^\lambda_{\mathbb{Z}}$. If $R$ is a field of characteristic 0 then the Specht modules
$S^\lambda_R$ yield a set of representatives of the isomorphism classes of simple $R\sym{n}$-modules, as $\lambda$ varies over the set
of partitions of $n$.

\smallskip

(b)\, An $R\sym{n}$-module $M$ is said to admit a {\sl Specht filtration} if there is a sequence of $R\sym{n}$-submodules $$\{0\}=M_0\subset M_1\subset\cdots\subset M_r\subset M_{r+1}=M$$ and partitions $\boldsymbol{\varrho}_1,\ldots,\boldsymbol{\varrho}_{r+1}$ of $n$ such that $M_i/M_{i-1}\cong S^{\boldsymbol{\varrho}_i}_R$, for $i\in\{1,\ldots,r+1\}$.
However, in general, $M$ may have several Specht filtrations, even if $R$ is a field.
Moreover, if $R$ is a field of characteristic $p\in\{2,3\}$ then the number of factors isomorphic to a given Specht $F\sym{n}$-module $S^\lambda_R$ may depend on the chosen filtration; this has been shown by Hemmer and Nakano in \cite{DHDN}.

\smallskip

(c)\, For the remainder of this paper, it will always be clear over which coefficient ring we are working, except in the proof of Theorem~\ref{T: irred specht}, where we shall have to deal with various coefficient rings. Thus, for ease of notation, we shall from now
just write $M^\lambda$ and $S^\lambda$ to denote the Young permutation $R\sym{n}$-module and Specht $R\sym{n}$-module, respectively, except in \ref{defi signed Young} and the remainder of Section \ref{sec labels} after \ref{rem modular system}.
\end{noth}

\begin{noth}\label{noth Specht}{\bf Simple modules and Young modules.}
%
Let again $n\in\NN$, and let $F$ be a field of characteristic $p\geq 0$.

(a)\, Let $F(n)$ be the trivial $F\sym{n}$-module and, for $H\leq \sym{n}$, we denote by $F(H)$ the trivial $FH$-module. If $H=\sym{\alpha}$ for some composition $\alpha$ of $n$ then we set $F(\alpha):=F(H)$. Let $\sgn(n)$ be the sign $F\sym{n}$-module. Similarly, for every subgroup $H\leq \sym{n}$, we denote by $\sgn(H)$ the sign module of $H$ over $F$, that is,
$\sgn(H)=\res_H^{\sym{n}}(\sgn(n))$. In the case when $H=\sym{\alpha}$, we set $\sgn(\alpha):=\sgn(\sym{\alpha})$.
Thus, if $\alpha=(\alpha_1,\ldots,\alpha_r)$, we have \[\sgn(\alpha)=\sgn(\alpha_1)\boxtimes\cdots\boxtimes \sgn(\alpha_r).\]

When no confusion concerning $H$ is possible, we simply write $F$ and $\sgn$ to denote the trivial $FH$-module and the sign $FH$-module, respectively. 

\smallskip

(b)\ For $\lambda\in\P(n)$, let $S_\lambda$ be the contragredient dual of $S^\lambda$. One has
\begin{equation}\label{eqn dual S}
S_\lambda=(S^\lambda)^*\cong S^{\lambda'}\otimes\sgn\,.
\end{equation}

Suppose that now $p>0$. Then, as $\lambda$ varies over the set of $p$-regular partitions of $n$, the module
$D^\lambda:=S^\lambda/\Rad(S^\lambda)$
varies over a set of representatives of the isomorphism classes of simple $F\sym{n}$-modules, and each of these modules
is self-dual.

\smallskip

Using the isomorphism (\ref{eqn dual S}), one can also parametrize the isomorphism classes of simple $F\sym{n}$-modules
by the $p$-restricted partitions of $n$; namely, we set $D_\mu:=\Soc(S^\mu)$, for each $\mu\in\RP(n)$. The relation between these two labellings of simple $F\sym{n}$-modules is  given by
\[D_{\lambda}\cong D^{\lambda'}\otimes \sgn\, ,\] for every $\lambda\in\RP(n)$.

\smallskip

For each $\lambda\in\RP(n)$, the $F\sym{n}$-module $D_\lambda\otimes \sgn$ is  of course simple again,
and we have $D_\lambda\otimes \sgn\cong D_{\m(\lambda)}$, where $\m:\RP(n)\to \RP(n)$ is the {\sl Mullineux map} on
the set of  $p$-restricted partitions of $n$. In this paper, we shall not need to know this map explicitly.

\smallskip

The combinatorial algorithm for computing the Mullineux map on the set of $p$-regular partitions was conjectured
by Mullineux in \cite{Mull}. The first proof of this conjecture is due to Ford an Kleshchev \cite{FoKl}; an alternative proof can be found in the work of Brundan and Kujawa \cite[\S6]{BrKu} in terms of $p$-restricted partitions.

\smallskip

Whenever $\lambda\in\P(n)$ is such that $\lambda=\kappa_p(\lambda)$, then
$\lambda$ is both $p$-regular and $p$-restricted. Furthermore, in this case, we have
$S^\lambda\cong D^\lambda\cong D_\lambda$ and $S^{\lambda'}\cong D^{\lambda'}\cong D_{\lambda'}$, so that
$D_\lambda\otimes \sgn\cong S^\lambda\otimes \sgn\cong S^{\lambda'}\cong D_{\lambda'}$, and
thus $\m(\lambda)=\lambda'$.

\smallskip

(c)\, Let $\lambda\in\P(n)$.
The indecomposable summands of the Young permutation $F\sym{n}$-modules are called Young modules.
In every decomposition of $M^\lambda$, there is a unique indecomposable direct summand that contains the Specht module $S^\lambda$; this module is unique up to isomorphism and is denoted by $Y^\lambda$. The remaining indecomposable direct summands of $M^\lambda$ are of the form $Y^\mu$ where $\mu\rhd\lambda$. By the results of Grabmeier~\cite[Satz 7.8]{JGrab} and Erdmann~\cite{KE}, the Young module $Y^\lambda$
is projective if and only if $\lambda\in\RP(n)$. In the latter case
$Y^\lambda$ is the projective cover of the simple module $D_\lambda$. As such, \[Y^\lambda\otimes \sgn\cong Y^{\m(\lambda)}\ .\]
%
%
\end{noth}

\begin{noth}\label{noth blocks Sn}{\bf Blocks of $F\sym{n}$.}
Let $F$ be a field of characteristic $p>0$.
By the Nakayama Conjecture, proved by Brauer and Robinson, the
blocks of the group algebra $F\sym{n}$ are labelled by the $p$-cores of the partitions of $n$.
Given $\lambda\in\P(n)$,
we shall denote by $b_{\kappa_p(\lambda)}$ the block of $F\sym{n}$ associated to $\kappa_p(\lambda)$. With this notation, each of the $F\sym{n}$-modules $Y^\lambda$, $S^\lambda$, $S_\lambda$, $D^\lambda$ and $D_\lambda$ belongs to the block $b_{\kappa_p(\lambda)}$,
for every admissible partition $\lambda$.

\smallskip

Since partitions $\lambda,\mu\in\P(n)$ have the same $p$-core if and only if they have the same $p$-content,
every block $b_{\kappa_p(\lambda)}$ of $F\sym{n}$ is also uniquely determined by the $p$-content $c=c_p(\lambda)$. In this case, we say that the block has $p$-content $c$.
\end{noth}

\begin{noth}\label{noth wreath}{\bf Wreath products and Sylow $p$-subgroups of $\sym{n}$.\,} Let $n\in\mathbb{Z}^+$.

(a)\, Let $G$ be any finite group, and let $H\leq \sym{n}$. We have the wreath product
$G\wr H:=\{(g_1,\ldots,g_n;\sigma): g_1,\ldots,g_n\in G,\, \sigma\in H\}$, whose multiplication is given by
$$(g_1,\ldots,g_n;\sigma)\cdot (h_1,\ldots,h_n;\pi)=(g_1h_{\sigma^{-1}(1)},\ldots,g_nh_{\sigma^{-1}(n)};\sigma\pi)\,,$$
for $g_1,\ldots,g_n,h_1,\ldots,h_n\in G$ and $\sigma,\pi\in H$.
We denote by $G^n$ the {\sl base group} $G^n:=\{(g_1,\ldots,g_n;1): g_1,\ldots,g_n\in G\}$ of $G\wr H$. If $U\leq H$ then
we set $U^\sharp:=\{(1,\ldots,1;\sigma): \sigma\in U\}\leq G\wr H$.

Whenever we have subgroups $H_1$ and $H_2$ of $H$ with $H_1\times H_2\leq H$, we may identify
the group $G\wr (H_1\times H_2)$ with the group $(G\wr H_1)\times (G\wr H_2)$.

Lastly, if $G\leq \sym{m}$, for some $m\geq 1$, then  we identify $G\wr H$ with a subgroup of the symmetric group $\sym{mn}$ in the
usual way: namely, $G^n$ is identified with a subgroup of the Young subgroup $\sym{(m^n)}$ and $H^\sharp$ is identified with the subgroup of $\sym{mn}$ that permutes $n$ successive blocks of size $m$ according to $H$.

\smallskip

(b)\, Let $p$ be a prime. We denote by $P_n$ a Sylow $p$-subgroup of $\sym{n}$. Recall from \cite[4.1.22, 4.1.24]{GJAK} that, if $n$ has $p$-adic
expansion $n=\sum_{i=0}^ra_ip^i$, for $a_0,\ldots,a_r\in\{0,\ldots,p-1\}$, we may choose
$P_n$ as a Sylow $p$-subgroup of the Young subgroup $\prod_{i=0}^r(\sym{p^i})^{a_i}$. Moreover, a Sylow
$p$-subgroup
$P_{p^i}$ of $\mathfrak{S}_{p^i}$  is isomorphic to the $i$-fold wreath product $C_p\wr\cdots \wr C_p$.
\end{noth}

The proof of the next lemma is straightforward, and is thus left to the reader.

\begin{lem}\label{lemma inflation}
Let $F$ be a field, let $G$ be a finite group, and let $N\leq H\leq G$ be such that $N\unlhd G$. Moreover, let $V$ be an $F[H/N]$-module. Then one has
$$\Inf_{G/N}^G(\ind_{H/N}^{G/N}(V))\cong \ind_{H}^G(\Inf_{H/N}^H(V))$$
as $FG$-modules.
\end{lem}

\begin{rem}\label{rem inflation}
Suppose that $G$ is a finite group. Let $N\unlhd G$, and let $K\leq H\leq \sym{n}$.
Then we have $N^n\leq G\wr K\leq G\wr H$, $N^n\unlhd G\wr H$, and we have natural group isomorphisms $(G\wr H)/N^n\cong (G/N)\wr H$
and $(G\wr K)/N^n\cong (G/N)\wr K$. If $V$ is an $F[(G/N)\wr K]$-module then Lemma~\ref{lemma inflation} applies, and
we get an isomorphism \[\Inf_{(G/N)\wr H}^{G\wr H}(\ind_{(G/N)\wr K}^{(G/N)\wr H}(V))\cong \ind_{G\wr K}^{G\wr H}(\Inf_{(G/N)\wr K}^{G\wr K}(V))\] of
$F[G\wr H]$-modules.
\end{rem}


\section{Signed Young Modules, Vertices, Green Correspondents, and Young--Green Correspondents}\label{sec signed Young}

Let $n\in\NN$. In this section, we collect the basic properties of indecomposable signed Young modules and prove our main result, Theorem~\ref{T:twists indecomposable signed Young modules}, concerning the tensoring of indecomposable signed Young modules by the sign representation.

\begin{defn}\label{defi signed Young}
Let $\P^2(n)$ be the set consisting of all pairs $(\lambda|\zeta)$ of partitions $\lambda,\zeta$ such that $|\lambda|+|\zeta|=n$. We allow that $\lambda$ or $\zeta$ be the empty partition $\varnothing$.
Let further $R\neq \{0\}$ be a commutative ring. For $(\lambda|\zeta)\in\P^2(n)$, one has the {\sl signed Young permutation $R\sym{n}$-module}
$$M_R(\lambda|\zeta):=\ind_{\mathfrak{S}_\lambda\times\mathfrak{S}_\zeta}^{\mathfrak{S}_n}(R\boxtimes \sgn)\,.$$
Here $R$ denotes the trivial $R\sym{\lambda}$-module, and $\sgn$ denotes the sign $R\sym{\zeta}$-module.
In the case when $\zeta=\varnothing$, one obtains the usual Young permutation module, that is, $M^\lambda_R=M_R(\lambda|\varnothing)$.
\end{defn}

For the remainder of this section, $F$ is  a field of characteristic $p\geq 3$. For convenience, suppose that $F$ is algebraically closed. Since we will only be considering $F\sym{n}$-modules for the remainder of this section, we drop the subscripts and write $M(\lambda|\zeta)=M_F(\lambda|\zeta)$.

\begin{noth}{\bf Signed Young permutation modules and indecomposable signed Young modules.}\label{noth signed young}\;
For $(\lambda|\zeta), (\alpha|\beta)\in\P^2(n)$, one says that $(\lambda|\zeta)$ {\sl dominates} $(\alpha|\beta)$,
and writes $(\lambda|\zeta)\unrhd (\alpha|\beta)$, if, for
all $k\geq 1$, one has

(a)\,  $\sum^k_{i=1}\lambda_i\geq \sum^k_{i=1}\alpha_i$, and

\smallskip

(b)\, $|\lambda|+\sum^k_{i=1}\zeta_i\geq |\alpha|+\sum^k_{i=1}\beta_i$.

\noindent This gives rise to a partial order $\unrhd$ on $\P^2(n)$, which is called the {\sl dominance order}. Sometimes we shall also write $(\alpha|\beta)\unlhd (\lambda|\zeta)$ instead of $(\lambda|\zeta)\unrhd (\alpha|\beta)$.

Suppose that $(\lambda|\zeta)\unrhd (\alpha|\beta)$. We have both $|\lambda|\geq |\alpha|$ and $|\zeta|\leq |\beta|$. Furthermore, in the case when $|\lambda|=|\alpha|$, one also has $|\beta|=|\zeta|$, and hence $\lambda\unrhd\alpha$ and $\zeta\unrhd\beta$, where here $\unrhd$ are the usual dominance orders on the sets of partitions of $|\alpha|$ and $|\beta|$, respectively.

\smallskip


Following Donkin \cite{SD}, one may define indecomposable signed Young modules inductively as follows. For any $(\lambda|p\mu)\in\P^2(n)$, the signed Young permutation $F\sym{n}$-module $M(\lambda|p\mu)$ has an indecomposable direct summand $Y(\lambda|p\mu)$ with Krull--Schmidt multiplicity 1 and the remaining indecomposable summands are
are already isomorphic to indecomposable direct summands of signed permutation modules
$M(\alpha|p\beta)$ such that $(\alpha|p\beta)\rhd(\lambda|p\mu)$. The indecomposable $F\sym{n}$-module $Y(\lambda|p\mu)$ is called the {\sl indecomposable signed Young module} labelled by $(\lambda|p\mu)\in\P^2(n)$. More generally, following \cite[2.3(7)]{SD}, for an arbitrary pair $(\alpha|\beta)\in \P^2(n)$, the signed Young permutation module $M(\alpha|\beta)$ is  isomorphic to a direct sum of some signed Young $F\sym{n}$-modules $Y(\lambda|p\mu)$ such that $(\lambda|p\mu)\unrhd (\alpha|\beta)$. By \cite[Corollary 5.2.9]{DHJKDN}, $Y(\lambda|p\mu)$ belongs to the block $b_{\kappa_p(\lambda)}$ of $F\sym{n}$.


%
 Notice that, whenever $\mu=\varnothing$, one recovers the usual Young module
$Y^\lambda\cong Y(\lambda|\varnothing)$.


%
\end{noth}

\begin{lem}\label{L:ind and res of permutation}
Let $n\in\NN$, and let $(\alpha|\beta)\in \P^2(n)$.

{\rm (a)}\, For every positive integer $m>n$, one has
\[\ind_{\sym{n}}^{\sym{m}}(M(\alpha|\beta))\cong M(\alpha\#(1^{m-n})|\beta)\,,\]
where $\alpha\#(1^{m-n})$ is the concatenation of the partitions $\alpha$ and $(1^{m-n})$.

\smallskip

{\rm (b)} For every non-negative integer $m<n$, the restriction $\res_{\sym{m}}^{\sym{n}}(M(\alpha|\beta))$ is isomorphic
to a direct sum of copies of signed Young permutation $F\sym{m}$-modules of the form
$M(\delta|\partial)$, where $(\delta|\partial)$ is the pair of
partitions obtained from the pairwise rearrangement of some pair of compositions obtained from $(\alpha|\beta)$ by removing $n-m$ nodes.
\end{lem}

\begin{proof}
As for part (a), we have
\begin{align*}
\ind_{\sym{n}}^{\sym{m}}(M(\alpha|\beta))&=\ind_{\sym{n}\times \sym{(1^{m-n})}}^{\sym{m}}(\ind_{\sym{\alpha}\times\sym{\beta}}^{\sym{n}}(F\boxtimes \sgn)\boxtimes F)\\
&\cong\ind_{\sym{\alpha}\times\sym{\beta}\times\sym{(1^{m-n})}}^{\sym{m}}(F\boxtimes \sgn\boxtimes F)\cong M(\alpha\#(1^{m-n})|\beta)\,.
\end{align*}

For part (b), we use the Mackey Formula to get $\res_{\sym{m}}^{\sym{n}}(M(\alpha|\beta))\cong \bigoplus_g N(g)$, where
\[ N(g)=\ind_{{}^g(\sym{\alpha}\times \sym{\beta})\cap \sym{m}}^{\sym{m}}\left (\res_{{}^g(\sym{\alpha}\times \sym{\beta})\cap \sym{m}}^{{}^g(\sym{\alpha}\times \sym{\beta})}({}^g(F\boxtimes \sgn))\right )\,,\] and where $g$ varies over a set of representatives of the double cosets $\sym{m}\backslash \sym{n}/(\sym{\alpha}\times \sym{\beta})$. Each direct summand $N(g)$ is induced from some subgroup of the form ${}^g(\sym{\alpha}\times \sym{\beta})\cap \sym{m}={}^g(\sym{\check{\delta}}\times \sym{\check{\partial}})$, for a pair of compositions $(\check{\delta}|\check{\partial})$ obtained from the pair $(\alpha|\beta)$ by removing some $n-m$ nodes. The subgroup
${}^g(\sym{\check{\delta}}\times \sym{\check{\partial}})$ is conjugate to the direct product of a pair of Young subgroups of $\sym{m}$. More precisely, after  rearranging  the parts of $\check{\delta}$ and $\check{\partial}$, respectively, we obtain a pair of partitions $(\delta|\partial)$ such that $M(\delta|\partial)$ is isomorphic to $N(g)$.
\end{proof}



\begin{rem}\label{rem sign twist}
Let $(\lambda|p\mu)\in\P^2(n)$. Since $Y(\lambda|p\mu)\mid M(\lambda| p\mu)$, the $F\mathfrak{S}_n$-module $Y(\lambda|p\mu)\otimes \sgn$ is an indecomposable
direct summand of the signed Young permutation module
$$M(\lambda|p\mu)\otimes\sgn\cong  \ind_{\mathfrak{S}_\lambda\times \mathfrak{S}_{p\mu}}^{\mathfrak{S}_n}((F\boxtimes \sgn)\otimes (\sgn\boxtimes \sgn))\cong M(p\mu|\lambda)\,.$$
Thus $Y(\lambda|p\mu)\otimes \sgn$ is again an indecomposable signed Young $F\mathfrak{S}_n$-module. In Theorem~\ref{T:twists indecomposable signed Young modules},
we shall determine the indecomposable  signed Young module label of $Y(\lambda|p\mu)\otimes \sgn$ by
exploiting the theory of Young vertices and the generalized Green correspondence in the sense
of Grabmeier \cite{JGrab} and Donkin \cite{SD}; we shall refer to the latter as {\sl Young--Green correspondence.}
We shall heavily use \cite[\S5.2]{SD}.
Some arguments there are quite subtle, since one has to be careful with various sign representations involved in the
definition of the Young--Green correspondents of indecomposable signed Young modules.
Therefore, for the reader's convenience and to make our arguments
as self-contained as possible, in Appendix~\ref{appendix}, we follow the arguments in \cite[\S5.2]{SD} to present the results using our set-up.
\end{rem}

\begin{noth}\label{noth vertex}{\bf Young vertices and Green vertices.}\, Next we recall some known facts from \cite{SD,JGrab} about the theory of vertices with respect to a Mackey system of a finite group $G$.

(a)\, Let $G$ be any finite group, and let $M$ be an $FG$-module. If $H\leq G$ is such that
$M\mid\ind_H^G(\res_H^G(M))$ then $M$ is called {\sl relatively $H$-projective}. In the case when
$M$ is an indecomposable $FG$-module and $Q$ is minimal such that $M$ is relatively $Q$-projective, $Q$ is called
a {\sl Green vertex} of $M$. By \cite{JG1}, the Green vertices of $M$ form a $G$-conjugacy class of $p$-subgroups of $G$. Moreover,
if $Q$ is a Green vertex of $M$ then there is an indecomposable $FQ$-module $S$, unique up isomorphism and $N_G(Q)$-conjugation, such that $M\mid \ind_Q^G(S)$. One calls $S$ a {\sl Green $Q$-source} of $M$.

In the case when $G=\sym{n}$ for some $n\in\NN$, a defect group of a block of $F\sym{n}$ of $p$-weight $w$ can be chosen as a Sylow $p$-subgroup of $\sym{pw}$. Thus every indecomposable $F\sym{n}$-module lying in a block of weight $w$ has a Green vertex that is a subgroup of $\sym{pw}$.
\smallskip

(b)\, There is a generalized vertex theory in terms of {\sl Mackey systems}.
A {\sl Mackey system} for $G$ is a set $\mathcal{M}$ of subgroups of $G$ that is closed under
$G$-conjugation and under taking intersections; moreover, one requires that $\mathcal{M}$ contains the trivial group $\{1\}$ and
that some element $H\in\mathcal{M}$ contains a Sylow $p$-subgroup of $G$.
For details, we refer the reader to \cite[\S 2]{JGrab} and \cite[\S1]{SD}. We briefly record the
results related to the symmetric groups that we shall need in this paper. Here we consider the Mackey system $\mathcal{Y}$
of Young subgroups of $\sym{n}$.

\smallskip

Suppose that $M$ is an indecomposable
$F\sym{n}$-module, and let $H\leq \sym{n}$ be a Young subgroup such that $M$ is relatively $H$-projective but is not
relatively $K$-projective for any proper Young subgroup $K$ of $H$.
 We call $H$ a {\sl Young vertex} of $M$. Young vertices of $M$ are uniquely determined up to $\sym{n}$-conjugation.
If $H$ is a  Young vertex of $M$ then there is some indecomposable $FH$-module $L$ such that $M\mid\ind_H^{\sym{n}}(L)$.
This module $L$ is called a {\sl Young $H$-source} of $M$, and is uniquely determined up to isomorphism and $N_{\sym{n}}(H)$-conjugation.
In the case when $L$ is one-dimensional, one calls $M$ a {\sl linear-source module}. By \cite[1.3(1)]{SD}, the indecomposable linear-source
$F\sym{n}$-modules with respect to $\mathcal{Y}$  are precisely the indecomposable  signed Young modules.

Whenever $U\leq \sym{n}$, one considers the set $\mathcal{Y}\downarrow U:=\{H\cap U: H\in\mathcal{Y}\}$
of subgroups of $U$. One then has an analogous notion of $(\mathcal{Y}\downarrow U)$-vertices and sources of indecomposable $FU$-modules.
We shall then speak of Young vertices and Young sources of indecomposable $FU$-modules as well. Note that
if $U$ is a Young subgroup of $\sym{n}$ then $\mathcal{Y}\downarrow U\subseteq \mathcal{Y}$.

%
\end{noth}

The following result is a special case of \cite[Lemma~3.9]{JGrab}, and relates Young vertices with
Green vertices.

\begin{prop}\label{prop young green}
Suppose that $U\leq \sym{n}$, and let $M$ be an indecomposable $FU$-module with Green vertex $Q$. Then
$$H:=\mathop{\bigcap_{K\in\mathcal{Y}\downarrow U}}_{Q\leq K} K$$
is a Young vertex of $M$. Moreover, one has $N_U(Q)\leq N_U(H)$.
\end{prop}

\begin{nota}\label{nota p-adic}
Let $(\lambda|p\mu)\in\P^2(n)$. Recall from \ref{noth abacus}(c) the $p$-adic expansions
\begin{equation}\label{eqn exp}
\lambda=\sum_{i=0}^{r_\lambda}p^i\cdot \lambda(i)\quad\text{ and }\quad \mu=\sum_{i=0}^{r_\mu}p^i\cdot \mu(i)
\end{equation}
of $\lambda$ and $\mu$, respectively.
In the case when $\lambda=\varnothing=\lambda(0)$, we set $r_\lambda:=0$, and in the case when $\mu=\varnothing=\mu(0)$,
we set $r_\mu:=0$. Set $r:=\max\{r_\lambda,r_\mu+1\}$. We define the following composition $\Rho(\lambda|p\mu)$ of $n$:
\begin{equation*}\label{eqn rho}
\Rho(\lambda|p\mu):=(1^{|\lambda(0)|},p^{|\lambda(1)|+|\mu(0)|},(p^2)^{|\lambda(2)|+|\mu(1)|},\ldots,(p^r)^{|\lambda(r)|+|\mu(r-1)|})\,,
\end{equation*} here, the first $|\lambda(0)|$ parts of $\Rho(\lambda|p\mu)$ are of size 1, followed by $|\lambda(1)|+|\mu(0)|$ parts of size $p$, and so on. The notation $\Rho$ may be viewed as a function from $\P^2(n)$ to the set of compositions of $n$ with $p$-power parts.
\end{nota}

The next theorem describes the Young vertices and sources as well as the Green vertices and sources of
indecomposable signed Young modules. Part (a) is due to Donkin, see \cite[\S5.1]{SD}.
The assertion of part (b) is certainly also well known, and a proof in the case when $\mu=\varnothing$ can, for instance,
be found in \cite{KE}. We include a short proof of  part (b) for convenience.

\begin{thm}\label{thm Young vertex} Let $(\lambda|p\mu)\in\P^2(n)$ and $\rho=\Rho(\lambda|p\mu)$.

\smallskip


{\rm (a)}\, The indecomposable signed Young module $Y(\lambda|p\mu)$ has  Young vertex $\sym{\rho}$ and linear Young source.

\smallskip


{\rm (b)}\, Any Sylow $p$-subgroup of $\sym{\rho}$ is a Green vertex of
the indecomposable signed Young module $Y(\lambda|p\mu)$.
Moreover, $Y(\lambda|p\mu)$ has a trivial Green source.
\end{thm}

\begin{proof}
Set $Y:=Y(\lambda|p\mu)$. As mentioned above, part (a) was proved in \cite[5.1(3)]{SD}.

\smallskip

To prove part (b), set $H:=\sym{\rho}$, and let $L$ be  a linear Young $H$-source of $Y$.
Let $R$ be a Sylow $p$-subgroup of $H$, and let $Q$ be a Green vertex of $Y$.
Since $Y$ is relatively $H$-projective, there is an indecomposable
$FH$-module $L'$ such that $L'\mid\res_H^{\sym{n}}(Y)$ and $Y\mid\ind_H^{\sym{n}}(L')$. By definition, $L'$ is also a
Young $H$-source of $Y$. Thus, by \cite[Lemma 2.10]{JGrab}, $L'\cong {}^gL$, for some $g\in N_{\sym{n}}(H)$. In particular, $L'$ has dimension 1 as well.
Therefore, $L'$ has Green vertex $R$ and trivial Green $R$-source $F$, so that $L'\mid\ind_R^H(F)$ and $F\mid\res_R^H(L')$.
This implies $Y\mid\ind_R^{\sym{n}}(F)$ and $F\mid\res_R^{\sym{n}}(Y)$.
So, by \cite[\S4, Lemma 3.4]{NT},  we get $Q\leq_{\sym{n}} R$ as well as $R\leq_{\sym{n}} Q$. Hence $Q=_{\sym{n}} R$ is a Green vertex of $Y$, and the trivial $FQ$-module $F$ is a Green $Q$-source of $Y$.
\end{proof}

\begin{noth}\label{noth Green}{\bf Young--Green correspondents and Green correspondents.}
Recall, for instance from \cite[\S4.4]{NT}, the notion of Green correspondence. There is an analogous notion of Green
correspondence with respect to a Mackey system of subgroups of a given finite group (see \cite[\S3]{JGrab}). In the case of the symmetric groups, we shall again consider the Mackey system $\mathcal{Y}$ of Young subgroups of $\sym{n}$, and
from now on refer to the resulting generalized Green correspondence as {\sl Young--Green correspondence}. A concise summary of
Grabmeier's theory of Young--Green correspondence can also be found in \cite[\S 1.1]{SD}. For convenience, we recall the
basic notion that we shall need in the case of the symmetric groups here.

Suppose that $M$ is an indecomposable $F\sym{n}$-module with Young vertex $H$, and let $N$ be a subgroup of $\sym{n}$ such that $N_{\sym{n}}(H)\leq N\leq \sym{n}$.
One defines the following sets of subgroups of $\sym{n}$:
\begin{align*}
\mathcal{X}&:=\{A\in\mathcal{Y}: A\leq H\cap {}^gH,\text{ for some } g\in \sym{n}\smallsetminus N\}\,,\\
\mathcal{A}&:=\{A\in\mathcal{Y}: A\leq H,\, {}^gA\notin \mathcal{X},\text{ for all }  g\in \sym{n}\}\,,\\
\mathcal{Z}&:=\{A\in\mathcal{Y}\downarrow N: A\leq N\cap {}^g H,\text{ for some } g\in\sym{n}\smallsetminus N\}\,.
\end{align*}
%
%
%
%
%
%
Note that the Young vertex $H$ of $M$ is contained in $\mathcal{A}$.
\end{noth}

With this notation, one has the following central theorem of Grabmeier \cite[Satz 3.7]{JGrab}:

\begin{thm}\label{thm Grab} Let $M$ be an indecomposable $F\sym{n}$-module with Young vertex $H$, and let $N$ be a subgroup of $\sym{n}$ such that $N_{\sym{n}}(H)\leq N\leq \sym{n}$.

\smallskip

{\rm (a)}\, There is an indecomposable $FN$-module $f(M)$, unique up to isomorphism, with the following properties:

\smallskip

\quad {\rm (i)}\, $M\mid\ind_N^{\sym{n}}(f(M))$,

\quad {\rm (ii)}\, $\res_N^{\sym{n}}(M)\cong f(M)\oplus M'$, where every indecomposable direct summand of $M'$ has a Young vertex
in $\mathcal{Z}$,

\quad {\rm (iii)}\, no Young vertex of $f(M)$ is contained in $\mathcal{Z}$.

\medskip

{\rm (b)}\, The modules $M$ and $f(M)$ have a common Young vertex contained in $\mathcal{A}$, and a common Young source.

\smallskip

{\rm (c)}\, If $M''$ is an indecomposable $FN$-module with Young vertex contained in $\mathcal{A}$ then the following are equivalent:

\smallskip

\quad {\rm (i)}\, $M''\cong f(M)$;

\quad {\rm (ii)}\, $M''\mid\res_N^{\sym{n}}(M)$;

\quad {\rm (iii)}\, $M\mid\ind_N^{\sym{n}}(M'')$.

\medskip

{\rm (d)}\, The function $f$ given by $M\mapsto f(M)$ is a bijection between the set of isomorphism classes of indecomposable $F\sym{n}$-modules
with Young vertex in $\mathcal{A}$ and the set of isomorphism classes of indecomposable $FN$-modules with Young vertex contained in
$\mathcal{A}$.
\end{thm}

\begin{rem}\label{rem Grab}
With the notation as in Theorem~\ref{thm Grab} we shall call the $FN$-module $f(M)$ the {\sl Young--Green correspondent}
of $M$ with respect to $N$.
Let $K\in\mathcal{A}$ be a common
Young vertex of $M$ and $f(M)$. Then $K=_{\sym{n}}H$, by \cite[Korollar 2.7]{JGrab}.
Let $g\in\sym{n}$ be such that $K={}^gH$. If $g\notin N$ then $K\leq H\cap{}^gH$, thus $K\in\mathcal{X}$, a contradiction.
Hence $K=_NH$, so that $H$ is also a Young vertex of $f(M)$.
%

Theorem~\ref{thm Grab} also shows that, in particular, $f(M)$ has a linear Young source if and only if $M$ has. Furthermore, $M$ and $f(M)$ have
a common Green vertex and a common Green source.
\end{rem}

\begin{lem}\label{lemma Green sign}
Let $M$ be an indecomposable $F\mathfrak{S}_n$-module, and let $H$ be  a Young vertex of $M$.
Then

\smallskip

{\rm (a)}\, the indecomposable $F\sym{n}$-module $M\otimes \sgn$ has Young vertex $H$;

\smallskip

{\rm (b)}\, if $N_{\sym{n}}(H)\leq N\leq \sym{n}$ and if $M'$ is the Young--Green correspondent of $M$ with respect to the subgroup $N$
then $M'\otimes\sgn(N)$ is the Young--Green correspondent of $M\otimes\sgn$ with respect to $N$.
\end{lem}

\begin{proof}
The $F\sym{n}$-module $M$ is relatively $K$-projective for a subgroup $K$ of $\sym{n}$ if and only if $M\otimes \sgn$ is.
Thus $M$ and $M\otimes\sgn$ have the same Young vertices. Analogously, $M'$ and $M'\otimes\sgn$ have the same
Young vertices. Since $M'$ is the Young--Green correspondent of $M$, $M$ and $M'$ have common Young vertex $H$, hence
so do $M\otimes \sgn$ and $M'\otimes \sgn$.
Since $M\mid\ind_N^{\sym{n}}(M')$, we also have $M\otimes \sgn\mid (\ind_N^{\sym{n}}(M'))\otimes\sgn\cong \ind_N^{\sym{n}}(M'\otimes \sgn)$. Thus, by \cite[Satz~3.7(iii)]{JGrab} (see also Theorem~\ref{thm Grab}(c)), $M'\otimes\sgn$ is the Young--Green correspondent of $M\otimes\sgn$ with respect to $N$.
\end{proof}

\begin{nota}\label{nota R}
Let $m,k\in \NN$ be such that $k\geq 2$. If $M$ is an $F\sym{k}$-module then the outer tensor product $M^{\otimes m}$ is naturally
a module for the base group of the wreath product $\sym{k}\wr\sym{m}$. Via tensor induction $M^{\otimes m}$ extends
to an $F[\sym{k}\wr\sym{m}]$-module. Thus, for $(g_1,\ldots,g_{m};\sigma)\in\sym{k}\wr\sym{m}$
and $x_1,\ldots,x_{m}\in M$, one then has
$$(g_1,\ldots,g_{m};\sigma)\cdot (x_1\otimes\cdots \otimes x_{m}):=g_1x_{\sigma^{-1}(1)}\otimes\cdots\otimes g_{m}x_{\sigma^{-1}(m)}\,.$$
By abuse of notation, we shall denote this module by $M^{\otimes m}$ as well.

Now consider the special case where $M=\sgn(k)$. Then the action of $\sym{k}\wr\sym{m}$ on the one-dimensional tensor-induced module
$\sgn(k)^{\otimes m}$ reads
\begin{align*}
(g_1,\ldots,g_{m};\sigma)\cdot (x_1\otimes\cdots \otimes x_{m})&=g_1x_{\sigma^{-1}(1)}\otimes\cdots\otimes g_{m}x_{\sigma^{-1}(m)}\\
&=\sgn(g_1)\cdots \sgn(g_{m})\cdot x_1\otimes\cdots\otimes x_{m}\,
\end{align*}
for $x_1,\ldots,x_{m}\in \sgn(k)$ and $(g_1,\ldots,g_{m};\sigma)\in\sym{k}\wr\sym{m}$.
%
Thus the group $\sym{m}^{\sharp}\leq \sym{k}\wr\sym{m}$
acts trivially on $\sgn(k)^{\otimes m}$.
On the other hand, $\sgn(k)^{\otimes m}$ has another extension, namely,
$$\sgn(k;m):=\sgn(k)^{\otimes m}\otimes \Inf_{\mathfrak{S}_{m}}^{\sym{k}\wr\sym{m}}(\sgn(m))\,.$$
So the action of $\sym{k}\wr\sym{m}$ on $\sgn(k;m)$ is given by
$$(g_1,\ldots,g_{m};\sigma)* (x_1\otimes\cdots \otimes x_{m})=\sgn(\sigma) \sgn(g_1)\cdots \sgn(g_{m})\cdot x_1\otimes\cdots\otimes x_{m}\,,$$
for $x_1,\ldots,x_{m}\in \sgn(k)$ and $(g_1,\ldots,g_{m};\sigma)\in\sym{k}\wr\sym{m}$.

\smallskip

Let $m_1,m_2\in \NN$ be such that $m=m_1+m_2$. Let further $\alpha\in\RP(m_1)$ and $\beta\in\RP(m_2)$.
Following \cite[\S5]{SD}, we construct an $F[\mathfrak{S}_k\wr\mathfrak{S}_m]$-module $R_k(\alpha|\beta)$ as follows.
Recall from \ref{noth Specht}(c) that the projective cover of the simple $F\sym{m_1}$-module $D_\alpha$ is the Young $F\sym{m_1}$-module $Y^\alpha$, and, similarly, $Y^\beta$ is the projective cover of the simple $F\sym{m_2}$-module $D_\beta$.
Set \[Y^\alpha_k:=\Inf_{\sym{m_1}}^{\sym{k}\wr\sym{m_1}}(Y^\alpha),\] and, similarly, we set $Y^\beta_k:=\Inf_{\sym{m_2}}^{\sym{k}\wr\sym{m_2}}(Y^\beta)$.
With the above notation, we obtain the following
$F[\sym{k}\wr\sym{m}]$-module
\begin{equation}\label{eqn Rk}
R_k(\alpha|\beta):=\ind_{(\sym{k}\wr\sym{m_1})\times(\sym{k}\wr\sym{m_2})}^{\sym{k}\wr\sym{m}}(Y^\alpha_k\boxtimes (Y^\beta_k\otimes \sgn(k;m_2)))\,.
\end{equation}
It is understood that, whenever $m_1=0$ and $m_2>0$ (respectively, $m_1>0$ and $m_2=0$), we have $R_k(\varnothing|\beta)=Y^\beta_k\otimes \sgn(k;m_2)$ (respectively, $R_k(\alpha|\varnothing)=Y^\alpha_k$). Furthermore, $R_k(\varnothing|\varnothing)$ is the trivial $F\sym{0}$-module if $m=0$.
\end{nota}

\begin{rem}\label{rem k odd}
 Retain the notation as in \ref{nota R} above.

 \smallskip

(a)\, Note that, in the case when $k$ is odd, we have
$\sgn(k;m_2)\cong \sgn(\sym{k}\wr\sym{m_2})$.
But if $k$ is even then $\sgn(k)^{\otimes m_2}\cong \sgn(\sym{k}\wr\sym{m_2})$.

\smallskip

(b)\,Let $\mathfrak{A}_k$ be the alternating subgroup of degree $k$ of $\sym{k}$. The wreath product $\sym{k}\wr\sym{m_2}$ contains the normal subgroup
$\mathfrak{A}_k^{m_2}$, which acts trivially on both $\sgn(k)^{\otimes m_2}$ and $\sgn(k;m_2)$. Thus we may
view both modules as modules of the quotient group
$$(\sym{k}\wr\sym{m_2})/\mathfrak{A}_k^{m_2}\cong (\sym{k}/\mathfrak{A}_k)\wr\sym{m_2}\cong \sym{2}\wr\sym{m_2}\,.$$ Via the latter group isomorphism, we also have
\begin{align*}
\sgn(k)^{\otimes m_2}&\cong \Inf_{\sym{2}\wr\sym{m_2}}^{\sym{k}\wr\sym{m_2}}(\sgn(2)^{\otimes m_2}), \\
\sgn(k;m_2)&\cong \Inf_{\sym{2}\wr\sym{m_2}}^{\sym{k}\wr\sym{m_2}}(\sgn(2;m_2))\,.
\end{align*}
However, $\sgn(\sym{2}\wr\sym{m_2})\cong \sgn(2)^{\otimes m_2}\not\cong \sgn(2;m_2)$.

\smallskip

(c)\, By part (b) and Lemma~\ref{lemma inflation}, we have $R_k(\alpha|\beta)\cong \Inf_{\sym{2}\wr\sym{m}}^{\sym{k}\wr\sym{m}}(R_2(\alpha|\beta))$
via the epimorphism $\sym{k}\wr\sym{m}\to (\sym{k}\wr\sym{m})/\mathfrak{A}_k^m\cong \sym{2}\wr\sym{m}$.
\end{rem}

\begin{lem}\label{lemma R sign}
Let $k\geq 3$ be an odd integer and $\alpha,\beta$ be $p$-restricted partitions. We have
$$R_k(\alpha|\beta)\otimes\sgn\cong R_k(\beta|\alpha)\,.$$
\end{lem}

\begin{proof}
By Remark~\ref{rem k odd}(a) and the Frobenius Formula, we have
\begin{align*}
R_k(\alpha|\beta)\otimes\sgn=&\left (\ind_{(\sym{k}\wr\sym{m_1})\times(\sym{k}\wr\sym{m_2})}^{\sym{k}\wr\sym{m}}(Y^\alpha_k\boxtimes (Y^\beta_k\otimes \sgn))\right )\otimes\sgn\\
\cong&\ind_{(\sym{k}\wr\sym{m_1})\times(\sym{k}\wr\sym{m_2})}^{\sym{k}\wr\sym{m}}((Y^\alpha_k\boxtimes (Y^\beta_k\otimes \sgn)\otimes\sgn))\\
\cong&\ind_{(\sym{k}\wr\sym{m_1})\times(\sym{k}\wr\sym{m_2})}^{\sym{k}\wr\sym{m}}((Y^\alpha_k\otimes \sgn)\boxtimes  Y^\beta_k)\cong R_k(\beta|\alpha)\,.
\end{align*}
\end{proof}

\begin{nota}\label{nota Nrho}
Let $(\lambda|p\mu)\in\P^2(n)$, and let $\lambda=\sum_{i=0}^{r_\lambda}p^i\cdot\lambda(i)$ and $\mu=\sum_{i=0}^{r_\mu}p^i\cdot\mu(i)$  be the $p$-adic expansions of $\lambda$ and $\mu$, respectively. Let $r=\max\{r_\lambda,r_\mu+1\}$ and $\rho=\Rho(\lambda|p\mu)$. For $i\in\{0,\ldots,r\}$, set $n_i:=|\lambda(i)|+|\mu(i-1)|$; by convention, $\mu(-1)=\varnothing$.
By Theorem~\ref{thm Young vertex}(a), the Young subgroup
$$\sym{\rho}=(\sym{1})^{n_0}\times (\sym{p})^{n_1}\times\cdots \times (\sym{p^r})^{n_r}$$
of $\sym{n}$ is a Young vertex of $Y(\lambda|p\mu)$; here the first non-trivial direct factor $\sym{p}$ is supposed to act on the set $\{n_0+1,\ldots,n_0+p\}$, and so on.

In the sequel, we shall
denote the normalizer $N_{\sym{n}}(\sym{\rho})$ by $N(\rho)$ and identify $N(\rho)$ with the direct
product
$$\sym{n_0}\times (\sym{p}\wr\sym{n_1})\times\cdots \times (\sym{p^r}\wr\sym{n_r})\,.$$
From now on we fix a Sylow $p$-subgroup $P_\rho:=(P_{1})^{n_0}\times (P_p)^{n_1}\times\cdots\times (P_{p^r})^{n_r}$ of $\sym{\rho}$ where $P_{p^i}$ is a Sylow $p$-subgroup of $\sym{p^i}$ as in \ref{noth wreath}(b).
Now observe that
\begin{align*}
N_{\sym{n}}(P_\rho)&=\sym{n_0}\times N_{\sym{n_1p}}((P_p)^{n_1})\times\cdots \times N_{\sym{n_rp^r}}((P_{p^r})^{n_r})\\
&=\sym{n_0}\times (N_{\sym{p}}(P_p)\wr \sym{n_1})\times\cdots\times (N_{\sym{p^r}}(P_{p^r})\wr\sym{n_r})\leq N(\rho)\,.
\end{align*}

For each $i\in\{1,\ldots,r\}$, $R_{p^i}(\lambda(i)|\mu(i-1))$ is a module for $F[\sym{p^i}\wr\sym{n_i}]$. We define the $FN(\rho)$-module
$$\mathbf{R}(\lambda|p\mu):=Y^{\lambda(0)}\boxtimes R_p(\lambda(1)|\mu(0))\boxtimes\cdots\boxtimes R_{p^r}(\lambda(r)|\mu(r-1))\,.$$
\end{nota}

With the above notation, we can now state Donkin's result on Young--Green correspondents of indecomposable signed Young modules (see also Appendix \ref{appendix}).

\begin{thm}[{\cite[5.2 (2)]{SD}}]\label{thm Young Green}
Let $(\lambda|p\mu)\in\P^2(n)$ and $\rho=\Rho(\lambda|p\mu)$. The $FN(\rho)$-module $\mathbf{R}(\lambda|p\mu)$ is the Young--Green correspondent of $Y(\lambda|p\mu)$ with respect to $N(\rho)$.
\end{thm}

As we have seen in Theorem~\ref{thm Young vertex}(b), any Sylow $p$-subgroup
of the Young subgroup $\sym{\rho}$ of $\sym{n}$ is a Green vertex of the indecomposable signed Young module $Y(\lambda|p\mu)$.
As an easy consequence of Theorem~\ref{thm Young Green}, we get the following corollary.

\begin{cor}\label{cor Green}
Let $(\lambda|p\mu)\in \P^2(n)$, let $\rho=\Rho(\lambda|p\mu)$, and let $P_{\rho}$ be the fixed Sylow $p$-subgroup of $\sym{\rho}$. Suppose that $N_{\sym{n}}(P_\rho)\leq H\leq N(\rho)$. Then $\res_{H}^{N(\rho)}(\mathbf{R}(\lambda|p\mu))$ is the Green correspondent
of $Y(\lambda|p\mu)$ with respect to the subgroup $H$.
\end{cor}

\begin{proof} It suffices to show that $\mathbf{R}(\lambda|p\mu)$ restricts indecomposably to $H$ and that the indecomposable restriction has
Green vertex $P_\rho$. First we shall prove the corollary in the case when $H=N_{\sym{n}}(P_\rho)$. For this it further suffices to show that, for $i\in\{1,\ldots,r\}$, the $F[\sym{p^i}\wr\sym{n_i}]$-module
$R_{p^i}(\lambda(i)|\mu(i-1))$ restricts indecomposably to $N_{\sym{p^i}}(P_{p^i})\wr\sym{n_i}$
and that the restriction has Green vertex $(P_{p^i})^{n_i}$. Fix $i\in\{1,\ldots,r\}$ for
the remainder of this proof, and set $m_1:=|\lambda(i)|$, $m_2:=|\mu(i-1)|$ (so that $n_i=m_1+m_2$) and $N:=N_{\sym{p^i}}(P_{p^i})$.
Recall from (\ref{eqn Rk}) that
\begin{align*}
R_{p^i}(\lambda(i)|\mu(i-1))&=\ind_{(\sym{p^i}\wr\sym{m_1})\times(\sym{p^i}\wr\sym{m_2})}^{\sym{p^i}\wr\sym{n_i}}(Y_{p^i}^{\lambda(i)} \boxtimes (Y_{p^i}^{\mu(i-1)}\otimes \sgn(p^i;m_2)))\\
&\cong \ind_{(\sym{p^i}\wr\sym{m_1})\times(\sym{p^i}\wr\sym{m_2})}^{\sym{p^i}\wr\sym{n_i}}(Y_{p^i}^{\lambda(i)}\boxtimes (Y_{p^i}^{\m(\mu(i-1))}\otimes \sgn(p^i)^{\otimes m_2}))\,,
\end{align*}
here we have used the fact that $\sgn(p^i; m_2)\cong \sgn(p^i)^{\otimes m_2} \otimes \Inf_{\sym{m_2}}^{\sym{p^i}\wr\sym{m_2}}(\sgn(m_2))$ as $F[\sym{p^i}\wr\sym{m_2}]$-modules and that $Y^{\mu(i-1)}\otimes\sgn\cong Y^{\m(\mu(i-1))}$
as $F\sym{m_2}$-modules, where $\m$ is the Mullineux map defined in \ref{noth Specht}(b).
To determine $\res_{N\wr\sym{n_i}}^{\sym{p^i}\wr\sym{n_i}}(R_{p^i}(\lambda(i)|\mu(i-1)))$, we first apply
the Mackey Formula. Every element $x\in \sym{p^i}\wr \sym{n_i}$ can be written as a product $x=yz$, for uniquely
determined elements $y\in \sym{n_i}^{\sharp}\leq N\wr \sym{n_i}$ and $z\in \sym{p^i}^{n_i}\leq (\sym{p^i}\wr\sym{m_1})\times (\sym{p^i}\wr\sym{m_2})$. Therefore, there is precisely one double coset in
$(N\wr \sym{n_i})\backslash \sym{p^i}\wr\sym{n_i}/ ((\sym{p^i}\wr\sym{m_1})\times (\sym{p^i}\wr\sym{m_2}))$,
and the Mackey Formula gives
\begin{align*}
&\ \res_{N\wr\sym{n_i}}^{\sym{p^i}\wr\sym{n_i}}(R_{p^i}(\lambda(i)|\mu(i-1)))\\
\cong&\
\ind_{(N\wr\sym{m_1})\times (N\wr\sym{m_2})}^{N\wr\sym{n_i}}(\Inf^{N\wr\sym{m_1}}_{\sym{m_1}}(Y^{\lambda(i)})\boxtimes (\Inf^{N\wr\sym{m_2}}_{\sym{m_2}}(Y^{\m(\mu(i-1))})\otimes \sgn(N)^{\otimes m_2}))=:X\,.
\end{align*}
Notice that $Y^{\lambda(i)}\boxtimes Y^{\m(\mu(i-1))}$ is an indecomposable
projective $F[\sym{m_1}\times\sym{m_2}]$-module and that both $\sgn(N)$ and the trivial $FN$-module
have Green vertex $P_{p^i}$.
Moreover,
\begin{align*}
&\Inf^{N\wr\sym{m_1}}_{\sym{m_1}}(Y^{\lambda(i)})\boxtimes (\Inf^{N\wr\sym{m_2}}_{\sym{m_2}}(Y^{\m(\mu(i-1))})\otimes \sgn^{\otimes m_2})\\
&\cong \Inf_{\sym{m_1}\times\sym{m_2}}^{(N\wr\sym{m_1})\times (N\wr\sym{m_2})}(Y^{\lambda(i)}\boxtimes Y^{\m(\mu(i-1))})\otimes (F^{\otimes m_1}\boxtimes \sgn^{\otimes m_2})\,.
\end{align*}
Thus
\cite[Proposition~5.1]{BK}
implies that $X$ is an indecomposable $F[N\wr\sym{n_i}]$-module with Green vertex $(P_{p^i})^{m_1+m_2}=(P_{p^i})^{n_i}$ as claimed.

Suppose now that $H$ is any subgroup satisfying $N_{\sym{n}}(P_\rho)\leq H\leq N(\rho)$.
Since $\mathbf{R}(\lambda|p\mu)$ restricts indecomposably to $N_{\sym{n}}(P_\rho)$, it also restricts indecomposably to $H$.
Let $Q$ be a Green vertex of $\res_H^{N(\rho)}(\mathbf{R}(\lambda|p\mu))$. Then $P_\rho\leq_H Q$, by \cite[\S4, Lemma 3.4]{NT}.
On the other hand, $\res_H^{N(\rho)}(\mathbf{R}(\lambda|p\mu))\mid \res_H^{\sym{n}}(Y(\lambda| p\mu))$, and $Y(\lambda|p\mu)$ has
also Green vertex $P_\rho$, by Theorem~\ref{thm Young vertex}(b). Thus $Q\leq_{\sym{n}} P_\rho$, by \cite[\S4, Lemma 3.4]{NT} again. This implies
$Q=_HP_\rho$, so that $P_\rho$ is indeed a Green vertex of $\res_H^{N(\rho)}(\mathbf{R}(\lambda|p\mu))$.
Hence $\res_H^{N(\rho)}(\mathbf{R}(\lambda|p\mu))$ is the Green correspondent of $Y(\lambda|p\mu)$ with respect to $H$.
%
\end{proof}

\begin{noth}\label{noth Broue}{\bf Brauer constructions of indecomposable signed Young modules.} We refer the reader to \cite{MB} for the details on {\sl Brauer constructions} of {\sl $p$-permutation modules}. For any finite group $G$, a $p$-permutation $FG$-module is a direct sum of some indecomposable $FG$-modules with trivial Green sources.
Given a $p$-subgroup $P$ of a finite group $G$, via Brauer construction one gets a one-to-one correspondence between the isomorphism classes of indecomposable trivial-Green-source $FG$-modules with Green vertex $P$ and the isomorphism classes of
indecomposable projective $F[N_G(P)/P]$-modules. More precisely,  an indecomposable $FG$-module $M$
with Green vertex $P$ and trivial Green $P$-source is sent to its {\sl Brauer quotient} $M(P)$.
The latter carries the structure of an $FN_G(P)$-module on which $P$ acts trivially. Moreover, $M(P)$ is isomorphic to the Green correspondent of $M$ with respect to the subgroup $N_G(P)$.

By Theorem~\ref{thm Young vertex}(b),  indecomposable signed Young modules are $p$-permutation modules, since they have trivial Green sources. Thus Corollary~\ref{cor Green} gives us the following:
\end{noth}

\begin{cor}\label{C: cor Broue} 
Let $(\lambda|p\mu)\in \P^2(n)$, let $\rho=\Rho(\lambda|p\mu)$, and let $P_{\rho}$ be the fixed Sylow $p$-subgroup of $\sym{\rho}$. Then $P_\rho$ acts trivially on the Green correspondent $\res_{N_{\sym{n}(P_{\rho})}}^{N(\rho)}(\mathbf{R}(\lambda|p\mu))$ of $Y(\lambda|p\mu)$ with respect to $N_{\sym{n}}(P_{\rho})$. Moreover, viewed as $F[N_{\sym{n}}(P_{\rho})/P_\rho]$-module, the restriction
$\res_{N_{\sym{n}}(P_{\rho})}^{N(\rho)}(\mathbf{R}(\lambda|p\mu))$ is isomorphic to the Brauer construction of $Y(\lambda|p\mu)$ with respect to $P_\rho$.
\end{cor}

We can now state our main result of this section. Recall the Mullineux map $\m$ on $p$-restricted partitions defined in \ref{noth Specht}(b).

\begin{thm}\label{T:twists indecomposable signed Young modules}
Let $(\lambda|p\mu)\in\P^2(n)$.
Then one has an isomorphism  of $F\sym{n}$-modules
 $$Y(\lambda|p\mu)\otimes \sgn\cong Y(\m(\lambda(0))+p\mu|\lambda-\lambda(0))\ .$$

\end{thm}

\begin{proof} 
Let $\rho=\Rho(\lambda|p\mu)=(1^{n_0},p^{n_1},(p^2)^{n_2},\ldots,(p^r)^{n_r})\,$, that is, $n_i=|\lambda(i)|+|\mu(i-1)|$ for $i\in\{0,\ldots,r\}$ where $r=\max\{r_\lambda,r_\mu+1\}$. By Theorem~\ref{thm Young vertex}(a) and Lemma~\ref{lemma Green sign}(a), both $Y(\lambda|p\mu)$ and $Y(\lambda|p\mu)\otimes\sgn$ have Young vertex $\sym{\rho}$. Set $\alpha:=\m(\lambda(0))+p\mu$ and $p\beta:=\lambda-\lambda(0)$. Next we show that $\Rho(\alpha|p\beta)=\rho$, so that $Y(\alpha|p\beta)$ also has Young vertex $\sym{\rho}$.

Since both $\lambda(0)$ and $\m(\lambda(0))$ are $p$-restricted, $\alpha$ has
$p$-adic expansion $\alpha=\m(\lambda(0))+p\mu=p^0\cdot \m(\lambda(0))+\sum_{i=1}^rp^i\cdot \mu(i-1)$.
Moreover, $\beta$ has $p$-adic expansion $\beta=\sum_{i=1}^rp^{i-1}\cdot \lambda(i)$.
Hence $|\alpha(0)|=|\m(\lambda(0))|=|\lambda(0)|=n_0$, and $|\alpha(i)|+|\beta(i-1)|=|\mu(i-1)|+|\lambda(i)|=n_i$, for every
$i\in\{1,\ldots,r\}$. Thus, by Theorem~\ref{thm Young vertex}(a), $Y(\alpha|p\beta)$ has Young vertex $\sym{\rho}$ as desired.

\smallskip

Now, in order to prove the isomorphism $Y(\lambda|p\mu)\otimes \sgn\cong Y(\alpha|p\beta)$, it
suffices to verify that the Young--Green correspondents of these modules with respect to $N(\rho)$ are isomorphic.
By Theorem~\ref{thm Young Green} and Lemma~\ref{lemma Green sign}(b), $Y(\lambda|p\mu)\otimes \sgn$ has Young--Green correspondent
\[\mathbf{R}(\lambda|p\mu)\otimes \sgn(N(\rho))=(Y^{\lambda(0)}\boxtimes R_p(\lambda(1)|\mu(0))\boxtimes \cdots\boxtimes R_{p^r}(\lambda(r)|\mu(r-1)))\otimes\sgn(N(\rho)).\]
Furthermore, together with Lemma~\ref{lemma R sign}, we have
\begin{align*}
&\ \mathbf{R}(\lambda|p\mu)\otimes \sgn(N(\rho))\\
\cong&\ (Y^{\lambda(0)}\otimes \sgn)\boxtimes (R_{p}(\lambda(1)|\mu(0))\otimes\sgn)\boxtimes \cdots\boxtimes (R_{p^r}(\lambda(r)|\mu(r-1))\otimes\sgn)\\
\cong&\ Y^{\m(\lambda(0))}\boxtimes R_p(\mu(0)|\lambda(1))\boxtimes\cdots\boxtimes R_{p^r}(\mu(r-1)|\lambda(r))\\
=&\ Y^{\alpha(0)}\boxtimes R_p(\alpha(1)|\beta(0))\boxtimes\cdots\boxtimes R_{p^r}(\alpha(r)|\beta(r-1))=\mathbf{R}(\alpha|p\beta)\,,
\end{align*}
which, by Theorem~\ref{thm Young Green}, is the Young--Green correspondent of $Y(\alpha|p\beta)$ with respect to $N(\rho)$.
Consequently, $Y(\lambda|p\mu)\otimes \sgn\cong Y(\alpha|p\beta)$, by \cite[Satz~3.7]{JGrab} (see also Theorem \ref{thm Grab}).
\end{proof}

\section{Simple Specht Modules}\label{sec simple Specht}


Let $F$ be a field of characteristic $p>0$. For convenience, suppose also that $F$ is algebraically closed.
In \cite{GJAM}, James and Mathas established a characterization of simple Specht $F\mathfrak{S}_n$-modules in the case when
$p=2$, and conjectured a characterization for $p\geq 3$.
The conjecture was proved by the work of Fayers \cite{MF1,MF2} and Lyle \cite{SL}. Following their work, the partitions labelling the simple Specht modules when $p\geq 3$ are now called the JM-partitions.

For the remainder of this section, let $p\geq 3$.
We shall first recall the characterization of simple Specht $F\sym{n}$-modules in terms of JM-partitions and the procedure of inducing any simple Specht module successively to obtain a simple Specht module belonging to a Rouquier block. We introduce a function $\Phi:\P(n)\to \P^2(n)$ (see Definition \ref{D: Phi}), which will eventually give us the correct labelling of a simple Specht module as an indecomposable signed Young module in Theorem \ref{T: irred specht}. We also study the effects of $\Phi$ on JM-partitions and pairs of adjacent JM-partitions (see \ref{noth Rouquier}(b)).

For an integer $m\geq 1$, we denote by $\nu_p(m)$  the largest non-negative integer $\ell$ such that $p^\ell$ divides $m$.

\begin{defn}\label{defn JM}
 A partition $\lambda\in\P(n)$ is called a {\sl JM-partition} if there are no nodes $(a,b)$, $(a,y)$ and $(x,b)$ in the Young diagram
 $[\lambda]$ such that
\smallskip

{\rm (a)}\ $\nu_p(h_\lambda(a,b))>0$, and

\smallskip

{\rm (b)}\ $\nu_p(h_\lambda(x,b))\neq \nu_p(h_\lambda(a,b))\neq \nu_p(h_\lambda(a,y))$.

\noindent We denote the subset of $\P(n)$ consisting of JM-partitions by $\JM(n)$.
\end{defn}

As an immediate consequence of the above definition, one has the following lemma.

\begin{lem}\label{L: conjugate JM}
Let $\lambda\in\P(n)$. Then $\lambda\in\JM(n)$ if and only if $\lambda'\in \JM(n)$.
\end{lem}

For the purpose of our paper, we need the following characterization of JM-partitions proved by Fayers. Recall the $p$-quotient notation we have adopted in \ref{noth abacus}(d).

\begin{prop}[{\cite[Proposition 2.1]{MF2}}]\label{P: JM-partition}
Let $\lambda\in\P(n)$. Then $\lambda\in \JM(n)$ if and only if,
for every abacus display of $\lambda$, there exist some integers  $i,j\in\{0,\ldots, p-1\}$ satisfying the following conditions:

\smallskip

{\rm (F1)}\,  $\lambda^{(k)}=\varnothing$, for  all $k\in\{0,\ldots,p-1\}$ such that  $i\neq k\neq j$;

\smallskip

{\rm (F2)}\,  if a position $i+ap$ on runner $i$ is unoccupied then every position $b>i+ap$ not
on runner $i$ is unoccupied;

\smallskip

{\rm (F3)}\,  if a position $j+cp$ on runner $j$ is occupied then every position $d<j+cp$ not on
 runner $j$ is occupied;

{\rm (F4)}\, $\lambda^{(i)}$ is a $p$-regular {\rm JM-}partition, and

\smallskip

{\rm (F5)}\, $\lambda^{(j)}$ is a $p$-restricted {\rm JM-}partition.
\end{prop}

\begin{rem}\label{rem i and j}
We emphasize again that the abacus display of a partition $\lambda$ depends on the length of the chosen sequence
of $\beta$-numbers. In particular, if $\lambda$ is a JM-partition then the numbers $i$ and $j$ in Proposition~\ref{P: JM-partition}
depend on the chosen $\beta$-numbers. However, the partitions $\lambda^{(i)}$ and $\lambda^{(j)}$ do not depend
on the choice of $\beta$-numbers. In other words, $\lambda^{(i)}$ and $\lambda^{(j)}$ are uniquely determined by the JM-partition $\lambda$.
\end{rem}

The following result gives us the characterization of the partition $\lambda$ when $S^\lambda$ is simple.

\begin{thm}[{\cite{MF1,MF2,SL}}]\label{T: irred classification}
Let $\lambda\in\P(n)$. Then the Specht $F\sym{n}$-module $S^\lambda$ is simple if and only if $\lambda\in \JM(n)$.
\end{thm}

We shall investigate JM-partitions in more detail. More generally, we are mostly interested in partitions with abacus displays satisfying the conditions (F1), (F2) and (F3) in Proposition~\ref{P: JM-partition}, for some runners $i$ and $j$.

\begin{lem}\label{L: i not j}
Let $\lambda\in\P(n)$, and let $i,j\in\{0,\ldots,p-1\}$ be integers satisfying the conditions (F1), (F2) and (F3) in Proposition~\ref{P: JM-partition}, for some abacus display of $\lambda$. If $\kappa_p(\lambda)\neq \lambda$ then $i\neq j$.
\end{lem}

\begin{proof}
Assume that $i=j$ and $\kappa_p(\lambda)\neq \lambda$.
By Proposition \ref{P: JM-partition}(F1), runner $i$ is the only runner such that $\lambda^{(i)}\neq \varnothing$. Since $\kappa_p(\lambda)\neq \lambda$, there are some integers $0\leq a<b$ such that the positions $i+ap$ and $i+bp$ are unoccupied and occupied, respectively. The position $i+ap+1$ is not on runner $i$. By Proposition \ref{P: JM-partition}(F2), the position $i+ap+1$ is unoccupied. On the other hand, by Proposition \ref{P: JM-partition}(F3), the position $i+ap+1$ is occupied, a contradiction.
\end{proof}



The next definition will be crucial to state Theorem~\ref{T: irred specht}, our main result  of this paper.

\begin{defn}\label{D: Phi}
Let $\Phi$ be the map $\Phi:\P(n)\to \P^2(n)$ defined by
$$\Phi(\lambda)=(\alpha|p\beta)\,,$$
where
$p\beta=\lambda'-\lambda'(0)$ and $\alpha=(\lambda'(0))'$.
\end{defn}

\begin{rem}\label{rem Phi} In order to obtain $\Phi(\lambda)$ for a given $\lambda=(\lambda_1,\ldots,\lambda_k)\in\P(n)$, one proceeds as follows. Successively remove all possible vertical rim $p$-hooks from the rightmost, that is, the  $\lambda_1$th column of $[\lambda]$, and let $\beta_{\lambda_1}$ be the total number of vertical $p$-hooks removed.
Next remove all possible vertical rim $p$-hooks from the
$(\lambda_1-1)$st column of the Young diagram of the resulting partition, and let
$\beta_{\lambda_1-1}$ be the total number of these vertical $p$-hooks.
Proceeding in this way from right to left, we end up with
the partition $\alpha$. Moreover, $\beta=(\beta_1,\ldots,\beta_{\lambda_1})$ counts the number of vertical $p$-hooks removed from each
column of the initial Young diagram $[\lambda]$.

\end{rem}

We shall be mostly interested in $\Phi(\lambda)$ in the case when $\lambda$ is a JM-partition. The next lemma shows, in particular, how $\Phi(\lambda)$ can be read off from the abacus display in Proposition~\ref{P: JM-partition} when $\lambda$ is a JM-partition.


\begin{lem}\label{L: stripping JM partition}
Let $\lambda\in\P(n)$, and let $i,j\in\{0,\ldots,p-1\}$ be integers satisfying the conditions (F1), (F2) and (F3) in Proposition~\ref{P: JM-partition}, for some abacus display of $\lambda$. Then the $p$-core of $\lambda$ is obtained by the independent procedures of stripping off all horizontal and vertical rim $p$-hooks of $\lambda$. Furthermore, one has
\[\Phi(\lambda)=(\kappa_p(\lambda)+p\lambda^{(i)}|p(\lambda^{(j)})')=(\kappa_p(\lambda)+\lambda-\lambda(0)|\lambda'-\lambda'(0))\ .\]
\end{lem}

\begin{proof} 
Let $B$ be the abacus display in the statement. If $\kappa_p(\lambda)=\lambda$ then there is nothing to show. Suppose that $\kappa_p(\lambda)\neq\lambda$. Let $i,j$ be the numbers as in Proposition \ref{P: JM-partition} with respect to $B$. By Lemma~\ref{L: i not j}, $i\neq j$. Label all the beads of $B$ by $B_1,\ldots,B_s$ such that their positions are $b_1>\cdots>b_s$, respectively. Recall how one can read off the partition $\lambda$ from $B$ as explained in \ref{noth abacus}(d). For each $k\in\{1,\ldots,s\}$, there are exactly $\lambda_k$ unoccupied positions $t$ such that $t<b_k$, and let this number be $u_B(B_k)$. So we get $\lambda=(u_B(B_1),u_B(B_2),\ldots,u_B(B_s))$.

Suppose that $\lambda^{(i)}\neq \varnothing$, and let $i+a_0p$ be the highest unoccupied position on runner $i$. Let $k_i$ be the unique largest number such that $i+a_0p<b_{k_i}$. Then $B_1,\ldots,B_{k_i}$ belong to runner $i$, by Proposition~\ref{P: JM-partition}(F2). Moving a bead $B_k$, where $1\leq k\leq k_i$, to the vacant position $b_k-p$ yields the abacus $B'$, where $u_{B'}(B_k)=u_B(B_k)-p$, and $u_{B'}(B_\ell)=u_B(B_\ell)$ whenever $\ell\neq k$. This is equivalent to stripping off a horizontal $p$-hook from the $k$th row of $\lambda$. If $\lambda^{(j)}$ is empty then, inductively, we obtain the $p$-core of $\lambda$ by stripping off all horizontal $p$-hooks from $\lambda$.

Let $k_j$ be the unique smallest number such that $B_{k_j}$ belongs to runner $j$; namely $B_{k_j}$ is the lowest bead on runner $j$. Furthermore, if $B_k$ belongs to runner $j$ then, by Proposition \ref{P: JM-partition}(F3), we have
\[b_{k+1}=b_k-1,\ldots,b_{k+p-1}=b_k-(p-1)\,,\]
and thus $u_B(B_k)=u_B(B_{k+1})=\cdots=u_B(B_{k+p-1})$. Suppose that $\lambda^{(j)}\neq \varnothing$. Moving a bead $B_k$ on runner $j$ to the vacant position $b_k-p$ is equivalent to moving the beads $B_\ell$, one for each $k\leq \ell\leq k+p-1$, from the position $b_\ell$ to the position $b_\ell-1$, and this yields the abacus $B''$, where $u_{B''}(B_\ell)=u_B(B_\ell)-1=u_B(B_k)-1$ if $k\leq \ell\leq k+p-1$, and $u_{B''}(B_\ell)=u_B(B_\ell)$ if $\ell\not\in\{k,k+1,\ldots,k+p-1\}$. This is equivalent to stripping off a vertical $p$-hook from the $\lambda_k$th column of $\lambda$. If $\lambda^{(i)}$ is empty then, inductively, we obtain the $p$-core of $\lambda$ by stripping off all vertical $p$-hooks from $\lambda$.

Suppose that $\lambda^{(i)}\neq \varnothing \neq \lambda^{(j)}$. Since $B_1,\ldots,B_{k_i}$ belong to runner $i$ and $i\neq j$, we have $k_j>k_i$. Furthermore, $i+a_0p>b_{k_j}$ by Proposition \ref{P: JM-partition}(F3). Thus stripping off a horizontal $p$-hook involves a bead of $B_1,\ldots,B_{k_i}$ and positions $t$ such that $t\geq i+a_0p$, and, on the other hand, stripping off a vertical $p$-hook involves $p$ beads $B_k,B_{k+1},\ldots,B_{k+p-1}$ such that $k\geq k_j$ and positions $t$ such that $t<i+a_0p$. Thus horizontal stripping involves the first $k_i$ rows, vertical stripping involves the $(k_j+1)$st and lower rows, and so these two stripping procedures are independent of each other.

We shall now prove that $\Phi(\lambda)$ has the desired form by using induction on the $p$-weight of $\lambda$. We deal with the case when $\lambda^{(i)}\neq \varnothing$ (the other case is similar).
Move every bead $B_k$ of $B_{k_i},\ldots,B_1$
in turn from the position $b_k$ to $b_k-p$. This is equivalent to stripping off $k_i$ horizontal $p$-hooks, one for each row, from the first $k_i$ rows of $\lambda$. Let the abacus obtained be $B'''$. The abaci $B,B'''$ have exactly the same configuration except on runner $i$. The abacus $B'''$ represents the partition $\mu=\lambda-(p^{k_i})$, where $(p^{k_i})$ is the partition $(p,\ldots,p)\vdash k_ip$. Furthermore, we have \[\mu^{(i)}=\lambda^{(i)}-(1^{k_i})\] and $\mu^{(k)}=\lambda^{(k)}$, for all $k\neq i$. Also, by the previous paragraph, we have that \[\mu'-\mu'(0)=\lambda'-\lambda'(0).\] We claim that $\mu$ again satisfies Proposition~\ref{P: JM-partition}(F1)-(F3). Observe that

\smallskip

(i)\, $\mu^{(k)}=\varnothing$ whenever $i\neq k\neq j$;

\smallskip

(ii)\, suppose that a position $i+ap$ is unoccupied on runner $i$ of $B'''$. Then $i+ap\geq i+a_0p$. If a position $b>i+ap$ is not on the runner $i$ of $B'''$ then the position $b$ is unoccupied in $B$, and hence is also unoccupied in $B'''$;

\smallskip

(iii)\, suppose that a position $j+cp$ is occupied in $B'''$. Then $j+cp\leq b_{k_j}$. If a position $d<j+cp$ is not on runner $j$ then the position is occupied in $B$, and hence is also occupied in $B'''$.





\medskip

\noindent
By induction, \[\Phi(\mu)=(\kappa_p(\mu)+p\mu^{(i)}|p(\mu^{(j)})')=(\kappa_p(\mu)+\mu-\mu(0)|\mu'-\mu'(0)).\] Let $\Phi(\lambda)=(\alpha|p\beta)$. Then
\begin{align*}
p\beta&=\lambda'-\lambda'(0)=\mu'-\mu'(0)=p(\mu^{(j)})'=p(\lambda^{(j)})'\\
\alpha&=\kappa_p(\mu)+p\mu^{(i)}+(p^{k_i})=\kappa_p(\mu)+p\lambda^{(i)}=\kappa_p(\lambda)+p\lambda^{(i)}.
\end{align*}
The case when $\lambda^{(j)}\neq \varnothing$ can be treated similarly by considering $\lambda'$.
\end{proof}

Together with Theorem \ref{T:twists indecomposable signed Young modules}, we obtain the following useful corollary.

\begin{cor}\label{C: tensor JM partitions}
Suppose that $\lambda\in\JM(n)$. We have
$$Y(\Phi(\lambda))\otimes \sgn\cong Y(\Phi(\lambda'))\,.$$
\end{cor}

\begin{proof}
By Lemma \ref{L: conjugate JM}, the conjugate partition $\lambda'$ is also
a JM-partition of $n$.
Set $p\sigma:=\lambda-\lambda(0)$ and $p\tau:=\lambda'-\lambda'(0)$. Then,
by Proposition~\ref{P: JM-partition} and Lemma \ref{L: stripping JM partition}, we have
$\Phi(\lambda)=(\kappa_p(\lambda)+p\sigma|p\tau)$ and $\Phi(\lambda')=(\kappa_p(\lambda')+p\tau|p\sigma)$.
Since $\kappa_p(\lambda)$ is a $p$-core, we have $\m(\kappa_p(\lambda))=\kappa_p(\lambda)'=\kappa_p(\lambda')$, by \ref{noth Specht}, and also $(\kappa_p(\lambda)+p\sigma)(0)=\kappa_p(\lambda)$. So, by
Theorem~\ref{T:twists indecomposable signed Young modules}, we get
\begin{align*}
Y(\Phi(\lambda))\otimes \sgn=Y(\kappa_p(\lambda)+p\sigma|p\tau)\otimes \sgn&\cong Y(\m(\kappa_p(\lambda))+p\tau|p\sigma)\\
&=Y(\kappa_p(\lambda')+p\tau|p\sigma)=Y(\Phi(\lambda'))\,.
\end{align*}
\end{proof}

\begin{noth}\label{noth Rouquier}{\bf Rouquier blocks and adjacent JM-partitions.}\,
Recall from \ref{noth blocks Sn} that, for $\lambda\in\P(n)$, we denote by $b_{\kappa_p(\lambda)}$ the block of $F\sym{n}$
labelled by the $p$-core $\kappa_p(\lambda)$ and the $p$-weight $\omega_p(\lambda)$ of $\lambda$ (or equivalently, labelled by the $p$-content $c_p(\lambda)$).

\smallskip

(a)\, Following the notation used by Fayers in \cite{MF2}, one calls $b_{\kappa_p(\lambda)}$ a {\sl Rouquier block} if $\kappa_p(\lambda)$ (and thus also $\lambda$) admits an abacus display such that the number of beads on  runner $i+1$ exceeds the number of beads on  runner $i$ by at least $\omega_p(\lambda)-1$, for all $i\in\{0,\ldots,p-2\}$.

\smallskip

(b)\, We shall follow Fayers's idea  \cite[\S3]{MF2} to induce a simple Specht $F\mathfrak{S}_n$-module $S^\lambda$ to a simple Specht module lying in a Rouquier block as follows.

Suppose that $\lambda\in\JM(n)$, and consider an abacus display $B$
of $\lambda$.
Suppose that $B$
is obtained from an $s$-element $\beta$-set.
Suppose further that, for some $\ell\in\{1,\ldots,p-1\}$, there are $r\geq 1$ more beads on runner $\ell-1$ than on runner $\ell$
and that whenever there is a bead in position $ap+\ell$ then there is also a bead in position $ap+\ell-1$ .
Putting this differently, the Young diagram $[\lambda]$ has precisely $r$ addable nodes of $p$-residue
$\overline{\ell-s}$ and no removable nodes of this $p$-residue, where $\overline{\ell-s}\in\{0,\ldots,p-1\}$ is the residue
of $\ell-s$ modulo $p$.
Swapping runners $\ell-1$ and $\ell$ of $B$
yields an abacus display of a JM-partition $\mu\in\JM(n+r)$, and the Young diagram $[\mu]$ is obtained by
adding all addable nodes of $p$-residue $\overline{\ell-s}$ to $[\lambda]$.

One calls $\lambda$ and $\mu$  {\sl a pair of adjacent JM-partitions}.
Also, the blocks $b_{\kappa_p(\lambda)}$ and $b_{\kappa_p(\mu)}$
of $F\sym{n}$ and $F\sym{n+r}$, respectively,  are said to form {\sl a pair of adjacent blocks}.
\end{noth}

\begin{lem}\label{L: conjugate pair of JM} If $\lambda\in\JM(n)$ and $\mu\in\JM(n+r)$ form a pair of adjacent JM-partitions then so do $\lambda'$ and $\mu'$.
\end{lem}

\begin{proof}
Since $\lambda$ and $\mu$ form  a pair of adjacent partitions, there is some $t\in\{0,\ldots,p-1\}$ such that
$[\lambda]$ has precisely $r$ addable nodes of $p$-residue $t$ and has no removable nodes of $p$-residue $t$. Equivalently, $\lambda'$ has precisely $r$ addable nodes of $p$-residue $p-t$ and has no removable nodes of $p$-residue $p-t$. By Lemma~\ref{L: conjugate JM}, both $\lambda',$ and $\mu'$ are JM-partitions. Moreover, since $[\mu]$ is obtained by adding all $t$-addable nodes to $[\lambda]$, the diagram $[\mu']$ is obtained by adding all addable nodes of $p$-residue $p-t$ to $[\lambda']$. Thus $\lambda'$ and $\mu'$ form a pair of adjacent partitions.
\end{proof}

For our purpose, we need to analyze the effect of the function $\Phi$ on a pair of adjacent JM-partitions.

\begin{lem}\label{L: simple partition p-adic}
Suppose that $\lambda\in\JM(n)$ and $\mu\in\JM(n+r)$ form a pair of adjacent JM-partitions, for some $r>0$, and let $B_\mu$ be the abacus display of $\mu$ obtained from an abacus display $B_\lambda$ of $\lambda$ by exchanging runners $\ell-1$ and $\ell$.
Let $i_\lambda,j_\lambda,i_\mu,j_\mu\in\{0,\ldots,p-1\}$  be numbers satisfying the conditions (F1)--(F5) in Proposition \ref{P: JM-partition} with respect to $B_\lambda$ and $B_\mu$, respectively.
Then one has $\Phi(\lambda)=(\kappa_p(\lambda)+p\sigma|p\beta)$ and $\Phi(\mu)=(\kappa_p(\mu)+p\sigma|p\beta)$, where
$\sigma=\lambda^{(i_\lambda)}=\mu^{(i_\mu)}$ and $\beta'=\lambda^{(j_\lambda)}=\mu^{(j_\mu)}$.
 \end{lem}

\begin{proof}
%
We shall justify that $\lambda^{(i_\lambda)}=\mu^{(i_\mu)}$ and $\lambda^{(j_\lambda)}=\mu^{(j_\mu)}$. Once we have done that, we obtain our desired result, since, by Lemma~\ref{L: stripping JM partition}, we have $\Phi(\lambda)=(\kappa_p(\lambda)+p\lambda^{(i_\lambda)}|p(\lambda^{(j_\lambda)})')$ and $\Phi(\mu)=(\kappa_p(\mu)+p\mu^{(i_\mu)}|p(\mu^{(j_\mu)})')$.

If $\lambda$ is a $p$-core then $\mu$ is a $p$-core. In this case, $\lambda^{(i_\lambda)}=\mu^{(i_\mu)}=\lambda^{(j_\lambda)}=\mu^{(j_\mu)}=\varnothing$.
Suppose that $\lambda$ is not a $p$-core, in which case $\mu$ is not a $p$-core either.
Suppose further that $\lambda^{(i_\lambda)}\neq\varnothing$; the case when $\lambda^{(j_\lambda)}\neq\varnothing$ can be treated similarly. Let $pa_0+i_\lambda$ be the highest unoccupied position on runner $i_\lambda$ of $B_\lambda$.
If $i_\lambda\notin\{\ell-1,\ell\}$ then we may choose $i_\mu=i_\lambda$, so that $\lambda^{(i_\lambda)}=\mu^{(i_\mu)}$. Once $\lambda^{(i_\lambda)}$ (respectively, $\mu^{(i_\mu)}$) is identified, $\lambda^{(j_\lambda)}$ (respectively, $\mu^{(j_\mu)}$) is uniquely determined. Thus $\lambda^{(j_\lambda)}=\mu^{(j_\mu)}$.

Next suppose that $i_\lambda\in\{\ell-1,\ell\}$.
Since $\lambda^{(i_\lambda)}\neq\varnothing$, there is some $c>a_0$ such that $cp+i_\lambda$ is occupied. If $i_\lambda=\ell$ then
position $cp+\ell-1$ would have to be both occupied by \ref{noth Rouquier}(b), and unoccupied, by Proposition \ref{P: JM-partition}(F2), a contradiction. Thus we must have $i_\lambda=\ell-1$.
%
%
%
%
%
%

Let $i_\lambda=\ell-1$. Then, by Proposition \ref{P: JM-partition}(F2), for any $b\geq a_0$, the position $\ell+bp$ is unoccupied. Since $\lambda^{(i_\lambda)}\neq\varnothing$, we have a bead at the position $\ell-1+cp$, for some $c>a_0$. Exchanging runners $\ell-1$ and $\ell$ of $B_\lambda$ we obtain that the position $\ell-1+cp$ is unoccupied, while position $\ell+cp$ is occupied in $B_\mu$. This shows that $j_\mu\neq \ell$. Since $\mu^{(\ell)}\neq\varnothing$, we necessarily have that $\ell=i_\mu$. So we obtain that $\mu^{(i_\mu)}=\mu^{(\ell)}=\lambda^{(\ell-1)}=\lambda^{(i_\lambda)}$, and hence $\mu^{(j_\mu)}=\lambda^{(j_\lambda)}$.
\end{proof}

\begin{rem}\label{rem adj}
The previous lemma shows that
if $\lambda\in\JM(n)$ and $\mu\in\JM(n+r)$ form a pair of adjacent JM-partitions
then the horizontal and vertical $p$-hooks removed from $[\lambda]$ and $[\mu]$ to obtain their respective $p$-cores are identical. Furthermore, $[\kappa_p(\mu)]$ is obtained by adding $r$ nodes of a particular $p$-residue to $[\kappa_p(\lambda)]$.
\end{rem}

\smallskip

We shall now turn our attention to the procedure of inducing a simple Specht module successively to obtain a simple Specht module in a Rouquier block. We begin by introducing some notation.  Moreover, we briefly mention some aspects of the modular branching rules of symmetric groups due to Kleshchev, which we shall need in the proof of Proposition \ref{P: ind and res of Specht}. For details we refer the reader to
\cite[Theorems~11.2.10--11.2.11]{AK}.

\begin{nota}\label{nota block induction}
Suppose that $G$ is a finite group and $H\leq G$. Let $B$ be a block of $FG$ with block idempotent $e_B$, and let $b$ be a block of $FH$ with block idempotent
$e_b$. Given an $FG$-module $M$ and an $FH$-module $N$, we then get the following $FH$-module and $FG$-module, respectively: \begin{align*}
M\downarrow_b&:=e_b\cdot \res_H^G(M)\,,\\
N\uparrow^B&:=e_B\cdot\ind_H^G(N)\,.
\end{align*}
\end{nota}

\begin{noth}{\bf Modular branching rules.\,}\label{noth branching}
Suppose that $D$ is a simple $F\sym{n}$-module belonging to a block with $p$-content $(\gamma_0,\ldots,\gamma_{p-1})$. Suppose further that $i\in\{0,\ldots,p-1\}$  and $r\in\NN$ with $r\leq n$ and such that $(\gamma_0,\ldots,\gamma_{i-1},\gamma_i-r,\gamma_{i+1},\ldots,\gamma_{p-1})$ is the $p$-content of a block $b$ of  $F\sym{n-r}$. Then $D\downarrow_b$ is either $\{0\}$ or there is an $F\sym{n-r}$-module $N$ belonging to $b$ such that $$D\downarrow_b\cong \bigoplus_{r!}N\,.$$
Similarly, if $s\in \NN$ is such that $(\gamma_0,\ldots,\gamma_{i-1},\gamma_i+s,\gamma_{i+1},\ldots,\gamma_{p-1})$ is the $p$-content of a block $B$ of $F\sym{n+s}$ then $D\uparrow^B$ is either $\{0\}$ or there is an $F\sym{n+s}$-module $M$ belonging to $B$ such that
$$D\uparrow^B\cong \bigoplus_{s!}M\,.$$
\end{noth}

\begin{prop}\label{P: ind and res of Specht} Let $\lambda\in \JM(n)$. Then there is a sequence of JM-partitions $\lambda=\boldsymbol{\varrho}_1,\boldsymbol{\varrho}_2,\ldots,\boldsymbol{\varrho}_t$ of $n=n_1<n_2<\cdots<n_t$, respectively, such that, for each $i\in\{1,\ldots,t-1\}$, one has the following:

\smallskip

{\rm (a)}\,  $\boldsymbol{\varrho}_i$ and $\boldsymbol{\varrho}_{i+1}$ form a pair of adjacent JM-partitions,

\smallskip

{\rm (b)}\, $S^{\boldsymbol{\varrho}_i}{\uparrow^{b_{\kappa_p(\boldsymbol{\varrho}_{i+1})}}}\cong\bigoplus_{(n_{i+1}-n_i)!} S^{\boldsymbol{\varrho}_{i+1}}$,

\smallskip

{\rm (c)}\, $S^{\boldsymbol{\varrho}_{i+1}}{\downarrow_{b_{\kappa_p(\boldsymbol{\varrho}_{i})}}}\cong\bigoplus _{(n_{i+1}-n_i)!} S^{\boldsymbol{\varrho}_{i}}$,

\smallskip

{\rm (d)}\,  $b_{\kappa_p(\boldsymbol{\varrho}_t)}$ is a Rouquier block.
\end{prop}

\begin{proof}
Following \cite[Lemma~3.1, Lemma~3.3]{MF2}, there is a sequence of
JM-partitions $\lambda=\boldsymbol{\varrho}_1,\boldsymbol{\varrho}_2,\ldots,\boldsymbol{\varrho}_t$ of $n=n_1<n_2<\cdots<n_t$, respectively, satisfying parts (a) and (d) and such that
$S^{\boldsymbol{\varrho}_i}{\uparrow^{b_{\kappa_p(\boldsymbol{\varrho}_{i+1})}}}$ has a filtration of $(n_{i+1}-n_i)!$ copies of the simple Specht module $S^{\boldsymbol{\varrho}_{i+1}}$ and
$S^{\boldsymbol{\varrho}_{i+1}}{\downarrow_{b_{\kappa_p(\boldsymbol{\varrho}_{i})}}}$ has a filtration of $(n_{i+1}-n_i)!$ copies of the simple Specht module $S^{\boldsymbol{\varrho}_{i}}$,
for $i\in\{1,\ldots,t-1\}$. In particular, $S^{\boldsymbol{\varrho}_i}{\uparrow^{b_{\kappa_p({\boldsymbol{\varrho}_{i+1}})}}}\neq\{0\}$ and
$S^{\boldsymbol{\varrho}_{i+1}}{\downarrow_{b_{\kappa_p(\boldsymbol{\varrho}_{i}})}}\neq\{0\}$.
 Fix $i\in\{1,\ldots,t-1\}$, and set $r:=n_{i+1}-n_i$.
Moreover, let $s\in\{0,\ldots,p-1\}$ be such that $[\boldsymbol{\varrho}_{i+1}]$ is obtained by adding $r$ nodes of $p$-residue $s$
to $[\boldsymbol{\varrho}_i]$.
As we have seen in Remark~\ref{rem adj}, the Young diagram $[\kappa_p(\boldsymbol{\varrho}_{i+1})]$ is obtained
by adding  to $[\kappa_p(\boldsymbol{\varrho}_i)]$ precisely $r$ nodes of a particular $p$-residue.
Hence, as recalled in \ref{noth branching}, by Kleshchev's modular branching rules \cite[Theorems 11.2.10--11.2.11]{AK}, there are an $F\sym{n_i}$-module $N$ and an $F\sym{n_{i+1}}$-module
 $M$ such that
$$S^{\boldsymbol{\varrho}_i}{\uparrow^{b_{\kappa_p(\boldsymbol{\varrho}_{i+1})}}}\cong \bigoplus_{r!}M\quad\text{ and }\quad
S^{\boldsymbol{\varrho}_{i+1}}{\downarrow_{b_{\kappa_p(\boldsymbol{\varrho}_i)}}}\cong \bigoplus_{r!}N\,.$$
Since $M\neq\{0\}$, $M$ has some composition factor $D$.
But  we already know that $S^{\boldsymbol{\varrho}_i}{\uparrow^{b_{\kappa_p(\boldsymbol{\varrho}_{i+1})}}}$ has a composition series whose factors are isomorphic to $S^{\boldsymbol{\varrho}_{i+1}}$ with multiplicity $r!$. This forces $S^{\boldsymbol{\varrho}_{i+1}}\cong M$.
Analogously, we get $N\cong S^{\boldsymbol{\varrho}_{i}}$.
\end{proof}

\begin{rem} In fact, \cite[Lemma~3.1, Lemma~3.3]{MF2} and the proof of Proposition \ref{P: ind and res of Specht} say that, as soon as we have a pair of adjacent JM-partitions $\lambda$ and $\mu$ such that $|\mu|>|\lambda|$, we have both
$S^\lambda{\uparrow^{b_{\kappa_p(\mu)}}}\cong \bigoplus_{r!}S^\mu$ and $S^\mu{\downarrow_{b_{\kappa_p(\lambda)}}}\cong \bigoplus_{r!}S^\lambda$, where $r=|\mu|-|\lambda|$.
\end{rem}

\section{Labelling Simple Specht Modules  as Signed Young Modules}\label{sec labels}

Throughout this section, let again $F$ be  a field of characteristic $p\geq 3$, and let $n\in \mathbb{Z}^+$. For convenience, suppose also that $F$ is algebraically closed.
In \cite{DH}, Hemmer proved that every simple Specht $F\sym{n}$-module is isomorphic to a signed Young $F\sym{n}$-module.
So, given a simple Specht $F\sym{n}$-module $S^\lambda$,  how does one determine $(\alpha|p\beta)\in\P^2(n)$ satisfying
$Y(\alpha|p\beta)\cong S^\lambda$?
This was posed as an open problem in \cite[Problem~5.2]{DH}.

A conjecture concerning the correct labelling was put forward by the first author in  \cite[Vermutung~5.4.2]{SDanz} and, independently, by the second author \cite[Conjecture 8.2]{KJL2}, and Orlob \cite[Vermutung A.1.10]{JO}. In this section, we confirm the conjecture by proving the following theorem. Recall the map $\Phi$ defined in Definition \ref{D: Phi}.


\begin{thm}\label{T: irred specht}
Suppose that $\lambda\in\JM(n)$.
The simple Specht $F\sym{n}$-module $S^\lambda$ is isomorphic to $Y(\Phi(\lambda))$.
\end{thm}

Our strategy of proving Theorem~\ref{T: irred specht} is similar to that employed by Fayers in  \cite{MF2} and by Hemmer in \cite{DH}. More precisely, we shall first reduce the verification of Theorem~\ref{T: irred specht} to the Rouquier block case and then show that the theorem holds true in this case.



\begin{lem}\label{L: second part}
Suppose that $\lambda\in\JM(n)$ and $\mu\in\JM(n+r)$ form a pair of adjacent JM-partitions, for some $r>0$. Suppose that $S^\mu\cong Y(\Phi(\mu))$ and $S^\lambda\cong Y(\alpha|p\beta)$. Then $\Phi(\lambda)=(\gamma|p\beta)$ for some partition $\gamma$.
Moreover, $p\beta=\lambda'-\lambda'(0)$.
\end{lem}

\begin{proof}
 Let $m=n+r$, and suppose that $\Phi(\mu)=(\xi|p\zeta)$. Since $Y(\alpha|p\beta)$ is a direct summand of $M(\alpha|p\beta)$, by Proposition~\ref{P: ind and res of Specht}(b), we have $S^\mu\mid \ind_{\sym{n}}^{\sym{m}}(S^\lambda)$, and hence, by Lemma \ref{L:ind and res of permutation}(a), we get
 $$ Y(\xi|p\zeta)\cong S^\mu\mid  \ind_{\sym{n}}^{\sym{m}}(S^\lambda)\cong  \ind_{\sym{n}}^{\sym{m}}(Y(\alpha|p\beta))\mid M(\alpha\#(1^{m-n})|p\beta)\,,$$
 where $\alpha\# (1 ^{m-n})$ denotes the concatenation of the partitions $\alpha$ and $(1^{m-n})$.
 Therefore, we have
\begin{equation}\label{eqn 1}
 (\xi|p\zeta)\unrhd (\alpha\#(1^{m-n})|p\beta)\,,
 \end{equation}
 by \cite[2.3(8)]{SD}.

 On the other hand, by Proposition~\ref{P: ind and res of Specht}(c), we have $S^\lambda\mid \res_{\sym{n}}^{\sym{m}}(S^\mu)$, and hence $S^{\lambda}\mid \res_{\sym{n}}^{\sym{m}}(M(\xi|p\zeta))$. By Lemma~\ref{L:ind and res of permutation}(b),
 $S^\lambda\cong Y(\alpha|p\beta)$ is isomorphic to a direct summand of $M(\delta|\partial)$ for suitable
 partitions $\delta$ and $\partial$. Therefore, we have
 \begin{equation}\label{eqn 2}
 (\alpha|p\beta)\unrhd (\delta|\partial)\,,
 \end{equation}
by \cite[2.3(8)]{SD} again.
Using (\ref{eqn 1}) and (\ref{eqn 2}), we obtain $p|\zeta|\leq p|\beta|\leq |\partial|$. By  Lemma~\ref{L:ind and res of permutation}(b), $\partial$ is the partition of the rearrangement of a composition obtained from $p\zeta$ after removing some nodes, so that we necessarily have $|\partial|\leq p|\zeta|$. This shows that $|\partial|=p|\beta|=p|\zeta|$, and hence we have removed no node from $p\zeta$ to obtain $\partial$, that is, $\partial=p\zeta$.
We must also have $|\alpha|=|\delta|$,  $|\zeta|=|\beta|$, and $|\xi|=|\alpha\# (1^{m-n})|$.
So, from (\ref{eqn 1}) and (\ref{eqn 2}), we  get $p\zeta\unrhd p\beta \unrhd \partial=p\zeta$. Thus $\beta=\zeta$, and there is some
partition $\gamma$ with $\Phi(\lambda)=(\gamma|p\beta)$, by Lemma~\ref{L: simple partition p-adic}. By the definition of $\Phi$, we have also
$p\beta=\lambda'-\lambda'(0)$.
\end{proof}

\begin{lem}\label{L: core is the p-adic exp}
If $S^\lambda$ is a simple Specht $F\sym{n}$-module and if $(\alpha|p\beta)\in\P^2(n)$ is such that
$S^\lambda\cong Y(\alpha|p\beta)$ then $\kappa_p(\alpha)=\alpha(0)=\kappa_p(\lambda)$.
\end{lem}

\begin{proof}
Let $w=\omega_p(\lambda)$, so that, by \cite[Theorem~6.2.45, 4.1.22, 4.1.24]{GJAK},  the Sylow $p$-subgroups of $\sym{pw}$ are defect groups of the block
containing $Y(\alpha|p\beta)$. By Theorem~\ref{thm Young vertex}, the indecomposable signed Young module $Y(\alpha|p\beta)$ has
Green vertex $P_{\rho}$, where $\rho=\Rho(\alpha|p\beta)$, so that $P_\rho=_{\sym{n}}Q$ for some subgroup $Q\leq \sym{pw}$.


In consequence of Kn\"orr's Theorem
\cite{RK} and \cite[Proposition~1.4]{JBO}, $Q$ does not have any fixed points on $\{1,\ldots,pw\}$ (see also \cite[Proposition~3.4]{SDanz1}). Therefore, $P_\rho$ fixes exactly $|\alpha(0)|$ numbers in $\{1,\ldots,n\}$, which implies that
\begin{equation}\label{eqn pw}
wp=n-|\alpha(0)|\,.
\end{equation}
On the other hand, recall from \ref{noth blocks Sn} and \ref{noth signed young} that both $S^\lambda$ and $Y(\alpha|p\beta)$ lie in the block $b_{\kappa_p(\lambda)}=b_{\kappa_p(\alpha)}$ of $F\sym{n}$, so that
$\kappa_p(\lambda)=\kappa_p(\alpha)$. Since $\alpha(0)$ is obtained from $\alpha$ by stripping off horizontal $p$-hooks only, we have $|\alpha(0)|\geq |\kappa_p(\alpha)|=|\kappa_p(\lambda)|=n-wp$. Using equation (\ref{eqn pw}), we have $|\alpha(0)|=|\kappa_p(\alpha)|$, and hence we conclude that $\alpha(0)=\kappa_p(\alpha)=\kappa_p(\lambda)$.
%
\end{proof}

\begin{prop}\label{P: reduction}
Theorem~\ref{T: irred specht} holds true if it holds true for all simple Specht modules belonging to Rouquier blocks.
\end{prop}

\begin{proof}
Assume that Theorem~\ref{T: irred specht} holds true for all simple Specht modules of symmetric groups belonging to Rouquier blocks.
Let $S^\lambda$ be any simple Specht $F\sym{n}$-module. By Proposition~\ref{P: ind and res of Specht}, there is a
sequence of JM-partitions $\lambda=\boldsymbol{\varrho}_1,\ldots, \boldsymbol{\varrho}_t$ of natural numbers $n=n_1<n_2<\cdots <n_t$
such that $\boldsymbol{\varrho}_i$ and $\boldsymbol{\varrho}_{i+1}$ are adjacent, for every $i\in\{1,\ldots,t-1\}$, and such that
$b_{\kappa_p(\boldsymbol{\varrho}_t)}$ is a Rouquier block. We argue by reverse induction on $i$ to show that
$S^{\boldsymbol{\varrho}_i}\cong Y(\Phi(\boldsymbol{\varrho}_i))$, for all $i\in\{1,\ldots,t\}$. If $i=t$ then we are done. Suppose that $i<t$ and that we have already
proved $S^{\boldsymbol{\varrho}_j}\cong Y(\Phi(\boldsymbol{\varrho}_j))$, for all $j\geq i+1$.
%
By Corollary~\ref{C: tensor JM partitions}, we then also have
$$S^{\boldsymbol{\varrho}_{i+1}'}\cong (S^{\boldsymbol{\varrho}_{i+1}})^*\otimes\sgn \cong S^{\boldsymbol{\varrho}_{i+1}}\otimes\sgn\cong Y(\Phi(\boldsymbol{\varrho}_{i+1}))\otimes \sgn\cong Y(\Phi(\boldsymbol{\varrho}_{i+1}'))\,.$$
Lemma~\ref{L: conjugate pair of JM} guarantees that $\boldsymbol{\varrho}_i'$ and $\boldsymbol{\varrho}_{i+1}'$ also form a pair of adjacent JM-partitions. Applying Lemma~\ref{L: second part} to both of the pairs $\boldsymbol{\varrho}_i,\boldsymbol{\varrho}_{i+1}$ and $\boldsymbol{\varrho}'_i,\boldsymbol{\varrho}'_{i+1}$, we have $S^{\boldsymbol{\varrho}_i}\cong Y(\alpha|\boldsymbol{\varrho}_i'-\boldsymbol{\varrho}_i'(0))$ and $S^{\boldsymbol{\varrho}_i'}\cong Y(\xi|\boldsymbol{\varrho}_i-\boldsymbol{\varrho}_i(0))$,
for suitable partitions $\alpha$ and $\xi$. By Lemma~\ref{L: core is the p-adic exp},
$\kappa_p(\alpha)=\alpha(0)=\kappa_p(\boldsymbol{\varrho}_i)$, and, by Theorem~\ref{T:twists indecomposable signed Young modules},
\begin{align*}
Y(\xi|\boldsymbol{\varrho}_i-\boldsymbol{\varrho}_i(0))&\cong S^{\boldsymbol{\varrho}_i}\otimes \sgn\cong Y(\alpha|\boldsymbol{\varrho}_i'-\boldsymbol{\varrho}_i'(0))\otimes \sgn\\
&\cong Y(\m(\alpha(0))+\boldsymbol{\varrho}_i'-\boldsymbol{\varrho}_i'(0)|\alpha-\alpha(0))\,.
\end{align*}
The equation above implies, in particular, that
$\alpha=\alpha(0)+(\boldsymbol{\varrho}_i-\boldsymbol{\varrho}_i(0))=\kappa_p(\boldsymbol{\varrho}_i)+(\boldsymbol{\varrho}_i-\boldsymbol{\varrho}_i(0))$.
Applying Lemma~\ref{L: stripping JM partition}, we have $\Phi(\boldsymbol{\varrho}_i)=(\kappa_p(\boldsymbol{\varrho}_i)+(\boldsymbol{\varrho}_i-\boldsymbol{\varrho}_i(0))|\boldsymbol{\varrho}_i'-\boldsymbol{\varrho}_i'(0))$. Therefore, $S^{\boldsymbol{\varrho}_i}\cong Y(\Phi(\boldsymbol{\varrho}_i))$ as required.
\end{proof}



Proposition \ref{P: reduction} reduces the proof of Theorem \ref{T: irred specht} to the Rouquier block case, and we are working towards completing it. In order to complete an argument of the proof of Theorem~\ref{T: irred specht} in this case, we need to consider lifting of signed Young permutation modules and simple Specht modules for $F\sym{n}$-modules.  As such, for the remainder of this section, we reintroduce the notation $S^\lambda_R$ and $M_R(\alpha|\beta)$ for the Specht and  signed Young permutation $RG$-modules respectively, where $R\neq\{0\}$ is a commutative ring (see \ref{noth perm Specht}(a)).

\begin{noth}\label{rem modular system}{\bf Lifting of trivial source modules.}\
Consider a $p$-modular system $(K,\mathcal{O},F)$. That is, $\mathcal{O}$
is a complete discrete valuation ring with quotient field $K$ of characteristic 0, and with residue field $F$.
Given a finite group $G$ and an $\mathcal{O}G$-module $M$, we obtain the $KG$-module $K\otimes_{\mathcal{O}} M$ as well
as the $FG$-module  $F\otimes_{\mathcal{O}} M$.

\smallskip

Recall, for instance from \cite[\S4.3]{NT}, that one can define Green vertices and Green sources for indecomposable $\mathcal{O}G$-modules. As for $FG$-modules one calls an $\mathcal{O}G$-module  a {\sl $p$-permutation module} if
it is a direct summand of a permutation $\mathcal{O}G$-module. By \cite[0.4]{MB} the indecomposable $p$-permutation
$\mathcal{O}G$-modules are precisely those with trivial Green sources.

\smallskip

We say that an $FG$-module $N$ {\sl lifts} to $\mathcal{O}G$ if there is some $\mathcal{O}G$-module $\hat{N}$ such that
$F\otimes_{\mathcal{O}} \hat{N}\cong N$. In general, such a lift is not unique. However, if $N$ is a $p$-permutation $FG$-module then
there is a $p$-permutation $\mathcal{O}G$-module $\hat{N}$, unique up to isomorphism, such that $F\otimes_{\mathcal{O}} \hat{N}\cong N$, by \cite[Theorem 4.8.9]{NT}. As well, if $N$ is a $p$-permutation $FG$-module then consider an indecomposable direct sum
decomposition $N=N_1\oplus\cdots \oplus N_r$. For $i\in\{1,\ldots,r\}$, let $\hat{N}_i$ be the unique $p$-permutation lift of
$N_i$ to $\mathcal{O}G$. Then $\hat{N}_1\oplus\cdots\oplus \hat{N}_r$ is the unique $p$-permutation lift of $N$ to $\mathcal{O}G$,
by  \cite[Theorem 4.8.9]{NT}.
\end{noth}

\begin{lem}\label{lemma young lift}
Let $(\lambda|\zeta)\in\P^2(n)$. Then the $\mathcal{O}\sym{n}$-module $M_{\mathcal{O}}(\lambda|\zeta)$ is the
$p$-permutation lift of the $F\sym{n}$-module $M_F(\lambda|\zeta)$.
\end{lem}

\begin{proof}
Clearly, $M_{\mathcal{O}}(\lambda|\zeta)$ is a lift of $M_F(\lambda|\zeta)$. So, by \cite[Theorem 4.8.9]{NT}, we only have to show
that $M_{\mathcal{O}}(\lambda|\zeta)$ is a $p$-permutation $\mathcal{O}\sym{n}$-module. By definition, we have
$M_{\mathcal{O}}(\lambda|\zeta)=\ind_{\sym{\lambda}\times\sym{\zeta}}^{\sym{n}}(\mathcal{O}\boxtimes \sgn)$. Since
$p\geq 3$, every Sylow $p$-subgroup $P$ of $\sym{\lambda}\times \sym{\zeta}$ is also contained in $\mathfrak{A}_n$. Therefore, the
$\mathcal{O}[\sym{\lambda}\times \sym{\zeta}]$-module $\mathcal{O}\boxtimes \sgn$ restricts to the trivial $\mathcal{O}P$-module, which has Green vertex $P$ and is its own Green $P$-source.
Hence $\mathcal{O}\boxtimes \sgn\mid\ind_{P}^{\sym{\lambda}\times \sym{\zeta}}(\mathcal{O})$, and then also
$M_{\mathcal{O}}(\lambda|\zeta)\mid \ind_P^{\sym{n}}(\mathcal{O})$. So $M_{\mathcal{O}}(\lambda|\zeta)$ is a $p$-permutation
module.
\end{proof}

\begin{lem}\label{lemma specht lift}
Suppose that $\lambda\in\mathrm{JM}(n)$, so that the Specht $F\sym{n}$-module
$S_F^\lambda$ is simple. Then $S_F^\lambda$ lifts to an $\mathcal{O}\sym{n}$-module. Moreover, for every lift $S$ of $S^\lambda_F$ to $\mathcal{O}\sym{n}$, one has $K\otimes_{\mathcal{O}}S\cong S^\lambda_K$.
\end{lem}

\begin{proof}
Since $S^\lambda_F$ is simple and thus a signed Young module, it is a $p$-permutation module, by Theorem~\ref{thm Young vertex}.
Hence $S^\lambda_F$ lifts to an $\mathcal{O}\sym{n}$-module $S$, by \cite[Theorem~4.8.9]{NT}.  By \cite[Proposition 16.17]{CR1}, the $K\sym{n}$-module
$K\otimes_{\mathcal{O}} S$ is then a simple $K\sym{n}$-module, so that there is some $\mu\in\P(n)$ with $K\otimes_{\mathcal{O}} S\cong S^\mu_K$. The $\mathcal{O}\sym{n}$-module $S$ is an $\mathcal{O}$-form of $S^\mu_K$, as is the $\mathcal{O}\sym{n}$-module
$S^\mu_{\mathcal{O}}$. By \cite[Proposition 16.16]{CR1}, the $F\sym{n}$-modules $F\otimes_{\mathcal{O}} S$ and
$F\otimes_{\mathcal{O}} S^\mu_{\mathcal{O}}$ both have the same composition factors, and are thus isomorphic to $S^\lambda_F$.
But $F\otimes_{\mathcal{O}} S^\mu_{\mathcal{O}}\cong S^\mu_F$.
Since
$p\geq 3$, this forces $\lambda=\mu$, by \cite[Corollary 13.17]{GJ1}.
\end{proof}

\begin{rem}\label{rem T}
Let $\lambda\in\JM(n)$ be such that the Specht module $S^\lambda_F$ lies in a Rouquier block. Furthermore, let $i,j\in\{0,\ldots,p-1\}$ be integers satisfying the conditions (F1)--(F5) in Proposition~\ref{P: JM-partition} for some abacus display $B$ of $\lambda$. By Lemma \ref{L: stripping JM partition}, we have $\Phi(\lambda)=(\kappa_p(\lambda)+p\sigma|p\tau)$, where $\lambda^{(i)}=\sigma$ and $(\lambda^{(j)})'=\tau$. The partition $\kappa_p(\lambda)+p\sigma$ is obtained from $B$ by moving all beads on runner $j$ as high as possible, which is equivalent to removing all vertical $p$-hooks from $[\lambda]$ (see the proof of Lemma~\ref{L: stripping JM partition}). Let $B'$ be the abacus display obtained. So $\kappa_p(\lambda)+p\sigma$ is a $p$-regular JM-partition with the same runners $i$ and $j$ satisfying conditions (F1)--(F5) for the abacus $B'$, that is, $S^{\kappa_p(\lambda)+p\sigma}_F$ is again a  simple Specht module. We conclude that $S^{\kappa_p(\lambda)+p\sigma}_F\cong Y^{\kappa_p(\lambda)+p\sigma}\cong Y(\kappa_p(\lambda)+p\sigma|\varnothing)$ (see, for instance, \cite[Proposition 1.1]{DH}). In particular, the $F\sym{n}$-module
\begin{equation}\label{eqn T}
T:=\ind_{\sym{|\kappa_p(\lambda)|+p|\sigma|}\times\sym{p\tau}}^{\sym{n}}(S^{\kappa_p(\lambda)+p\sigma}_F\boxtimes\sgn)
\end{equation}
is isomorphic to a direct summand of the signed Young permutation module $M_F(\kappa_p(\lambda)+p\sigma|p\tau)$.
In \cite{DH} Hemmer examined  the module $T$ and, in particular, showed that $S^\lambda_F$ is isomorphic to a direct summand of
$T$ (see \cite[proof of Theorem~4.2]{DH}).
\end{rem}

We are now in the position to prove Theorem \ref{T: irred specht}.

\begin{proof}[Proof of Theorem \ref{T: irred specht}]
By Proposition \ref{P: reduction}, it suffices to consider the Rouquier block case.
Suppose that $\lambda$ is a JM-partition of $n$ such that the Specht $F\sym{n}$-module $S^\lambda_F$ belongs  to some Rouquier block. By Lemma \ref{L: stripping JM partition}, we have $\Phi(\lambda)=(\kappa_p(\lambda)+p\sigma|p\tau)$, for some partitions $\sigma$ and $\tau$.
Suppose that $(\alpha|p\beta)\in\P^2(n)$ is such that $S^\lambda_F\cong Y(\alpha|p\beta)$.
By Lemma \ref{L: core is the p-adic exp}, we have $\kappa_p(\alpha)=\alpha(0)=\kappa_p(\lambda)$.
We aim to prove that $(\alpha|p\beta)=\Phi(\lambda)$.

Consider the $F\sym{n}$-module
$T$ defined as in (\ref{eqn T}).
As mentioned in Remark~\ref{rem T}, we have
$Y(\alpha|p\beta)\cong S^\lambda_F\mid T\mid M_F(\kappa_p(\lambda)+p\sigma|p\tau)$, so that $(\alpha|p\beta)\unrhd (\kappa_p(\lambda)+p\sigma|p\tau)$, by \cite[2.3(8)]{SD}. In particular, we have $p|\beta|\leq p|\tau|$. We shall show that, in fact, $p|\beta|=p|\tau|$.


\smallskip

The Specht $F\sym{n}$-module $S^{\lambda'}_F$ is also simple,  lies in a Rouquier block, and  satisfies
$$S^{\lambda'}_F\cong (S^\lambda_F)^*\otimes \sgn\cong S^\lambda_F\otimes \sgn\cong Y(\alpha|p\beta)\otimes\sgn\cong Y(\m(\alpha(0))+p\beta|\alpha-\alpha(0))\,,$$
by Theorem~\ref{T:twists indecomposable signed Young modules}.

Now, replacing $\lambda$ by $\lambda'$ in Remark~\ref{rem T}, we also deduce that
$S^{\lambda'}_F$ is isomorphic to a direct summand of the signed permutation module $M_F(\kappa_p(\lambda')+p\tau|p\sigma)$. Thus, since $\m(\kappa_p(\lambda))=\kappa_p(\lambda)'=\kappa_p(\lambda')$ (see the final paragraph of \ref{noth Specht}(b)), we have
$$(\m(\alpha(0))+p\beta|\alpha-\alpha(0))=(\m(\kappa_p(\lambda))+p\beta|\alpha-\kappa_p(\lambda))\unrhd (\kappa_p(\lambda)'+p\tau|p\sigma)\,.$$
In particular, we have $|\m(\kappa_p(\lambda))|+p|\beta|\geq |\kappa_p(\lambda)'|+p|\tau|$, and hence $p|\beta|\geq p|\tau|$. This shows that $p|\beta|=p|\tau|$, and thus $|\alpha|=|\kappa_p(\lambda)|+p|\sigma|$. The condition $(\alpha|p\beta)\unrhd (\kappa_p(\lambda)+p\sigma|p\tau)$ is thus equivalent to $\alpha\unrhd\kappa_p(\lambda)+p\sigma$ and $p\beta\unrhd p\tau$. The next aim is to show that we indeed have $\alpha=\kappa_p(\lambda)+p\sigma$ and $p\beta=p\tau$, which will then complete the proof
of the theorem.

We have $S^\lambda_F\mid M_F(\alpha|p\beta)$. We claim that $S^\lambda_K\mid M_K(\alpha|p\beta)$. By Lemma~\ref{lemma young lift}, $M_{\mathcal{O}}(\alpha|p\beta)$ is the unique
$p$-permutation lift of $M_F(\alpha|p\beta)$ to $\mathcal{O}\sym{n}$. Moreover, let $S$ be the unique $p$-permutation lift
of $S^\lambda_F$ to $\mathcal{O}\sym{n}$. Then $S\mid M_{\mathcal{O}}(\alpha|p\beta)$, by \cite[Theorem~4.8.9]{NT}, and thus also
$(K\otimes_{\mathcal{O}} S)\mid (K\otimes_{\mathcal{O}} M_{\mathcal{O}}(\alpha|p\beta))$ as $K\sym{n}$-modules.
By Lemma~\ref{lemma specht lift}, we know that $K\otimes_{\mathcal{O}} S\cong S^\lambda_K$, and clearly
$K\otimes_{\mathcal{O}} M_{\mathcal{O}}(\alpha|p\beta)\cong M_K(\alpha|p\beta)$. Since $K$ has characteristic 0, this means that
$S^\lambda_K$ is a composition factor of $M_K(\alpha|p\beta)$. This justifies our claim.

By Corollary~\ref{cor comp series}, the composition multiplicity
of $S^\lambda_K$ in $M_K(\alpha|p\beta)$ equals the number of
semistandard $\lambda$-tableaux of type $(\alpha|p\beta)$.
%
%
%
%
Hence there must exist a semistandard $\lambda$-tableau of type $(\alpha|p\beta)$.
Let $\sigma=(\sigma_1,\ldots,\sigma_r)$ be such that $\sigma_r\neq 0$.
Then $\lambda_i=(\kappa_p(\lambda))_i+p\sigma_i$, for all $1\leq i\leq r$. Recall that $\alpha\unrhd \kappa_p(\lambda)+p\sigma$. If $\alpha_1>\lambda_1$ then there would be no semistandard $\lambda$-tableau of type $(\alpha|p\beta)$, since the $\alpha_1$ copies of the number $1$ would have to be assigned to the first row of $[\lambda]$ and clearly there would not be enough nodes in the first row of $[\lambda]$.
Thus $\alpha_1=\lambda_1$. Arguing inductively we obtain that $\lambda_i=\alpha_i$ for all $1\leq i\leq r$, that is,  $\alpha_i=(\kappa_p(\lambda))_i+p\sigma_i$, for all $1\leq i\leq r$. Using the fact that $\alpha(0)=\kappa_p(\lambda)$, we have $\alpha-\alpha(0)=(p\sigma_1,\ldots,p\sigma_r,\ldots)$, and hence $\alpha-\alpha(0)=(p\sigma_1,\ldots,p\sigma_r)$, since $|\alpha|-|\alpha(0)|=p|\sigma|$. This shows that $\alpha=\kappa_p(\lambda)+p\sigma$.

\smallskip

Repeating the argument in the previous paragraph with the Specht module $S^{\lambda'}_F$, which is isomorphic to a
direct summand of the signed Young permutation module $M_F(\kappa_p(\lambda)'+p\tau|p\sigma)$, we deduce that $\m(\alpha(0))+p\beta=\kappa_p(\lambda)'+p\tau$. Since $\m(\alpha(0))=\m(\kappa_p(\lambda))=\kappa_p(\lambda)'$, we obtain $p\beta=p\tau$.  Altogether, this gives
 $$(\alpha|p\beta)=(\kappa_p(\lambda)+p\sigma|p\tau)=\Phi(\lambda).$$ The proof is now complete.
\end{proof}




\section{Some consequences}\label{S: conseq}

In this section, we derive some immediate consequences from our results in Theorem~\ref{T:twists indecomposable signed Young modules} and Theorem~\ref{T: irred specht}.
So far, we know the Green vertices and Green correspondents of indecomposable signed Young modules. We shall
see below that also the cohomological varieties and complexities of indecomposable signed Young modules can easily
be determined; in particular, we obtain these invariants for simple Specht modules.
%
Furthermore, using a result of Gill \cite{CG}, we determine the periods, and, in the case of weight $1$, the minimal projective resolutions of non-projective periodic indecomposable signed Young modules. Throughout this section, $F$ denotes an algebraically closed field of characteristic $p>0$.

From now on we denote the trivial $FG$-module, for every finite group $G$, simply by $F$.


\begin{noth}\label{nota cohom}{\bf Cohomological variety and complexity of modules.}\  For the theory and notation of cohomological variety and complexity, we refer the reader to \cite[\S5]{DB}, and briefly summarize here the notation we shall use in the following.

Suppose that $G$ is any finite group. Then $H^\cdot(G,F)$ denotes the cohomology ring $H^*(G,F)$ if $p=2$,
and it denotes the subring generated by elements of even degrees if $p\geq 3$.
If $H$ is a subgroup of $G$ then one has a restriction map $\res_{G,H}:H^{\cdot}(G,F)\to H^{\cdot}(H,F)$, which then induces a map $\res^*_{G,H}:V_H\to V_G$ on the level of
{\sl cohomological varieties} of $H$ and $G$.

Let $M$ be an $FG$-module. The {\sl cohomological variety} $V_G(M)$ of $M$ is the maximal ideal spectrum of $V_G$ containing a certain ideal. Furthermore, we denote by $c_G(M)$ the {\sl complexity} of $M$, which is the rate of growth of a minimal projective resolution of $M$, or equivalently, the dimension of the cohomological variety $V_G(M)$. For precise definitions of $V_G(M)$ and $c_G(M)$,
see \cite[\S5]{DB}.

The cohomological variety $V_G(M)$ and complexity $c_G(M)$ of an indecomposable $FG$-module $M$ are related to its Green vertices and sources as follows.
\end{noth}

\begin{lem}[\protect{\cite[Proposition~2.15]{KJL2}}]\label{lem cohom sce}
Let $M$ be an indecomposable $FG$-module, let $Q$ be a Green vertex of $M$, and let $S$ be a Green $Q$-source of $M$.
Then one has $V_G(M)=\res^*_{G,Q}(V_Q(S))$. Moreover, the complexity $c_G(M)$ of $M$ equals the complexity
$c_Q(S)$ of $S$. If the dimension of $S$ is coprime to $p$ then $c_Q(S)$ equals the $p$-rank of $Q$.
\end{lem}

\begin{proof}
Since $M\mid \ind^G_Q(S)$, we have $V_G(M)\subseteq \res^*_{G,Q}(V_Q(S))$, by \cite[Proposition 5.7.5]{DB} and \cite[Proposition 8.2.4]{LE}. Conversely, since $S\mid \res^G_Q(M)$, we have $\res^*_{G,Q}(V_Q(S))\subseteq \res^*_{G,Q}(V_Q(M))\subseteq V_G(M)$, by \cite[Proposition 5.7.5]{DB}.

By  \cite[Proposition 2.12(vii) and (ix)]{KJL2}), we have $c_G(M)=c_Q(S)$. The last statement follows from \cite[Corollary 5.8.5]{DB}.
\end{proof}

Applying the general theory to indecomposable signed Young modules, we obtain the following proposition. 

\begin{prop}\label{P: variety and complx of signed Young}
Let $p\geq 3$, let $(\lambda|p\mu)\in\P^2(n)$, and let $\rho=\Rho(\lambda|p\mu)$. 
The cohomological variety of the indecomposable signed Young module $Y(\lambda|p\mu)$ is
\[\res^*_{\sym{n},P_\rho}(V_{P_\rho}(F))\,,\]
where $P_\rho$ is the fixed Sylow $p$-subgroup of the Young vertex $\sym{\rho}$ of $Y(\lambda|p\mu)$. In particular, $Y(\lambda|p\mu)$ has complexity $|\mu|+\frac{1}{p}(|\lambda|-|\lambda(0)|)$.
\end{prop}

\begin{proof} By Theorem \ref{thm Young vertex}(b), $Y(\lambda|p\mu)$ has Green vertex $P_\rho$ and trivial Green $P_\rho$-source $F$.  By Lemma~\ref{lem cohom sce}, the first assertion is justified. Since the trivial $FP_\rho$-module $F$ has dimension coprime to $p$, the complexity of $Y(\lambda|p\mu)$ is the $p$-rank of $P_\rho$, which is \[(|\lambda(1)|+|\mu(0)|)+p(|\lambda(2)|+|\mu(1)|)+p^2(|\lambda(3)|+|\mu(2)|)+\cdots=|\mu|+\frac{|\lambda|-|\lambda(0)|}{p}.\]
\end{proof}

Recall that, for every finite group $G$, an $FG$-module has complexity $0$ if and only if it is projective, and has complexity 1 if and only if it is non-projective periodic. Applying Proposition~\ref{P: variety and complx of signed Young} we obtain the following corollary. Note that
the assertion in part (b) can also be found in the paper of Hemmer and Nakano \cite[Corollary~3.3.3]{DHDN2} in the case of
Young modules.

\begin{cor}\label{C: periodic signed Young}
Let $p\geq 3$, and let $(\lambda|p\mu)\in\P^2(n)$.
The indecomposable signed Young $F\sym{n}$-module $Y(\lambda|p\mu)$ is

\smallskip

{\rm (a)}\, projective if and only if $\mu=\varnothing$ and $\lambda$ is $p$-restricted,

\smallskip

{\rm (b)}\,  non-projective periodic if and only if either $\mu=(1)$ and $\lambda$ is $p$-restricted, or $\mu=\varnothing$ and $\lambda=\lambda(0)+(p)$.
\end{cor}


Let $\alpha\in\P(n)$.
Recall from \cite[Corollary~13.18]{GJ1} that, in the case when $p$ is odd, or $p=2$ and $\alpha$ is $2$-regular, the Specht $F\sym{n}$-module $S^\alpha$ is indecomposable.
If $S^\alpha$ is simple then  $S^\alpha$ is isomorphic to $D^{\alpha^R}$, where $\alpha^R$ is the {\sl $p$-regularization} of the partition $\alpha$ (see \cite[6.3.59]{GJAK}). In the case when $p$ is odd, Hemmer's result in \cite{DH} implies that a simple Specht $F\sym{n}$-module
$S^\alpha$ and the associated simple module $D^{\alpha^R}$ have trivial Green sources. Our result Theorem~\ref{T: irred specht} allows us to describe the Green vertices, Green correspondents, cohomological varieties and complexities for all simple Specht
$F\sym{n}$-modules in the case where $p\geq 3$.
Recall the map $\Phi:\P(n)\to \P^2(n)$ in Definition~\ref{D: Phi}.

\begin{cor}\label{C: 1} Let $p\geq 3$, and suppose that $S^\alpha$ is a simple Specht $F\sym{n}$-module. Let $\Phi(\alpha)=(\lambda|p\mu)$, let $\rho=\Rho(\lambda|p\mu)$, and let $N(\rho)$ be the normalizer of $\sym{\rho}$ in $\sym{n}$. Then the $F\sym{n}$-module $S^\alpha\cong D^{\alpha^R}$ has
\smallskip

{\rm (a)}\, Green vertex a Sylow $p$-subgroup $P_\rho$ of $\sym{\rho}$,

\smallskip

{\rm (b)}\, Green correspondent $\res_H^{N(\rho)}(\mathbf{R}(\lambda|p\mu))$ with respect to $N_{\sym{n}}(P_\rho)\leq H\leq N(\rho)$,

\smallskip

{\rm (c)}\, cohomological variety $\res^*_{\sym{n},P_{\rho}}(V_{P_{\rho}}(F))$, and

\smallskip

{\rm (d)}\, complexity equal to the $p$-weight of $\alpha$.
\end{cor}

\begin{proof} By Theorem \ref{T: irred specht}, Theorem~\ref{thm Young vertex}, Corollary~\ref{cor Green} and Proposition~\ref{P: variety and complx of signed Young}, we obtain (a), (b), (c), and the complexity of $S^\alpha$ is  $|\mu|+\frac{1}{p}(|\lambda|-|\lambda(0)|)$. By the definition of $\Phi$ and Lemma~\ref{L: stripping JM partition}, we get (d).
\end{proof}

\begin{rem}\label{rem known vertices}
Suppose again that $p\geq 3$.
Corollary~\ref{C: 1} thus reveals the class of all simple Specht $F\sym{n}$-modules and the class of simple $F\sym{n}$-modules
\[\{D^{\alpha^R}\ :\ \alpha\in\JM(n)\}\,,\]
whose Green vertices, Green correspondents, cohomological varieties and complexities are known. In the literature, there are other
special classes of Specht or simple $F\sym{n}$-modules whose Green vertices or complexities are computed:
results concerning Green vertices and complexities can, for instance, be found in  \cite{EGKJLMW}, \cite{KJL2,KJL,KJL4}, \cite{GMMP} and
\cite{MW,MW2} for Specht modules, and in \cite{SDanz1,SDanz2,DG,DKZ}, \cite{DHDN2}, \cite{KJLKMT2}, \cite{JMRZ} and \cite{MW}
for simple modules.

In fact, Corollary~\ref{C: 1} recovers Wildon's result \cite[Theorem~2]{MW} when $p\geq 3$.
\end{rem}


\begin{cor}
Suppose that $p\geq 3$. If $p$ does not divide $n$ then, for any $r\in \{0,\ldots,n-1\}$, a Sylow $p$-subgroup of $\sym{n-r-1}\times \sym{r}$ is a Green vertex of
the Specht module $S^{(n-r,1^r)}$.
\end{cor}
\begin{proof}
By \cite{Peel}, the Specht module $S^{(n-r,1^r)}$ is simple.
Let $n-r-1=\sum_{i=0}^k s_ip^i$ and $r=\sum_{i=0}^\ell r_ip^i$ be the $p$-adic expansions of the integers $n-r-1$ and $r$, respectively. Then Lemma~\ref{L: stripping JM partition} gives
\[\Phi((n-r,1^r))=\left ((1+s_0,1^{r_0})+\sum_{i= 1}^k p^i(s_i)|\sum_{i= 1}^\ell p^i(r_i)\right ).\]
Moreover,  $S^{(n-r,1^r)}\cong Y(\Phi((n-r,1^r)))$ by Theorem~\ref{T: irred specht}.
By Theorem~\ref{thm Young vertex},
$Y(\Phi((n-r,1^r)))$ has Young vertex $\sym{\rho}$, where $\rho=(1^{1+s_0+r_0},p^{s_1+r_1},(p^2)^{s_2+r_2},\ldots)$.
By Corollary~\ref{C: 1}, the Sylow $p$-subgroups of $\sym{\rho}$ are  Green vertices of
$S^{(n-r,1^r)}$.
On the other hand, every Sylow $p$-subgroup of
\begin{align*}
\ & (\sym{1})^{s_0}\times (\sym{p})^{s_1}\times \cdots\times (\sym{p^k})^{s_k}\times (\sym{1})^{r_0}\times (\sym{p})^{r_1}\times \cdots \times (\sym{p^\ell})^{r_\ell}\\
=&\  \sym{(1^{s_0},p^{s_1},\ldots,(p^k)^{s_k})}\times \sym{(1^{r_0},p^{r_1},\ldots,(p^\ell)^{r_\ell})} =_{\sym{n}}\sym{\rho}\,
\end{align*}
is  a Sylow $p$-subgroup of $\sym{n-r-1}\times\sym{r}$, by \cite[4.1.22]{GJAK}.
\end{proof}

\begin{rem}\label{rem char 2}
Now we deal with the simple Specht modules when $p=2$. In general, for any prime $p$, if $\lambda$ is $p$-regular and $S^\lambda$ is simple then $S^\lambda\cong Y^\lambda$ (see, for instance, \cite[Proposition~1.1]{DH}).

Now let $p=2$. The simple Specht modules $S^\lambda$ have been classified by James and Mathas \cite[Main Theorem]{GJAM}. By the classification, any of the associated partitions $\lambda$ is either

\smallskip

(a)\,  $2$-regular,

\smallskip

(b)\, $2$-restricted, or

\smallskip

(c)\, $\lambda=(2,2)$.

\medskip

Notice that, if $\lambda$ is $2$-restricted then $S^\lambda\cong S^{\lambda'}\otimes \sgn\cong S^{\lambda'}$, where $n=|\lambda|$. Hence also $S^\lambda\cong Y^{\lambda'}$, in this case. Since, by \cite{KE,JGrab,DHDN2}, one knows the Green vertices, Green correspondents, cohomological varieties and complexities of Young modules, one also knows these invariants
for all simple Specht $F\sym{n}$-modules $S^\lambda$, except when $\lambda=(2,2)$. When $\lambda=(2,2)$, it is easy to compute that the simple Specht module $S^{(2,2)}$ has the Green vertex $E=\langle (1,2)(3,4),(1,3)(2,4)\rangle$ (see, for instance, \cite[Lemma 6]{MW}) and trivial Green $E$-source. Since $E$ is normal in $\sym{4}$, the Green correspondent of $S^{(2,2)}$ with respect to $N_{\sym{4}}(E)=\sym{4}$ is the $F\sym{4}$-module $S^{(2,2)}$ itself. By Lemma~\ref{lem cohom sce}, the cohomological variety of $S^{(2,2)}$ is $\res^*_{\sym{4},E}(V_E(F))$, and hence the complexity of $S^{(2,2)}$ is $2$.

Together with Corollary \ref{C: 1}, we have a complete understanding of the properties we desired about all simple Specht modules over fields of positive characteristics.
\end{rem}



\begin{rem} Corollary~\ref{C: periodic signed Young}(b) classifies the class of non-projective periodic signed Young modules. In what follows, we make use of Gill's result \cite{CG} to obtain the periods of non-projective periodic indecomposable signed Young modules and their minimal projective resolutions in the weight 1 case, and hence deduce similar results for all simple Specht modules and the class of simple modules mentioned in Remark~\ref{rem known vertices}. By \cite[Theorem 1]{CG}, a non-projective periodic Young module $Y^\lambda$ has period $2p-2$ when $p\geq 3$, and it has period $1$ when $p=2$. Moreover,
if $Y^\lambda$ belongs to a block of weight 1 then one can describe a minimal projective resolution of $Y^\lambda$ \cite[\S4.1]{CG}.
\end{rem} 

\begin{cor}\label{C: 3}
Let $p\geq 3$. Every non-projective periodic indecomposable signed Young $F\sym{n}$-module has period $2p-2$. 
\end{cor}

\begin{proof}
Let $(\lambda|p\mu)\in\P^2(n)$ be such that the signed Young module $Y(\lambda|p\mu)$ is non-projective and periodic.
 By Corollary \ref{C: periodic signed Young}, there are only two possibilities. In the case when $\mu=\varnothing$, the signed Young module $Y(\lambda|p\mu)$ is isomorphic to the Young module $Y^\lambda$. Thus we only need to consider the signed
 Young $F\sym{n}$-modules $Y(\lambda|p(1))$, where $\lambda$ is a $p$-restricted partition. By Theorem \ref{T:twists indecomposable signed Young modules},
 \[Y(\lambda|p(1))\otimes \sgn\cong Y(\m(\lambda)+p(1)|\varnothing)\,.\]
 Since $Y(\m(\lambda)+p(1)|\varnothing)\cong Y^{\m(\lambda)+p(1)}$, it has period $2p-2$, by \cite[Theorem 1]{CG}. Hence the same holds true for $Y(\lambda|p(1))$.
\end{proof}

\begin{cor}
Let $S^\alpha$ be a simple, non-projective and periodic Specht $F\sym{n}$-module.
If $p\geq 3$ then $S^\alpha\cong D^{\alpha^R}$ has period $2p-2$. If $p=2$ then $S^\alpha\cong D^{\alpha^R}$ has period $1$.
\end{cor}

\begin{proof}
Let $p\geq 3$. So $S^\alpha\cong Y(\Phi(\alpha))$ by Theorem~\ref{T: irred specht}.
So $S^\alpha$ has period $2p-2$, by Corollary~\ref{C: 3}. If $p=2$ then, as we have seen in Remark~\ref{rem char 2}, the
result of James and Mathas \cite{GJAM} implies that $S^\alpha$ is isomorphic to one of the indecomposable Young
$F\sym{n}$-modules $Y^\alpha$ or $Y^{\alpha'}$, or  that $\alpha=(2,2)$. By Remark~\ref{rem char 2}, the Specht
$F\sym{4}$-module $S^{(2,2)}$ has complexity 2, and hence is not periodic.
By \cite[Theorem 1]{CG}, every non-projective indecomposable Young $F\sym{n}$-module that is periodic has period $1$.
Thus the assertion also holds true when $p=2$.
\end{proof}

\begin{noth}\label{proj resol}{\bf Minimal projective resolutions.} \,
Suppose that $p\geq 3$.
A minimal projective resolution of a non-projective periodic Young $F\sym{n}$-module belonging to a block of $p$-weight $1$ is known, by the result of Gill \cite[\S4.1]{CG}. We make use of this result to obtain a minimal projective resolution of a non-projective periodic indecomposable signed Young $F\sym{n}$-module of the form $Y(\lambda|p(1))$ that belongs to a block of $p$-weight 1.
Note that $\lambda$ must then be  a $p$-core, i.e. $\kappa_p(\lambda)=\lambda$, which will be fixed from now on. Recall from \ref{noth Specht} that
$\m(\lambda)=\lambda'$.
So, by Theorem \ref{T:twists indecomposable signed Young modules}, we know that
\[Y(\lambda|p(1))\cong Y(\m(\lambda)+p(1)|\varnothing)\otimes \sgn\cong Y^{\m(\lambda)+p(1)}\otimes\sgn=Y^{\lambda'+p(1)}\otimes\sgn\,.\]

There are precisely $p$ partitions of $n$ of $p$-weight 1 whose $p$-core is $\lambda$. We shall denote these by
$\boldsymbol{\varrho}_0=\lambda+p(1),\boldsymbol{\varrho}_1,\ldots,\boldsymbol{\varrho}_{p-1}$.
Moreover, we may choose our labelling such that $\boldsymbol{\varrho}_0\rhd \boldsymbol{\varrho}_1\rhd\cdots\rhd \boldsymbol{\varrho}_{p-1}$. For each $i\in\{1,\ldots,p-1\}$, the partition $\boldsymbol{\varrho}_i$
is $p$-restricted.
Taking the conjugates of these partitions, we obtain the partitions $\lambda'+p(1)=\boldsymbol{\varrho}'_{p-1}\rhd\cdots\rhd \boldsymbol{\varrho}'_1\rhd \boldsymbol{\varrho}'_0$, which lie in the block of weight 1 labelled by the $p$-core $\lambda'$.

By \cite[\S4.1]{CG}, the Young module $Y^{\lambda'+p(1)}=Y^{\boldsymbol{\varrho}_{p-1}'}$ has a minimal projective resolution of the
form
\begin{equation}\label{eqn resol}
\cdots \to Y^{\boldsymbol{\varrho}_{p-2}'}\to Y^{\boldsymbol{\varrho}_{p-3}'}\to\cdots\to Y^{\boldsymbol{\varrho}_{0}'}\to Y^{\boldsymbol{\varrho}_{0}'}\to Y^{\boldsymbol{\varrho}_{1}'}\to\cdots \to Y^{\boldsymbol{\varrho}_{p-2}'}\to Y^{\boldsymbol{\varrho}_{p-1}'}\to \{0\}\,.
\end{equation}
Since the partitions $\boldsymbol{\varrho}_{0}',\ldots,\boldsymbol{\varrho}_{p-2}'$
are $p$-restricted, we have $\boldsymbol{\varrho}_{i}'=\boldsymbol{\varrho}_{i}'(0)$, for all $i\in\{0,\ldots,p-2\}$.
Theorem~\ref{T:twists indecomposable signed Young modules} implies $Y^{\boldsymbol{\varrho}_{i}'}\otimes\sgn\cong Y^{\m(\boldsymbol{\varrho}_{i}')}$, for
$i\in\{0,\ldots,p-2\}$, and, after tensoring (\ref{eqn resol}) by $\sgn$, we obtain a minimal projective resolution
\begin{multline}\label{eqn resol 2}
\cdots \to Y^{\m(\boldsymbol{\varrho}_{p-2}')}\to Y^{\m(\boldsymbol{\varrho}_{p-3}')}\to \cdots\to Y^{\m(\boldsymbol{\varrho}_{0}')}\to Y^{\m(\boldsymbol{\varrho}_{0}')}\to Y^{\m(\boldsymbol{\varrho}_{1}')}\to\cdots\\ \to Y^{\m(\boldsymbol{\varrho}_{p-2}')}\to Y(\lambda|p(1))\to \{0\}
\end{multline}
of $Y^{\lambda'+p(1)}\otimes\sgn\cong Y(\lambda|p(1))$. The partitions $\boldsymbol{\varrho}_1,\ldots,\boldsymbol{\varrho}_{p-1},\boldsymbol{\varrho}_{0}',\ldots,\boldsymbol{\varrho}_{p-2}'$ are $p$-restricted, and the partitions $\boldsymbol{\varrho}_0,\ldots,\boldsymbol{\varrho}_{p-2},\boldsymbol{\varrho}_{1}',\ldots,\boldsymbol{\varrho}_{p-1}'$ are $p$-regular. The structures of both of the Specht modules $S^{\boldsymbol{\varrho}_i}$ and $S^{\boldsymbol{\varrho}_{i}'}$
are well known, for every $i\in\{0,\ldots,p-1\}$. In fact, every block of $F\sym{n}$ of $p$-weight 1 is Scopes equivalent, thus, in particular, Morita equivalent to the principal
block of $F\sym{p}$ (see \cite{Scopes}). The Specht modules in the principal block of $F\sym{p}$ are labelled by the hook partitions
of $p$, whose module structures have been determined by Peel \cite{Peel}.
With our notation, we get
$S^{\boldsymbol{\varrho}_0}\cong D^{\boldsymbol{\varrho}_0}$, $S^{\boldsymbol{\varrho}_{p-1}}\cong D^{\boldsymbol{\varrho}_{p-2}}$,
$S^{\boldsymbol{\varrho}_{0}'}\cong D^{\boldsymbol{\varrho}_{1}'}$, $S^{\boldsymbol{\varrho}_{p-1}'}\cong D^{\boldsymbol{\varrho}_{p-1}'}$. Furthermore, for $i\in\{1,\ldots,p-2\}$, the Specht modules $S^{\boldsymbol{\varrho}_i}$
and $S^{\boldsymbol{\varrho}_{i}'}$ both have composition length 2 and Loewy structures
$$S^{\boldsymbol{\varrho}_i}=\begin{bmatrix}  D^{\boldsymbol{\varrho}_i}\\D^{\boldsymbol{\varrho}_{i-1}} \end{bmatrix}\quad\text{ and }\quad S^{\boldsymbol{\varrho}_{i}'}=\begin{bmatrix}  D^{\boldsymbol{\varrho}_{i}'}\\D^{\boldsymbol{\varrho}_{i+1}'} \end{bmatrix}\,,$$
respectively. Hence, by \ref{noth Specht}, we get
$D_{\boldsymbol{\varrho}_{i}'}\cong\Soc(S^{\boldsymbol{\varrho}_{i}'})\cong D^{\boldsymbol{\varrho}_{i+1}'}$, and thus
\begin{align*}
D_{\m(\boldsymbol{\varrho}_{i}')}&\cong D_{\boldsymbol{\varrho}_{i}'}\otimes\sgn\cong D^{\boldsymbol{\varrho}_{i+1}'}\otimes\sgn\cong \Hd(S^{\boldsymbol{\varrho}_{i+1}'}\otimes\sgn)\\
&\cong \Hd((S^{\boldsymbol{\varrho}_{i+1}})^*)\cong \Soc(S^{\boldsymbol{\varrho}_{i+1}})\cong D_{\boldsymbol{\varrho}_{i+1}}\,,
\end{align*}
for all $i\in\{0,\ldots,p-2\}$.
This implies $\m(\boldsymbol{\varrho}_{i}')=\boldsymbol{\varrho}_{i+1}$, for $i\in\{0,\ldots,p-2\}$.
Together with (\ref{eqn resol 2}), this gives the minimal projective resolution
\[\cdots \to Y^{\boldsymbol{\varrho}_{p-1}}\to Y^{\boldsymbol{\varrho}_{p-2}}\to\cdots\to Y^{\boldsymbol{\varrho}_1}\to Y^{\boldsymbol{\varrho}_1}\to Y^{\boldsymbol{\varrho}_2}\to\cdots\to Y^{\boldsymbol{\varrho}_{p-1}}\to Y(\lambda|p(1))\to \{0\}\]
of $Y(\lambda|p(1))$.

\medskip

In \cite[\S4.1]{CG} Gill  also determines the Heller translates
$\Omega^i(Y^{\boldsymbol{\varrho}_{p-1}'})$ in terms of their Loewy structures, for all $i\in\{0,\ldots, 2p-2\}$.
Since $\Omega^i(Y^{\boldsymbol{\varrho}_{p-1}'})\otimes\sgn\cong \Omega^i(Y^{\boldsymbol{\varrho}_{p-1}'}\otimes\sgn)\cong \Omega^i(Y(\lambda|p(1)))$, for all integers $i$, our previous considerations and Gill's result yield the following Loewy structures
of the Heller translates of $Y(\lambda|p(1))$:
$$\Omega^i(Y(\lambda|p(1)))\cong\begin{cases}
                                                       \begin{bmatrix} D_{\boldsymbol{\varrho}_{p-1-i}}\\ D_{\boldsymbol{\varrho}_{p-i}}\end{bmatrix}&\text{ if } 1\leq i\leq p-2\,,\\
                                                       D_{\boldsymbol{\varrho}_1}&\text{ if } i=p-1\,,\\
                                                       \begin{bmatrix} D_{\boldsymbol{\varrho}_{i-p+2}}\\ D_{\boldsymbol{\varrho}_{i-p+1}}\end{bmatrix}&\text{ if } p\leq i\leq 2p-3\,,\\
                                                       Y(\lambda|p(1))&\text{ if } i=2p-2\,.
\end{cases}$$
\end{noth}

\begin{appendix}

\section{Appendix: Proof of Theorem~\ref{thm Young Green}}\label{appendix}

Throughout this appendix, let $F$ be a field of characteristic $p\geq 3$. For convenience, suppose also that $F$ is algebraically closed. In the following, we shall establish a proof of Theorem~\ref{thm Young Green}, following the arguments given by Donkin in \cite[\S5.2]{SD}.

To this end, let $n\geq 1$, and let $(\lambda|p\mu)\in\P^2(n)$. We recall the notation defined in \ref{nota p-adic} and \ref{nota Nrho}. We have the $p$-adic expansions $\lambda=\sum_{i=0}^{r_\lambda}p^i\cdot\lambda(i)$ and $\mu=\sum_{i=0}^{r_\mu}p^i\cdot\mu(i)$ of $\lambda$ and $\mu$, respectively. Let $r:=\max\{r_\lambda,r_\mu+1\}$, let $n_i=|\lambda(i)|+|\mu(i-1)|$, for every $i\in\{0,\ldots,r\}$ (where $\mu(-1)=\varnothing$), and let $\rho=\Rho(\lambda|p\mu)=(1^{n_0},p^{n_1},\ldots,(p^r)^{n_r})$. Furthermore, we identify the normalizer $N_{\sym{n}}(\sym{\rho})$ with $N(\rho)=\sym{n_0}\times (\sym{p}\wr\sym{n_1})\times\cdots \times (\sym{p^r}\wr\sym{n_r})$.

As we have seen in Theorem~\ref{thm Young vertex}(a), the  signed Young $F\sym{n}$-module $Y(\lambda|p\mu)$ has Young vertex $\sym{\rho}$. Let 
\begin{equation*}\label{eqn Arho}
A(\rho):=(\sym{1})^{n_0}\times (\mathfrak{A}_{p})^{n_1}\times\cdots \times (\mathfrak{A}_{p^r})^{n_r}\unlhd N(\rho)\,,
\end{equation*}
so that $N(\rho)/A(\rho)\cong \sym{n_0}\times (\sym{2}\wr\sym{n_1})\times\cdots\times (\sym{2}\wr\sym{n_r})$.
By \cite[1.1(1)]{SD},
the isomorphism classes of indecomposable signed Young $F\sym{n}$-modules with Young vertex $\sym{\rho}$ are
in bijection with the isomorphism classes of indecomposable projective $F[N(\rho)/A(\rho)]$-modules. More precisely, this bijection is induced
by the Young--Green correspondence with respect to $N(\rho)$. Recall the notions of $R_k(\alpha|\beta)$ and $\mathbf{R}(\nu|p\delta)$ from \ref{nota R} and \ref{thm Young Green}. The next proposition together with the observation in Remark~\ref{rem k odd} will
imply that the Young--Green correspondent of $Y(\lambda|p\mu)$ with respect to $N(\rho)$ must be isomorphic
to one of the $FN(\rho)$-modules $\mathbf{R}(\nu|p\delta)$, where $(\nu|p\delta)\in \P^2(n)$.

\begin{prop}\label{prop proj modules}
Let $m\in \NN$, and let $\RP^2(m)$ be the set of pairs $(\alpha|\beta)$ of
$p$-restricted partitions such that $|\alpha|+|\beta|=m$. Then, as $(\alpha|\beta)$ varies over $\RP^2(m)$, the $F[\sym{2}\wr\sym{m}]$-module
$R_2(\alpha|\beta)$ varies over a set of representatives of the isomorphism classes of indecomposable projective $F[\sym{2}\wr\sym{m}]$-modules.
\end{prop}

\begin{proof}
Fix $m_1,m_2\in\NN$ such that $m_1+m_2=m$. For $\alpha\in\RP(m_1)$ and $\beta\in\RP(m_2)$, we consider the
$F[\sym{2}\wr\sym{m}]$-module
$$D_2(\alpha|\beta):=\ind_{\sym{2}\wr(\sym{m_1}\times\sym{m_2})}^{\sym{2}\wr\sym{m}}((F(2)^{\otimes m_1}\boxtimes\sgn(2)^{\otimes m_2})\otimes \Inf_{\sym{m_1}\times\sym{m_2}}^{\sym{2}\wr(\sym{m_1}\times\sym{m_2})}(D_\alpha\boxtimes D_{\m(\beta)}))\,.$$
Recall from Remark~\ref{rem k odd} that $\sgn(2)^{\otimes m_2}=\sgn(\sym{2}\wr\sym{m_2})$, and $F(2)^{\otimes m_1}$ is of course the trivial $F[\sym{2}\wr\sym{m_1}]$-module. As $\alpha$ varies over $\RP(m_1)$, the module $D_\alpha$ varies over a set of representatives of the isomorphism classes of simple $F\sym{m_1}$-modules, and as $\beta$ varies over $\RP(m_2)$, both $D_\beta$ and $D_{\m(\beta)}$ vary over a set of representatives of the isomorphism classes of simple $F\sym{m_2}$-modules.

Thus, by \cite[Theorem~4.3.34]{GJAK},
the modules $D_2(\alpha|\beta)$, as $(\alpha|\beta)$ varies over $\RP^2(m)$, vary over a set of representatives of the isomorphism classes of simple $F[\sym{2}\wr\sym{m}]$-modules. Given $(\alpha|\beta)\in \RP^2(m)$, we have $D_\alpha\boxtimes D_{\m(\beta)}\cong \Hd(Y^{\alpha}\boxtimes Y^{\m(\beta)})$ as $F[\sym{m_1}\times\sym{m_2}]$-modules. Hence $D_2(\alpha|\beta)$ is
isomorphic to a quotient module of $R_2(\alpha|\beta)$. Since $p\neq 2$, both the $F\sym{2}$-modules $F(2)$ and $\sgn(2)$ are projective. Therefore, $R_2(\alpha|\beta)$ is a projective indecomposable $F[\sym{2}\wr\sym{m}]$-module, for every $(\alpha|\beta)\in \RP^2(m)$; a proof of  this can, for instance, be found in \cite[Proposition~5.1]{BK}.
This forces that $R_2(\alpha|\beta)$ is a projective cover of $D_2(\alpha|\beta)$. So, altogether, we obtain that $R_2(\alpha|\beta)$, as $(\alpha|\beta)$ varies over $\RP^2(m)$, varies over a set of representatives of the isomorphism classes of indecomposable projective $F[\sym{2}\wr\sym{m}]$-modules.
\end{proof}

\begin{cor}\label{cor L}
Given the composition $\rho=(1^{n_0},p^{n_1},\ldots,(p^r)^{n_r})$ of $n$, let $L\subseteq \P^2(n)$ be the set of pairs $(\nu|p\delta)$ such that
$\nu$ and $\delta$ have $p$-adic expansions $\nu=\sum^{r_\nu}_{i=0} p^i\cdot \nu(i)$ and $\delta=\sum^{r_\delta}_{i=0} p^i\cdot \delta(i)$ satisfying $r= \max\{r_\nu,r_\delta+1\}$ and $n_i=|\nu(i)|+|\delta(i-1)|$, for every $i\in\{0,\ldots,r\}$ (where $\delta(-1)=\varnothing$).
Then the $FN(\rho)$-modules $\mathbf{R}(\nu|p\delta)$, as $(\nu|p\delta)$ varies over the set $L$,
form a set of representatives of the isomorphism classes of Young--Green correspondents of indecomposable
signed Young $F\sym{n}$-modules with Young vertex $\sym{\rho}$.
\end{cor}

\begin{proof}
Recall that $N(\rho)/A(\rho)\cong  \sym{n_0}\times (\sym{2}\wr\sym{n_1})\times\cdots\times (\sym{2}\wr\sym{n_r})$.
Via this isomorphism, in consequence of Proposition~\ref{prop proj modules}, the modules
$$Y^{\nu(0)}\boxtimes R_2(\nu(1)|\delta(0))\boxtimes\cdots\boxtimes R_2(\nu(r)|\delta(r-1)),$$ as $(\nu|p\delta)$ varies over the set $L$, form a set of representatives of the isomorphism classes of indecomposable projective $F[N(\rho)/A(\rho)]$-modules.
Recall that if $n_i=0$ for some $i\in\{0,\ldots,r\}$ then $\sym{2}\wr\sym{n_i}$ is just the trivial group and $R_2(\nu(i)|\delta(i-1))=R_2(\varnothing|\varnothing)$ is the trivial $F[\sym{2}\wr\sym{n_i}]$-module.
So, by \cite[1.1(1)]{SD}, the isomorphism classes of indecomposable signed Young $F\sym{n}$-modules with Young vertex $\sym{\rho}$
are, via Young--Green correspondence, in bijection with the inflations
\begin{align*}
&\Inf_{N(\rho)/A(\rho)}^{N(\rho)}(Y^{\nu(0)}\boxtimes R_2(\nu(1)|\delta(0))\boxtimes\cdots\boxtimes R_2(\nu(r)|\delta(r-1)))\\
\cong& Y^{\nu(0)}\boxtimes R_p(\nu(1)|\delta(0))\boxtimes\cdots\boxtimes R_{p^r}(\nu(r)|\delta(r-1))=\mathbf{R}(\nu|p\delta)\,,
\end{align*}
as $(\nu|p\delta)$ varies over $L$.
\end{proof}

\begin{nota}\label{nota Jk} Let $m_1,m_2\in\NN$, and let $m:=m_1+m_2$. Let further $k\geq 2$. Given partitions $\alpha=(\alpha_1,\ldots,\alpha_s)\in\P(m_1)$ and $\beta=(\beta_1,\ldots,\beta_t)\in\P(m_2)$, we denote by $I_k(\alpha|\beta)$ the following $F[\sym{k}\wr(\sym{\alpha}\times\sym{\beta})]$-module:
\[I_k(\alpha|\beta):=F(\sym{k}\wr\sym{\alpha})\boxtimes \sgn(k;\beta),\] where $\sgn(k;\beta)=\res^{\sym{k}\wr\sym{m_2}}_{\sym{k}\wr\sym{\beta}}(\sgn(k;m_2))$.
Note that the trivial $F[\sym{k}\wr\sym{\alpha}]$-module $F(\sym{k}\wr\sym{\alpha})$ is also isomorphic to the inflation of the trivial $F\sym{\alpha}$-module to $\sym{k}\wr\sym{\alpha}$. Moreover, if $k$ is odd then
$ \sgn(k;\beta)\cong\sgn(\sym{k}\wr\sym{\beta})$ (see Remark~\ref{rem k odd}). Now set
$$J_k(\alpha|\beta):=\begin{cases}
\ind_{\sym{k}\wr(\sym{\alpha}\times\sym{\beta})}^{\sym{k}\wr\sym{m}}(I_k(\alpha|\beta))&\text{ if } m>0\,,\\
F(k) &\text{ if } m=0\,.
\end{cases}$$

\smallskip

With the above notation we obtain the $FN(\rho)$-module
$$\mathbf{J}(\lambda|p\mu):=M^{\lambda(0)}\boxtimes J_p(\lambda(1)|\mu(0))\boxtimes\cdots\boxtimes J_{p^r}(\lambda(r)|\mu(r-1))\,.$$
\end{nota}

\begin{lem}\label{lemma R J} Suppose that both $\alpha$ and $\beta$ are $p$-restricted. Then
the  $F[\sym{k}\wr\sym{m}]$-module $R_k(\alpha|\beta)$ is isomorphic to a direct summand of  $J_k(\alpha|\beta)$.
\end{lem}


\begin{proof}
If $m=0$ then the assertion is clear. Suppose that $m\geq 1$.
Since
\begin{align*}
R_k(\alpha|\beta)&=\ind_{(\sym{k}\wr\sym{m_1})\times(\sym{k}\wr\sym{m_2})}^{\sym{k}\wr\sym{m}}(Y^\alpha_k\boxtimes (Y^\beta_k\otimes \sgn(k;m_2)),\\
J_k(\alpha|\beta)&=\ind_{\sym{k}\wr(\sym{\alpha}\times\sym{\beta})}^{\sym{k}\wr\sym{m}}(I_k(\alpha|\beta))\cong \ind_{\sym{k}\wr(\sym{m_1}\times\sym{m_2})}^{\sym{k}\wr\sym{m}}\left (\ind_{\sym{k}\wr(\sym{\alpha}\times\sym{\beta})}^{\sym{k}\wr(\sym{m_1}\times\sym{m_2})}(I_k(\alpha|\beta))\right )\,,
\end{align*}
 it suffices to show that $Y^\alpha_k\boxtimes (Y^\beta_k\otimes \sgn(k;m_2))$ is isomorphic to a direct summand of $\ind_{\sym{k}\wr(\sym{\alpha}\times\sym{\beta})}^{\sym{k}\wr(\sym{m_1}\times\sym{m_2})}(I_k(\alpha|\beta))$.

Firstly, by Lemma~\ref{lemma inflation}, we have
\begin{align*}
&\ind_{\sym{k}\wr(\sym{\alpha}\times\sym{\beta})}^{\sym{k}\wr(\sym{m_1}\times\sym{m_2})}(I_k(\alpha|\beta))\\
\cong& \ind_{\sym{k}\wr\sym{\alpha}}^{\sym{k}\wr\sym{m_1}}(F)\boxtimes \ind_{\sym{k}\wr\sym{\beta}}^{\sym{k}\wr\sym{m_2}}\left (\sgn(k;\beta)\right )\\
\cong& \ind_{\sym{k}\wr\sym{\alpha}}^{\sym{k}\wr\sym{m_1}}(\Inf_{\sym{\alpha}}^{\sym{k}\wr\sym{\alpha}}(F))\boxtimes \ind_{\sym{k}\wr\sym{\beta}}^{\sym{k}\wr\sym{m_2}}\left (\sgn(k;\beta)\right )\\
\cong& \Inf_{\sym{m_1}}^{\sym{k}\wr\sym{m_1}}(\ind_{\sym{\alpha}}^{\sym{m_1}}(F)) \boxtimes \ind_{\sym{k}\wr\sym{\beta}}^{\sym{k}\wr\sym{m_2}}\left (\sgn(k;\beta)\right )\\
\cong& \Inf_{\sym{m_1}}^{\sym{k}\wr\sym{m_1}}(M(\alpha|\varnothing))\boxtimes \ind_{\sym{k}\wr\sym{\beta}}^{\sym{k}\wr\sym{m_2}}\left (\sgn(k;\beta)\right )\,.
\end{align*}
Secondly, as $F[\sym{k}\wr\sym{\beta}]$-modules, we have
\begin{align*}
\sgn(k;\beta)&\cong \Inf_{\sym{\beta}}^{\sym{k}\wr\sym{\beta}}(\sgn(\beta))\otimes (\sgn(k)^{\otimes \beta_1}\boxtimes \cdots \boxtimes \sgn(k)^{\otimes \beta_t})\\
&\cong \Inf_{\sym{\beta}}^{\sym{k}\wr\sym{\beta}}(\sgn(\beta))\otimes \res^{\sym{k}\wr\sym{m_2}}_{\sym{k}\wr\sym{\beta}} (\sgn(k)^{\otimes m_2}).
\end{align*}
Hence, applying the Frobenius Formula and Lemma~\ref{lemma inflation}, we get
\begin{align*}
\ind_{\sym{k}\wr\sym{\beta}}^{\sym{k}\wr\sym{m_2}}\left (\sgn(k;\beta)\right ) \cong& \ind_{\sym{k}\wr\sym{\beta}}^{\sym{k}\wr\sym{m_2}}(\Inf_{\sym{\beta}}^{\sym{k}\wr\sym{\beta}}(\sgn(\beta)))\otimes \sgn(k)^{\otimes m_2}\\
\cong& \Inf_{\sym{m_2}}^{\sym{k}\wr\sym{m_2}}(\ind_{\sym{\beta}}^{\sym{m_2}}(\sgn(\beta)))\otimes \sgn(k)^{\otimes m_2}\\
\cong& \Inf_{\sym{m_2}}^{\sym{k}\wr\sym{m_2}}(M(\varnothing|\beta))\otimes \sgn(k)^{\otimes m_2}\\
\cong&  \Inf_{\sym{m_2}}^{\sym{k}\wr\sym{m_2}}(M(\beta|\varnothing)\otimes\sgn(m_2))\otimes \sgn(k)^{\otimes m_2}\\
\cong&  \Inf_{\sym{m_2}}^{\sym{k}\wr\sym{m_2}}(M(\beta|\varnothing))\otimes\Inf_{\sym{m_2}}^{\sym{k}\wr\sym{m_2}}(\sgn(m_2))\otimes \sgn(k)^{\otimes m_2}\\
\cong& \Inf_{\sym{m_2}}^{\sym{k}\wr\sym{m_2}}(M(\beta|\varnothing))\otimes\sgn(k;m_2)\,.
\end{align*}
Therefore,
$$\ind_{\sym{k}\wr(\sym{\alpha}\times\sym{\beta})}^{\sym{k}\wr(\sym{m_1}\times\sym{m_2})}(I_k(\alpha|\beta))\cong\Inf_{\sym{m_1}}^{\sym{k}\wr\sym{m_1}}(M(\alpha|\varnothing))\boxtimes \left (\Inf_{\sym{m_2}}^{\sym{k}\wr\sym{m_2}}(M(\beta|\varnothing))\otimes\sgn(k;m_2)\right )\,.$$
Since $Y^\alpha\mid M^\alpha=M(\alpha|\varnothing)$ and $Y^\beta\mid M^\beta=M(\beta|\varnothing)$, we conclude that \[Y^\alpha_k\boxtimes (Y^\beta_k\otimes \sgn(k;m_2))\mid \ind_{\sym{k}\wr(\sym{\alpha}\times\sym{\beta})}^{\sym{k}\wr(\sym{m_1}\times\sym{m_2})}(I_k(\alpha|\beta)).\]
\end{proof}

\begin{rem}\label{rem Nrho}
In the proof of the next lemma we shall have to determine the structure of the intersections of
$N(\rho)$ with $\sym{n}$-conjugates of the Young subgroup $\sym{\lambda}\times \sym{p\mu}$.  To this end, suppose that $\lambda=(\lambda_1,\ldots,\lambda_u)$
and $\mu=(\mu_1,\ldots,\mu_v)$, for some $u,v\in\NN$. Given the $p$-adic expansions (\ref{eqn exp}) of $\lambda$ and $\mu$,
we thus have $\lambda_j=\sum_{i=0}^rp^i\cdot \lambda(i)_j$ as well as $p\mu_\ell=\sum_{i=1}^rp^i\cdot \mu(i-1)_\ell$ for
$j\in\{1,\ldots,u\}$ and $\ell\in\{1,\ldots, v\}$, where here we set $\lambda(i):=\varnothing$, for $i>r_\lambda$, and $\mu(k):=\varnothing$, for
$k>r_\mu$. Therefore, we have
$$H:=\prod_{j=1}^u\left( \sym{\lambda(0)_j}\times\prod_{i=1}^r\sym{p^i}^{\lambda(i)_j}\right)\times \prod_{\ell=1}^v\prod_{i=1}^r\sym{p^i}^{\mu(i-1)_\ell}\leq \prod_{j=1}^u\sym{\lambda_j}\times\prod_{\ell=1}^v\sym{p\mu_\ell}=\sym{\lambda}\times\sym{p\mu}\leq \sym{n}\,.$$
%
Now let $g\in \sym{n}$ be such that
\begin{align*}
{}^gH&= \sym{\lambda(0)}\times\prod_{i=1}^r\left(\prod_{j=1}^u\sym{p^i}^{\lambda(i)_j}\times\prod_{\ell=1}^v\sym{p^i}^{\mu(i-1)_\ell}\right)\\
&\leq {}^g\sym{\lambda_1}\times\cdots \times {}^g\sym{\lambda_u}\times {}^g\sym{p\mu_1}\times\cdots\times {}^g\sym{p\mu_v}={}^g(\sym{\lambda}\times\sym{p\mu})\,.
\end{align*}
Then also ${}^gH\leq N(\rho)$, since ${}^gH\leq \sym{\lambda(0)}\times \sym{p}^{n_1}\times\cdots \times \sym{p^r}^{n_r}$.
On the other hand, every element of $N(\rho)\cap {}^g(\sym{\lambda}\times\sym{p\mu})$ fixes
the orbits of ${}^g(\sym{\lambda}\times\sym{p\mu})$ and permutes the orbits of ${}^gH$. This implies that
\begin{equation*}\label{eqn Nrho}
N(\rho)\cap{}^g(\sym{\lambda}\times\sym{p\mu})=\sym{\lambda(0)}\times \prod_{i=1}^r\sym{p^i}\wr(\sym{\lambda(i)}\times\sym{\mu(i-1)})\,.
\end{equation*}
We illustrate this by an example as below.
\end{rem}

\begin{eg}\label{expl Nrho}
Consider the case when $p=3$, $\lambda=(14,3)$ and $\mu=(10,1)$. Then $\lambda=(2)+3\cdot (1,1)+3^2\cdot (1)$
and $\mu=(1,1)+3^2 \cdot (1)$, so that $r=3$ and
$\lambda(0)=(2)$, $\lambda(1)=(1,1)$, $\lambda(2)=(1)$, $\mu(0)=(1,1)$, $\mu(2)=(1)$. Also,  $n_0=2$, $n_1=2+2=4$, $n_2=1=n_3$.
Hence $\rho=(1^2,3^4,9,27)$ and $N(\rho)=\sym{2}\times(\sym{3}\wr\sym{4})\times \sym{9}\times\sym{27}$.
On the other hand, the group called $H$ in Remark~\ref{rem Nrho} above has the form
$$H=\sym{2}\times \sym{3}\times\sym{9}\times \sym{3}\times \sym{3}\times\sym{27}\times\sym{3}\leq \sym{14}\times\sym{3}\times\sym{30}\times\sym{3}=\sym{\lambda}\times\sym{p\mu}\,.$$
So we take $g\in\sym{50}$ to be the following permutation
$$g:=\begin{pmatrix} 1&2&3&4&5&6&\cdots&14&15&\cdots&20&21&\cdots&47&48&49&50\\1&2&3&4&5&15&\cdots&23&6&\cdots&11&24&\cdots&50&12&13&14\end{pmatrix}$$
to get
\begin{align*}
{}^gH&=\sym{2}\times\sym{3}\times\sym{3}\times\sym{3}\times\sym{3}\times\sym{9}\times\sym{27}\\
&\leq
\mathfrak{S}_{\{1,\ldots,5,15,\ldots,23\}}\times\mathfrak{S}_{\{6,7,8\}}\times \mathfrak{S}_{\{9,10,11,24,\ldots,50\}}\times\mathfrak{S}_{\{12,13,14\}} \\
&={}^g(\sym{\lambda}\times\sym{p\mu})
\end{align*}
and $N(\rho)\cap{}^g(\sym{\lambda}\times\sym{p\mu})=\sym{2}\times (\sym{3})^4\times\sym{9}\times\sym{27}$.
\end{eg}

\begin{lem}\label{lemma R J M}
For $(\lambda|p\mu)\in\P^2(n)$, one has
$$\mathbf{R}(\lambda|p\mu)\big\arrowvert \mathbf{J}(\lambda|p\mu)\big\arrowvert \res_{N(\rho)}^{\sym{n}}(M(\lambda|p\mu))$$
as $FN(\rho)$-modules, where $\rho=\Rho(\lambda|p\mu)$.
\end{lem}

\begin{proof}
As usual, let $r=\max\{r_\lambda,r_\mu+1\}$. Since $\mathbf{R}(\lambda|p\mu)=Y^{\lambda(0)}\boxtimes R_p(\lambda(1)|\mu(0))\boxtimes\cdots\boxtimes R_{p^r}(\lambda(r)|\mu(r-1))$
and $\mathbf{J}(\lambda|p\mu)=Y^{\lambda(0)}\boxtimes J_p(\lambda(1)|\mu(0))\boxtimes\cdots\boxtimes J_{p^r}(\lambda(r)|\mu(r-1))$,
we obtain $\mathbf{R}(\lambda|p\mu)\mid \mathbf{J}(\lambda|p\mu)$, by Lemma~\ref{lemma R J}.

As explained in Remark~\ref{rem Nrho}, there is some $g\in\sym{n}$
such that $N(\rho)\cap{}^g(\sym{\lambda}\times\sym{p\mu})=\sym{\lambda(0)}\times \prod_{i=1}^r\sym{p^i}\wr(\sym{\lambda(i)}\times\sym{\mu(i-1)})$.
By the Mackey Formula,
$$\ind_{N(\rho)\cap{}^g(\sym{\lambda}\times\sym{p\mu})}^{N(\rho)}\left (\res_{N(\rho)\cap {}^g(\sym{\lambda}\times\sym{p\mu})}^{{}^g(\sym{\lambda}\times\sym{p\mu})}(F({}^g\sym{\lambda})\boxtimes \sgn({}^g\sym{p\mu}))\right )\big\arrowvert \res_{N(\rho)}^{\sym{n}}(M(\lambda|p\mu))\,.$$
Furthermore, we have
\begin{align*}
&\res_{N(\rho)\cap {}^g(\sym{\lambda}\times\sym{p\mu})}^{{}^g(\sym{\lambda}\times\sym{p\mu})}(F({}^g\sym{\lambda})\boxtimes \sgn({}^g\sym{p\mu}))\\
\cong &F(\lambda(0))\boxtimes\left (\boxtimes_{i=1}^r(F(\sym{p^i}\wr\sym{\lambda(i)})\boxtimes\sgn(\sym{p^i}\wr\sym{\mu(i-1)}))\right )\,,
\end{align*}
when identifying each wreath product $\sym{p^i}\wr(\sym{\lambda(i)}\times\sym{\mu(i-1)})$ with
$(\sym{p^i}\wr\sym{\lambda(i)})\times (\sym{p^i}\wr\sym{\mu(i-1)})$ as usual.
For each $i\in\{1,\ldots,r\}$, let $\mu(i-1)=(\mu(i-1)_1,\ldots,\mu(i-1)_t)$; here $\mu(i-1)_t$ may be $0$.
Since $p$ is odd, we have   $\sgn(\sym{p^i}\wr \sym{\mu(i-1)})\cong \sgn(p^i;\mu(i-1)_1)\boxtimes\cdots\boxtimes \sgn(p^i;\mu(i-1)_t)$, by  Remark~\ref{rem k odd}.
This shows that
\begin{align*}
&\ind_{N(\rho)\cap{}^g(\sym{\lambda}\times\sym{p\mu})}^{N(\rho)}\left (\res_{N(\rho)\cap {}^g(\sym{\lambda}\times\sym{p\mu})}^{{}^g(\sym{\lambda}\times\sym{p\mu})}(F({}^g\sym{\lambda})\boxtimes \sgn({}^g\sym{p\mu}))\right )\\
\cong &\ind_{\sym{\lambda(0)}}^{\sym{n_0}}(F(\lambda(0)))\boxtimes\left (\boxtimes_{i=1}^r\ind_{\sym{p^i}\wr(\sym{\lambda(i)}\times\sym{\mu(i-1)})}^{\sym{p^i}\wr\sym{n_i}}(I_{p^i}(\lambda(i)\mid\mu(i-1)))\right )\\
 \cong &M^{\lambda(0)}\boxtimes J_p(\lambda(1)|\mu(0))\boxtimes \cdots\boxtimes J_{p^r}(\lambda(r)|\mu(r-1))=\mathbf{J}(\lambda|p\mu)\,,
\end{align*}
which completes the proof of the lemma.
 \end{proof}

To conclude this appendix, we shall now prove Theorem~\ref{thm Young Green}.

\begin{proof}[Proof of Theorem~\ref{thm Young Green}]
Let $L\subseteq \P^2(n)$ be the set of pairs of partitions defined as in Corollary~\ref{cor L}. That is, fixing $\rho=(1^{n_0},p^{n_1},\ldots,(p^r)^{n_r})$, we have $(\nu|p\delta)\in L$ if and only
if $r=\max\{r_\nu,r_\delta+1\}$ and $|\nu(i)|+|\delta(i-1)|=n_i$, for every $i\in\{0,1,\ldots,r\}$. By Corollary~\ref{cor L}, there is a
bijection $\psi:L\to L$ induced by the Young--Green correspondence with respect to $N(\rho)$
of indecomposable signed Young $F\sym{n}$-modules with Young vertex $\sym{\rho}$.
We have to show that  $\psi(\alpha|p\beta)=(\alpha|p\beta)$, for all $(\alpha|p\beta)\in L$.

Let $(\alpha|p\beta)\in L$, and assume that $(\nu|p\delta)=\psi(\alpha|p\beta)$.
By \cite[Satz~6.3]{JGrab}, the Krull--Schmidt multiplicity of the indecomposable signed Young module $Y(\alpha|p\beta)$ as a direct summand
of $M(\nu|p\delta)$ equals the Krull--Schmidt multiplicity of the Young--Green correspondent $\mathbf{R}(\nu|p\delta)$ of $Y(\alpha|p\beta)$ as a
direct summand of $\res_{N(\rho)}^{\sym{n}}(M(\nu|p\delta))$. By Lemma~\ref{lemma R J M}, we know that
$\mathbf{R}(\nu|p\delta)\mid \res_{N(\rho)}^{\sym{n}}(M(\nu|p\delta))$, so that also $Y(\alpha|p\beta)\mid M(\nu|p\delta)$. By
\cite[2.3(8)]{SD}, this implies $\psi(\alpha|p\beta)=(\nu|p\delta)\unlhd (\alpha|p\beta)$. Now we argue by induction on the length of $(\alpha|p\beta)$ in the poset $(L,\unlhd)$. Suppose that $(\alpha|p\beta)$ is a minimal element in $(L,\unlhd)$, that is, has length 0 in $(L,\unlhd)$. Then
$\psi(\alpha|p\beta)\unlhd (\alpha|p\beta)$ forces $\psi(\alpha|p\beta)= (\alpha|p\beta)$. Now suppose that
$(\alpha|p\beta)$ has some length $\ell\geq 1$ in $(L,\unlhd)$ and that $(\zeta|p\xi)=\psi(\zeta|p\xi)$, for all
$(\zeta|p\xi)\in L$ of length less that $\ell$. Then $\psi(\alpha|p\beta)\unlhd (\alpha|p\beta)$, and if
$(\zeta|p\xi)\lhd (\alpha|p\beta)$ then $(\zeta|p\xi)=\psi(\zeta|p\xi)$ by induction. Since $\psi$ is bijective, this implies $\psi(\alpha|p\beta)=(\alpha|p\beta)$, so that $Y(\alpha|p\beta)$ has Young--Green correspondent $\mathbf{R}(\alpha|p\beta)$ as desired.
\end{proof}


\section{Signed Young Rule}\label{S: signed Young Rule}

Throughout this appendix, $F$ is a field of arbitrary characteristic.

Given a Young permutation module, Young's Rule (see  \cite[14.1, 17.14]{GJ1}) describes the multiplicities of Specht modules occurring as factors of certain Specht series of Young permutation modules. More precisely, given a partition $\alpha\in \P(n)$,  the Young permutation $F\sym{n}$-module $M^\alpha=M(\alpha|\varnothing)$ admits a Specht filtration in which a Specht $F\sym{n}$-module $S^\lambda$ occurs as a factor with multiplicity equal to the number of {\sl semistandard $\lambda$-tableaux of type $\alpha$}.

In this appendix, following \cite{KJLKMT}, we describe the multiplicities of the Specht $F\sym{n}$-modules occurring as factors of certain Specht series of a signed Young permutation  $F\sym{n}$-module. Theorem~\ref{T:Specht factors of sgn permutation} can thus be seen  as  a generalization of Young's Rule.

To generalize the notion of a  semistandard tableau of a particular shape and type,
consider the set of `primed' positive integers $\{1',2',3',\ldots\}$. We define the obvious total ordering on the set $\{1',2',3',\ldots\}$,
by declaring \[1'<2'<3'<\cdots.\]

\begin{defn}[{\cite[\S2.4]{NM}}] \label{defn semistd}
Let $n\in \mathbb{Z}^+$, and let $(\alpha|\beta)\in \P^2(n)$. Moreover, let $\lambda\in\P(n)$.
A {\sl semistandard $\lambda$-tableaux $\mathfrak{t}$ of type $(\alpha|\beta)$} is obtained by filling the Young diagram
$[\lambda]$ with  $\alpha_i$ entries equal to $i$ and $\beta_j$ entries equal to $j'$, for $i,j\in\mathbb{Z}^+$, such that the following conditions are satisfied.


\smallskip

(a)\, The sub-tableau $\mathfrak{s}$ of $\mathfrak{t}$ occupied by the non-primed integers
is a  semistandard $\mu$-tableaux of type $\alpha$ for some sub-partition $\mu$ of $\lambda$, that is, the
non-primed integers are strictly increasing along every column from top to bottom and are weakly increasing along every
row from left to right.

\smallskip


(b)\, The skew  tableau $\mathfrak{t}\smallsetminus \mathfrak{s}$ occupied by the primed integers forms a conjugate semistandard skew
$(\lambda/ \mu)$-tableau, that is,  the primed integers are weakly increasing along every column from top to bottom and strictly increasing along every row from left to right.

\smallskip

\noindent
In the case when $\beta=\varnothing$, one recovers the usual semistandard $\lambda$-tableau $\mathfrak{t}$ of type $\alpha$.
\end{defn}

\begin{eg}\label{eg semistd} Let $\lambda=(4,3,2,2)$, and let $(\alpha|\beta)=((3,3,1)|(2,2))$. There are a total of five semistandard $\lambda$-tableaux of type $(\alpha|\beta)$ as below.

\[\begin{array}{ccccc}
  {\begin{pspicture}(0,0)(2.5,2.5) \psset{xunit=.3cm,yunit=.3cm}
  \rput(1,1){$1'$} \rput(3,1){$2'$} \rput(1,3){3} \rput(3,3){$1'$} \rput(1,5){2} \rput(3,5){2} \rput(5,5){$2'$} \rput(1,7){1} \rput(3,7){1} \rput(5,7){1} \rput(7,7){2} \psline(0,0)(4,0) \psline(0,2)(4,2) \psline(0,4)(6,4) \psline(0,6)(8,6) \psline(0,8)(8,8) \psline(0,0)(0,8) \psline(2,0)(2,8) \psline(4,0)(4,8) \psline(6,4)(6,8) \psline(8,6)(8,8)
  \end{pspicture}}&
  {\begin{pspicture}(0,0)(2.5,2.5) \psset{xunit=.3cm,yunit=.3cm}
  \rput(1,1){$1'$} \rput(3,1){$2'$} \rput(1,3){$1'$} \rput(3,3){$2'$} \rput(1,5){2} \rput(3,5){2} \rput(5,5){$3$} \rput(1,7){1} \rput(3,7){1} \rput(5,7){1} \rput(7,7){2} \psline(0,0)(4,0) \psline(0,2)(4,2) \psline(0,4)(6,4) \psline(0,6)(8,6) \psline(0,8)(8,8) \psline(0,0)(0,8) \psline(2,0)(2,8) \psline(4,0)(4,8) \psline(6,4)(6,8) \psline(8,6)(8,8)
  \end{pspicture}}&
  {\begin{pspicture}(0,0)(2.5,2.5) \psset{xunit=.3cm,yunit=.3cm}
  \rput(1,1){$1'$} \rput(3,1){$2'$} \rput(1,3){$1'$} \rput(3,3){$2'$} \rput(1,5){2} \rput(3,5){2} \rput(5,5){$2$} \rput(1,7){1} \rput(3,7){1} \rput(5,7){1} \rput(7,7){3} \psline(0,0)(4,0) \psline(0,2)(4,2) \psline(0,4)(6,4) \psline(0,6)(8,6) \psline(0,8)(8,8) \psline(0,0)(0,8) \psline(2,0)(2,8) \psline(4,0)(4,8) \psline(6,4)(6,8) \psline(8,6)(8,8)
  \end{pspicture}}&
  {\begin{pspicture}(0,0)(2.5,2.5) \psset{xunit=.3cm,yunit=.3cm}
  \rput(1,1){$1'$} \rput(3,1){$2'$} \rput(1,3){$3$} \rput(3,3){$2'$} \rput(1,5){2} \rput(3,5){2} \rput(5,5){$2$} \rput(1,7){1} \rput(3,7){1} \rput(5,7){1} \rput(7,7){$1'$} \psline(0,0)(4,0) \psline(0,2)(4,2) \psline(0,4)(6,4) \psline(0,6)(8,6) \psline(0,8)(8,8) \psline(0,0)(0,8) \psline(2,0)(2,8) \psline(4,0)(4,8) \psline(6,4)(6,8) \psline(8,6)(8,8)
  \end{pspicture}}&
  {\begin{pspicture}(0,0)(2.5,2.5) \psset{xunit=.3cm,yunit=.3cm}
  \rput(1,1){$1'$} \rput(3,1){$2'$} \rput(1,3){$3$} \rput(3,3){$1'$} \rput(1,5){2} \rput(3,5){2} \rput(5,5){$2$} \rput(1,7){1} \rput(3,7){1} \rput(5,7){1} \rput(7,7){$2'$} \psline(0,0)(4,0) \psline(0,2)(4,2) \psline(0,4)(6,4) \psline(0,6)(8,6) \psline(0,8)(8,8) \psline(0,0)(0,8) \psline(2,0)(2,8) \psline(4,0)(4,8) \psline(6,4)(6,8) \psline(8,6)(8,8)
  \end{pspicture}}
\end{array}\]
\end{eg}

\begin{nota}\label{nota LR}
Suppose that $\lambda\in\P(n)$, and let $\gamma$ and $\xi$ be partitions such that $|\gamma|+|\xi|=n$.
Denote by $c_{\gamma,\xi}^\lambda$ the {\sl Littlewood--Richardson coefficient} in the sense of \cite[\S16]{GJ1}.
By
\cite[2.8.13]{GJAK}, the $F\sym{n}$-module $\ind_{\sym{|\gamma|}\times\sym{|\xi|}}^{\sym{n}}(S^\gamma\boxtimes S^\xi)$
admits a Specht filtration such that $S^\lambda$ occurs with multiplicity $c_{\gamma,\xi}^\lambda$ in this filtration.

Moreover, if $\alpha\in\P(n)$ then, by Young's Rule, the Young permutation $F\sym{n}$-module $M^\alpha=M(\alpha|\varnothing)$
admits a  Specht filtration
in which $S^\lambda$ occurs with multiplicity equal to the number of semistandard $\lambda$-tableaux of type $\alpha$, and we denote this multiplicity by $y_{\alpha,\lambda}$.

By convention, we set $c^\varnothing_{\varnothing,\varnothing}:=1$ and $y_{\varnothing,\varnothing}:=1$.
\end{nota}

\begin{lem} \label{lemma y}
Let $m\in \mathbb{Z}^+$, and let $\beta\in\P(m)$. Then
the signed Young permutation $F\sym{m}$-module $M(\varnothing|\beta)$ has a Specht series
in which, for each $\xi\in\P(m)$, the Specht $F\sym{m}$-module
$S^{\xi}$ occurs as a factor with multiplicity $y_{\beta,\xi'}$.
\end{lem}

\begin{proof}
By Young's Rule \cite[17.14]{GJ1}, the Young permutation $F\sym{m}$-module $M^\beta=M(\beta|\varnothing)$ has
a Specht filtration in which the Specht $F\sym{m}$-module $S^{\xi'}$ occurs with multiplicity  $y_{\beta,\xi'}$.
Now recall from \ref{noth Specht} that we have $S^{\xi}\cong (S^{\xi'})^*\otimes\sgn$ and $(M(\beta|\varnothing))^*\otimes \sgn\cong M(\beta|\varnothing)\otimes \sgn\cong M(\varnothing|\beta)$. Consequently, $M(\varnothing|\beta)$ admits a Specht filtration in which
$S^\xi$ occurs as a factor with multiplicity  $y_{\beta,\xi'}$.
\end{proof}

The proof of the following lemma requires the notion of the {\sl rectification} of a skew tableau. Since we do not need this notion elsewhere in our paper, we refer the reader to the reference \cite[\S5.1]{WF} for the necessary definitions and notation that we follow.

\begin{lem}\label{L: number of sstd}
Let $\lambda\in\P(n)$, and let $(\alpha|\beta)\in\P^2(n)$.
The number of semistandard $\lambda$-tableaux of type $(\alpha|\beta)$ is
$$\sum_{\gamma\vdash |\alpha|,\xi\vdash |\beta|} y_{\alpha,\gamma}y_{\beta,\xi'}c_{\gamma,\xi}^\lambda\,.$$
\end{lem}

\begin{proof}
There is a one-to-one correspondence between the set of semistandard $\lambda$-tableaux of type $(\alpha|\beta)$ and the set of
pairs consisting of a semistandard $\gamma$-tableau of type $\alpha$ and a conjugate semistandard skew $(\lambda/ \gamma)$-tableau of type $\beta$, as  $\gamma$ varies over the set of sub-partitions of $\lambda$ of size $|\alpha|$.
Thus, it suffices to show that, for each such $\gamma$, the number of conjugate semistandard skew $(\lambda/ \gamma)$-tableau of type $\beta$ is \[\sum_{\xi\vdash |\beta|} y_{\beta,\xi'}c_{\gamma,\xi}^\lambda\,.\]
Fix a partition $\xi$ of $|\beta|$ and  a semistandard $\xi'$-tableau $\mathfrak{t}'$ of type $\beta$. Let $S((\lambda/\gamma)',\mathfrak{t}')$ be the set consisting of semistandard skew $(\lambda/\gamma)'$-tableaux whose rectification is $\mathfrak{t}'$ (see \cite[\S5.1]{WF}). By \cite[\S5.1 Corollary 2(iii)]{WF}, the set $S((\lambda/\gamma)',\mathfrak{t}')$ has the same cardinality as $S(\lambda/\gamma,\mathfrak{t})$, which is also the same as the Littlewood--Richardson coefficient $c^\lambda_{\gamma,\xi}$. This shows that, for each conjugate semistandard $\xi$-tableau $\mathfrak{t}$ of type $\beta$, and hence every semistandard $\xi'$-tableau $\mathfrak{t}'$ of type $\beta$, we get all semistandard skew $(\lambda/\gamma)$-tableaux whose rectification is $\mathfrak{t}$. Hence we get a total of $y_{\beta,\xi'}c^\lambda_{\gamma,\xi}$ conjugate semistandard skew $(\lambda/ \gamma)$-tableau of type $\beta$.
\end{proof}

\begin{thm}[Signed Young Rule {\cite[Theorem 3.5]{KJLKMT}}]\label{T:Specht factors of sgn permutation}
Let $(\alpha|\beta)\in\P^2(n)$.
The signed Young permutation $F\sym{n}$-module $M(\alpha|\beta)$ has a Specht series in which, for each $\lambda\in\P(n)$,
the Specht module $S^\lambda$ occurs as a factor with multiplicity equal to the number of semistandard $\lambda$-tableaux of type $(\alpha|\beta)$.
\end{thm}

\begin{proof}
Suppose that $n_1=|\alpha|$ and $n_2=|\beta|$. By  Young's Rule and Lemma~\ref{lemma y}, we obtain Specht series of $M(\alpha|\varnothing)$ and of $M(\varnothing|\beta)$ in which
the Specht $F\sym{n_1}$-module $S^\gamma$ and the Specht $F\sym{n_2}$-module $S^{\xi}$ occur as factors with multiplicities
$y_{\alpha,\gamma}$ and $y_{\beta,\xi'}$, respectively. Moreover, we have
$M(\alpha|\beta)\cong\ind_{\sym{n_1}\times\sym{n_2}}^{\sym{n}}(M(\alpha|\varnothing)\boxtimes M(\varnothing|\beta))$.
Thus, by \cite[2.8.13]{GJAK}, $M(\alpha|\beta)$ has a Specht filtration in which
$S^\lambda$ occurs as a factor with multiplicity
$$\sum_{\gamma\vdash n_1,\ \xi\vdash n_2} y_{\alpha,\gamma}y_{\beta,\xi'}c_{\gamma,\xi}^\lambda\,.$$
The result now follows from Lemma \ref{L: number of sstd}.
\end{proof}

As an immediate corollary we have

\begin{cor}\label{cor comp series}
Suppose that $F$ has characteristic $0$, let $(\alpha|\beta)\in\P^2(n)$, and let $\lambda\in\P(n)$. Then
the composition factors of the signed Young permutation $F\sym{n}$-module $M(\alpha|\beta)$ isomorphic to
$S^\lambda$ equals the number of semistandard $\lambda$-tableaux of type $(\alpha|\beta)$.
\end{cor}

\end{appendix}

\bigskip

\noindent
{\sc S.D.: Catholic University of Eichst\"att-Ingolstadt, Department of Mathematics and Geography, Ostenstrasse 28, 85072 Eichst\"att, Germany\\ {\sf susanne.danz@ku.de}}

\bigskip

\noindent
{\sc K.J.L.: Division of Mathematical Sciences, Nanyang Technological University, SPMS-MAS-03-01, 21 Nanyang Link, Singapore 637371}\\ {\sf limkj@ntu.edu.sg}

\end{document}